\documentclass[a4paper, reqno, twoside,12p]{amsart}
\usepackage{etex}
\usepackage{amsfonts,amssymb,graphics,amsthm,amsmath}
\usepackage{latexsym}
\usepackage[2emode]{psfrag}
\usepackage{mathrsfs}
\usepackage[all, cmtip]{xy}
\usepackage{enumerate}
\usepackage[latin1]{inputenc}
\usepackage{lscape}
\usepackage{a4wide}
\usepackage[bookmarks=false]{hyperref}

\usepackage{txfonts}

\usepackage[all]{pstricks}
\usepackage{pst-text,pst-node,pst-coil}
\usepackage{tikz}

\usepackage{cancel}
\usepackage[normalem]{ulem}
\usetikzlibrary{arrows,shapes,backgrounds}

\usepgflibrary{arrows}
\usetikzlibrary{topaths}
\usetikzlibrary{er}

\input{xy}
 \xyoption{all}

\newtheorem{theorem}{Theorem}[subsection]
\newtheorem{proposition}[theorem]{Proposition}
\newtheorem{corollary}[theorem]{Corollary}
\newtheorem{lemma}[theorem]{Lemma}
\newtheorem*{theorem*}{Theorem}
\newtheorem*{proposition*}{Proposition}
\newtheorem*{corollary*}{Corollary}
\newtheorem*{lemma*}{Lemma}
\theoremstyle{definition}
\newtheorem{definition}[theorem]{Definition}

\newtheorem*{punto*}{ }

\newtheorem{example}[theorem]{Example}
\newtheorem{remark}[theorem]{Remark}
\newtheorem*{remark*}{Remark}

\newtheorem*{definition*}{Definition}
\newtheorem{problem}{Problem}

\newcommand{\coring}[1]{\mathfrak{#1}}
\newcommand{\tensor}[1]{\otimes_{\Sscript{#1}}}
\newcommand{\tensfun}[1]{\underset{{#1}}{\otimes}}
\newcommand{\rmod}[1]{\mathsf{Mod}_{#1}}
\newcommand{\lmod}[1]{{}_{#1}\mathsf{Mod}}
\newcommand{\rcomod}[1]{\mathsf{Comod}_{#1}}

\newcommand{\bimod}[2]{{}_{#1}\mathsf{Mod}_{#2}}

\newcommand{\frcomod}[1]{\mathsf{comod}_{#1}}

\renewcommand{\hom}[3]{\mathrm{Hom}_{\Sscript{#1}}(#2,#3)}
\newcommand{\rend}[2]{\mathrm{End}(#1_{#2})}

\newcommand{\rcomatrix}[2]{#2^* \tensor{#1} #2}

\newcommand{\bara}[1]{\overline{#1}}

\newcommand{\lr}[1]{\left(\underset{}{}  #1 \right)}

\newcommand{\cat}[1]{\mathcal{#1}}
\newcommand{\fk}[1]{\mathfrak{#1}}
\newcommand{\Sf}[1]{\mathsf{#1}}
\newcommand{\Scr}[1]{\mathscr{#1}}

\newcommand{\counidad}{\varepsilon}

\newcommand{\can}[1]{\mathsf{can}_{#1}}

\newcommand{\dosmatrix}[4]{\begin{pmatrix} #1 & #2 \\ #3 & #4
\end{pmatrix}}

\newcommand{\bo}[1]{\boldsymbol{#1}}

\newcommand{\peque}[1]{\scriptscriptstyle{#1}}
\newcommand{\rhom}[3]{\mathrm{Hom}_{{\text{-}} #1}(#2,#3)}
\newcommand{\lhom}[3]{\mathrm{Hom}_{#1{\text{-}}}(#2,#3)}
\newcommand{\recons}[1]{\mathscr{R}(#1)}
\newcommand{\Ae}{A^{\Sscript{\Sf{e}}}}
\newcommand{\LR}[1]{\left(\underset{}{} #1 \right)}
\newcommand{\bcirc}{\boldsymbol{\circ}}
\newcommand{\Der}[2]{\mathbf{Der}_{#1}(#2)}
\newcommand{\diama}{{\scriptstyle{\lozenge}}}
\newcommand{\mat}[1]{\mathit{mat}(#1)}
\newcommand{\columna}[2]{\left(\begin{array}{c} #1 \\ \vdots \\ #2 \end{array}\right)}

\newcommand{\Diff}[1]{\mathbf{Diff}_{\Sscript{#1}}}

\newcommand{\Sscript}[1]{\scriptscriptstyle{#1}}
\newcommand{\Um}{U^{\bcirc}_{\Sscript{(M)}}}
\newcommand{\Umx}{U^{\bcirc}_{\Sscript{(M), \, x}}}
\newcommand{\Mo}{\langle M \rangle_{\Sscript{\otimes}}}
\newcommand{\omegam}{\omega_{|\Sscript{\Mo}}}
\newcommand{\hHm}{\mathscr{H}_{\Sscript{M}}}

\newcommand{\Algc}{\mathsf{Alg}_{\Sscript{\mathbb{C}}}}
\newcommand{\Grpds}{\mathsf{Grpds}}

\newcommand{\vect}[1]{{\rm vect}_{\Sscript{#1}}}

\newcommand{\II}{\mathbb{I}}
\newcommand{\Ao}{A^{\Sscript{o}}}
\newcommand{\ao}{a^{\Sscript{o}}}
\newcommand{\bop}{b^{\Sscript{o}}}

\newcommand{\RA}[1]{\rR_{\Sscript{A}}(#1)}
\newcommand{\RAA}[1]{\rR_{\Sscript{A\tensor{}A}}(#1)}
\newcommand{\RAAu}[1]{\rR_{\Sscript{A_1\tensor{}A_1}}(#1)}
\newcommand{\RAAd}[1]{\rR_{\Sscript{A_2\tensor{}A_2}}(#1)}
\newcommand{\Rdos}[2]{\rR_{\Sscript{A_1\tensor{}A_2}}(#1, #2)}
\newcommand{\Roma}[1]{\rR_{\Sscript{A_{#1}\tensor{}A_{#1}}}(\omega_{#1})}
\newcommand{\Romaa}[2]{\rR_{\Sscript{A_{#1}\tensor{}A_{#2}}}(\omega_{#1}, \omega_{#2})}
\newcommand{\Raa}[2]{\rR_{\Sscript{A}}(\omega_{#1}, \omega_{#2})}
\newcommand{\Ra}[1]{\rR_{\Sscript{A}}(\omega_{#1})}

\newcommand{\IsomF}[2]{ \underline{{\rm Isom}}^{\Sscript{\otimes}}{}_{\Sscript{A_1\tensor{}A_2}}\big( \omega_{#1}, \omega_{#2}\big)  }
\newcommand{\Isomf}[2]{ \underline{{\rm Isom}}^{\Sscript{\otimes}}{}_{\Sscript{A}}\big( \omega_{#1}, \omega_{#2}\big)  }
\newcommand{\AutF}[1]{\underline{{\rm Aut}}^{\Sscript{\otimes}}(#1)}
\newcommand{\Alg}[1]{{\rm Alg}_{\Sscript{#1}}}
\newcommand{\sR}{{}_{\Sscript{s}}R}
\newcommand{\tR}{{}_{\Sscript{t}}R}

\newcommand{\balpha}{\bo{\alpha}}
\newcommand{\bbeta}{\bo{\beta}}

\newcommand{\bsigma}{\bo{\sigma}}
\newcommand{\bgamma}{\bo{\gamma}}
\newcommand{\bdelta}{\bo{\delta}}
\newcommand{\brho}{\bo{\varrho}}
\newcommand{\bpartial}{\bo{\partial}}
\newcommand{\biota}{\bo{\iota}}
\newcommand{\bV}{\bo{V}}
\newcommand{\bW}{\bo{W}}
\newcommand{\bv}{\bo{v}}

\newcommand{\tensorC}{\tensor{\mathbb{C}}}
\newcommand{\class}[1]{\Big[ \bara{#1} \Big]}

\newcommand{\bCx}{\mathbb{C}_{\Sscript{x}}}
\newcommand{\omegax}{\omega_{\Sscript{x}}}
\newcommand{\Autd}[2]{{\rm Aut}_{\Sscript{(#1,\, \partial)}}\big( (#2,\bpartial)\big)}
\newcommand{\AutdF}[2]{\underline{\rm Aut}_{\Sscript{(#1,\, \partial)}}\big( (#2,\bpartial)\big)}
\newcommand{\Autdt}[1]{{\rm Aut}_{\Sscript{(A\tensor{}#1,\, \partial\tensor{}#1)}}\big( (\cP\tensor{}#1,\bpartial\tensor{}#1)\big)}
\newcommand{\Partial}[1]{\partial^{\Sscript{#1}}}

\newcommand{\kn}{\bo{k}_{(k_{\Sscript{1}},\, \cdots,\, k_{\Sscript{n}})}}
\newcommand{\KK}[1]{\bo{k}(#1)}
\newcommand{\lrfrac}[2]{\left( \frac{#1}{#2} \right)}

\newcommand{\affLine}[1]{\mathbb{A}_{\Sscript{#1}}^{1}}

\newcommand{\dP}{P^{\Sscript{\vee}}}
\newcommand{\dQ}{Q^{\Sscript{\vee}}}

\newcommand{\gG}{\mathscr{G}}
\newcommand{\hH}{\mathscr{H}}

\newcommand{\kK}{\mathscr{K}}

\newcommand{\oO}{\mathscr{O}}

\newcommand{\qQ}{\mathscr{Q}}
\newcommand{\rR}{\mathscr{R}}
\newcommand{\sS}{\mathscr{S}}

\newcommand{\uU}{\mathscr{U}}

\newcommand{\Cc}{\mathbb{C}}
\newcommand{\cA}{{\mathcal A}}

\newcommand{\cD}{{\mathcal D}}

\newcommand{\cH}{{\mathcal H}}
\newcommand{\cI}{{\mathcal I}}
\newcommand{\cJ}{{\mathcal J}}
\newcommand{\cK}{{\mathcal K}}

\newcommand{\cM}{{\mathcal M}}

\newcommand{\cO}{{\mathcal O}}
\newcommand{\cP}{{\mathcal P}}

\newcommand{\cV}{{\mathcal V}}
\newcommand{\cW}{{\mathcal W}}
\newcommand{\cX}{{\mathcal X}}

\begin{document}
\allowdisplaybreaks

\title[Finite dual of a cocommutative Hopf algebroid, differential matrix equations and PV theory.]{On the finite dual of a cocommutative Hopf algebroid.  Application to  linear differential matrix equations and Picard-Vessiot theory.}
\author{Laiachi El Kaoutit}
\address{Universidad de Granada, Departamento de \'{A}lgebra y IEMath. Facultad de Educaci\'{o}n, Econon\'ia y Tecnolog\'ia de Ceuta. Cortadura del Valle, s/n. E-51001 Ceuta, Spain}
\email{kaoutit@ugr.es}
\urladdr{http://www.ugr.es/~kaoutit/}
\author{Jos\'e G\'omez-Torrecillas}
\address{Department of Algebra and CITIC, Universidad de Granada,  
 E18071 Granada, Spain} 
\email{gomezj@ugr.es}
\urladdr{http://www.ugr.es/~gomezj/}
\date{\today}
\thanks{Supported by grants MTM2016--77033-P and  MTM2013--41992-P from the Spanish Ministerio de Econom\'ia y Competitividad and the European Union FEDER}

\maketitle

\vspace{-1.1cm}

\begin{abstract}
A fundamental tool of Differential Galois Theory is the assignment of an algebraic group to each finite-dimensional differential module over differential field in such a way that the category of differential modules it generates is equivalent, as a symmetric monoidal category, to the category of representations of the group. Its underlying set is then recognized as the group of differential automorphisms of the Picard-Vessiot field extension of the base field for this differential module. These results can be obtained by means of a Tannaka reconstruction process, applied to the abelian category of finite-dimensional differential modules. In this paper, we explore the possibility of extending this theory when the differential field is replaced by a more general differential ring $A$. In this case, it is reasonable to deal with differential modules which are finitely generated and projective over $A$. A major obstacle is that this category is not abelian, in contrast with the classical case when $A$ is a field.  To overcome this difficulty, we develop some fundamental results concerning the \emph{finite dual}, hereby introduced, of a cocommutative Hopf algebroid, and a canonical monoidal functor sending (differential) modules to comodules. This functor is proved to be an equivalence of categories whenever a canonical ring homomorphism, which is introduced in this work, with domain in the aforementioned finite dual and with values in the convolutional ring, is injective. Module-theoretical sufficient conditions to get its injectivity are investigated.  This machinery is applied to  differential modules and their Picard-Vessiot theory.
\end{abstract}

\vspace{-0.3cm}
\begin{small}
\tableofcontents
\end{small}

\pagestyle{headings}

\section*{Introduction}\label{sec:intro}

\subsection*{Motivation, background and overview}
The classical Galois theory of linear differential matrix equations with coefficients in a differential field, has two essential interrelated parts. The first one deals with the representation theory of the differential Galois group attached to the system, while the second  seeks for  the space of solutions, namely, a simple differential algebra (the Picard-Vessiot extension) generated by a fundamental solution matrix and the inverse of its determinant. The group of differential automorphisms of this simple differential algebra coincides with the underlying set of the Galois group.  In this paper we aim to explore the possibility of extending the first part of this theory when  differential field  is replaced by a differential ring, making use of the finite dual Hopf algebroid of the co-commutative Hopf algebroid built from a suitable Lie-Rinehart algebra (the module of global sections of the tangent bundle).  
Concerning the second part of the classical theory, we follow  
 Yves Andr\'e's approach given in \cite{Andre:2001}. Although, in \cite{Andre:2001}, no use of groupoids or Hopf algebroids was made, we are able to identify the Picard-Vessiot extension of the differential polynomial algebra for any finite rank differential module, as the total isotropy Hopf algebra (over the base ring) of a commutative Hopf sub-algebroid of this finite dual Hopf algebroid.

Given a Lie-Rinehart algebra $L$ over a Dedekind domain $A$, we construct an affine groupoid over (the spectrum of) $A$ whose category of $A$--profinite\footnote{i.e.,~ finitely generated and projective $A$-modules.} representations is isomorphic, as a monoidal symmetric category, to the category of $A$--profinite right modules over the universal enveloping algebroid of $L$, that is, the category of representations of $L$. In particular, this applies to the global sections of any Lie algebroid over an irreducible smooth curve over an algebraically closed field. For instance, if $A$ is the coordinate ring of the complex affine line $\mathbb{A}^{\Sscript{1}}_{\Sscript{\mathbb{C}}}$ and $L$ is the module of global sections of the transitive Lie algebroid of vector fields, then we have a monoidal equivalence between the category of all differential modules $\Diff{A}$ and the category of $A$-profinite comodules over a commutative Hopf algebroid $U^{\bcirc}$, the \emph{finite dual} of the first Weyl $\mathbb{C}$--algebra $U$ viewed as the universal enveloping algebroid of $L$. Assume now we are given  a differential  $A$-module (or equivalently a linear differential matrix equation) $(M, \partial)$ of rank $m$ and consider the smallest category $\Mo$ generated by $(M,\partial)$ and its dual, and closed under tensor products and sub-quotients. Then, in analogy with the classical differential Galois theory, we show that there exists an affine algebraic $\mathbb{C}$--groupoid $\hHm$ whose category of representations is equivalent as a monoidal symmetric category to $\Mo$.  We construct the representing algebra of (the objects of) $\hHm$ as a finitely generated Hopf sub-algebroid $U_{(M)}^\circ$ of $U^{\circ}$. Furthermore, we show that  the (fibre) groupoid $\hHm(\mathbb{C})$ is transitive and there is a monomorphism  $\hHm(\mathbb{C})\hookrightarrow  \gG^{\Sscript{m}} $ of groupoids, where $\gG^{\Sscript{m}}=\big( \mathbb{A}_{\Sscript{\mathbb{C}}}^1 \times  GL_m(\mathbb{C}) \times  \mathbb{A}_{\Sscript{\mathbb{C}}}^1 ; \, \mathbb{A}_{\Sscript{\mathbb{C}}}^1\big)$ is the  pull-back groupoid of the general linear group $GL_m(\mathbb{C})$ along  the map $\mathbb{A}_{\Sscript{\mathbb{C}}}^1 \to \{*\}$. The algebraic groupoid  $\hHm(\mathbb{C})$ which is unique up to weak equivalences,  is then termed \emph{the differential Galois groupoid} attached to $(M,\partial)$.   Our approach is, in a sense, a naive one, since we do not construct a solution space (perhaps a differential algebra over $A \tensor{\mathbb{C}} A$) whose `groupoid of differential automorphisms' coincides with $\hHm(\mathbb{C})$. We do, however, construct a Picard-Vessiot extension, in the sense of Andr\'e \cite{Andre:2001}, associated to a differential module $M$  as a quotient Hopf $A$-algebra of $U_{(M)}^\circ$ by the Hopf ideal generated by the image of the subtraction of the source from the target. In other words, this is the total isotropy Hopf algebra of $U_{(M)}^\circ$   (Proposition \ref{prop:difPV}).  In this way, any of the isotropy groups of  $\hHm(\mathbb{C})$ is recognized as  the group of differential automorphisms of this total isotropy Hopf algebra (Proposition \ref{prop:AutAsAlgG}).

It is noteworthy to mention that in the literature there are other notions of \emph{Galois groupoid} in the differential context. For instance, Malgrange in \cite{Malgrange:2001}, based on Umemura's theory \cite{Umemura:1996},  introduces a differential Galois groupoid (a Lie groupoid of a foliation). In the previous situation of the differential module $(M,\partial)$ with rank $m=2$, the attached Hopf algebroid of Malgrange's Galois groupoid is the quotient Hopf algebroid of the Hopf algebroid $\mathbb{C}[x,x_{\Sscript{1}}, x_{\Sscript{2}}, y, y_{\Sscript{1}}, y_{\Sscript{2}}, y_{\Sscript{j}}^{\Sscript{\alpha}}, det(y_{\Sscript{j}}^{\Sscript{\epsilon_{i}}})^{-1}]_{\alpha=(\alpha_{0,}\alpha_{1},\alpha_{2}) \,  \in \,  \mathbb{N}^{3}}$ over the algebra $\mathbb{C}[x,x_{\Sscript{1}}, x_{\Sscript{2}}]$, by some differential Hopf ideal, see the last part of  Example \ref{exam:Rank2}. Thus, as a groupoid in the set theoretical sense, the objects set of Malgrange's Galois groupoid is the affine space  $\mathbb{A}_{\Sscript{\mathbb{C}}}^3$  rather than $\mathbb{A}_{\Sscript{\mathbb{C}}}^1$, as it would be expected. Thus, Malgrange's and Umemura's approaches run in a totally different direction, see Subsection \ref{ssec:MalUme}  for more details and explanations.

We already work in a more general context. Thus, let $U$ be  a co-commutative right Hopf algebroid over a commutative ring $A$, and let $\mathcal{A}_U$ be the category of all right $U$--modules which are finitely generated and projective as $A$--modules ($A$--profinite, for short). The category $\mathcal{A}_U$ is monoidal symmetric and rigid, and obviously additive, but it is not abelian in general (even if $A$ is a Dedekind domain). Thus, the Tannakian recontruction from \cite{Deligne:1990,Bruguieres:1994} will not produce a monoidal equivalence from $\mathcal{A}_U$ to the category of $A$--profinite representations of some affine groupoid. However, we will show that such a reconstruction process gives a commutative Hopf algebroid $U^{\bcirc}$ over $A$ and, what is crucial to get our results in the case where $\mathcal{A}_U$ is not abelian, a homomorphism of rings $\zeta$ from $U^{\bcirc}$ to the convolution ring $U^*$. We prove that if  $\zeta : U^{\bcirc} \to U^*$ is injective, then $\mathcal{A}_U$ is isomorphic, as a symmetric monoidal category, to the category $\mathcal{A}^{U^{\bcirc}}$ of $A$--profinite right $U^{\bcirc}$--comodules. In particular, this machinery can be applied when $U = U(L)$ is the universal enveloping algebroid of a Lie-Rinehart algebra $L$ over $A$, producing a an affine groupoid represented by $U^{\bcirc}$ with the `same' representation theory than $L$. 

We give some sufficient conditions for the injectivity of $\zeta$. In particular, we show that $\zeta$ is injective for every co-commutative Hopf algebroid $U$ whenever $A$ is a Dedekind domain.

\subsection*{Description of the main results}

The paper is organized as follows. In Section \ref{preliminares}, besides the statement of the general notations used along the paper,  we collect some basic information on Hopf algebroids an their relationship to Lie-Rinehart algebras.  Section \ref{sec:ICC} is devoted to describe those aspects of the reconstruction process related to corings (or cog\'ebro\"ids, according to \cite{Bruguieres:1994}) that will be useful in the sequel, including a review on the bialgebroids and Hopf algebroids constructed from fibre functors. Principal bi-bundles attached to  pairs of fibre functors are also considered, with the aim of clarifying the Picard-Vessiot extension constructed in Section \ref{sec:Weyl}.

The application of the reconstruction procces to a morphism of (possibly non commutative) rings $A \to R$ leads  in Section \ref{sec:FDRE} to the (right) finite dual $A$--coring $R^{\bcirc}$, and a functor $\chi$ from the category  $\mathcal{A}_R$ of $A$--profinite right $R$--modules and the category $\mathcal{A}^{R^{\bcirc}}$ of $A$--profinite right $R^{\bcirc}$--comodules. In contrast with the case of algebras over a field, it is not known if this functor is an equivalence of categories for a general $A$--ring $R$. We construct a homomorphism of $A$--bimodules $\zeta : R^{\bcirc} \to R^*$, where $R^*$ denotes the right dual, as an $A$--module, of $R$, and it is shown that, if $\zeta$ is injective, then $\chi : \mathcal{A}_R \to \mathcal{A}^{R^{\bcirc}}$ is an equivalence of categories (Proposition \ref{isomor}). A module-theoretical condition implying the injectivity of $\zeta$ are investigated in Proposition \ref{QP}. As a consequence, $\chi$ is an equivalence of categories if $A$ is a right hereditary right noetherian ring (e.g. if it is a Dedekind domain or a semi-simple Artinian ring). 

Section \ref{sec:coHalgd} starts by showing that, if $(U,A)$ is a right bialgebroid, then the map $\zeta: U^{\bcirc} \to U^*$ is a homomorphism of $A \otimes A^{op}$--rings. In fact, $\zeta$ becomes part of a canonical structure of left bialgebroid $(U^{\bcirc},A)$. Theorem \ref{thm:lduality} contains our main contribution of Section \ref{sec:coHalgd}, namely, the fact that if $U$ is a cocommutative Hopf algebroid over a commutative ring $A$, then  $\chi$ is a monoidal equivalence between $\mathcal{A}_U$ and $\mathcal{A}^{U^{\bcirc}}$. 

Section \ref{sec:Weyl} developes the aforementioned application of our theory to differential modules, as described above.  Besides, in Subsection \ref{ssec:MalUme}, we explicitly describe the  structure maps of the commutative Hopf algebroid attached to  what is in the literature known as Malgrange's groupoid (or $D$-groupoid) \cite{Malgrange:2001}. We also discuss in this section, using  some specific examples,  Umemura's \cite{Umemura:2009} approach to these groupoids.

\section{Preliminaries: Bialgebroids and Hopf algebroids}\label{preliminares}
In this section, for the convenience of the reader unfamiliar with Hopf algebroids, we introduce some basic notions and results concerning this theory. We also fix some notations and terminology to be used along the paper. 

\subsection{Notations and basic notions} 
We work over a unital commutative ground ring $K$. All additive categories will be assumed to be $K$--additive categories, and additive functors between them will be $K$--linear.  For a category $\cat{A}$, the notation $X \in \cat{A}$ means that $X$ is an object of $ \cat{A}$. The identity morphism of an object $X \in \cat{A}$ is denoted simply by the object itself or  by the symbol $1_X$. Algebra means associative and unital $K$-algebra and the notation $Z(A)$ stands for the center of an algebra $A$.   We denote by $\Ae:=A\tensor{K}\Ao$ the enveloping algebra of an algebra $A$, where $A^{\Sscript{op}}$ is the opposite algebra of $A$, whose elements are distinguished from those of $A$ by using the notation $\ao \in A^{\Sscript{op}}$, for $a \in A$.  For simplicity, the same notation $\Ae:=A\tensor{L}A^{\Sscript{op}}$ will be used when $K \to L$ is commutative ring extension and $A$ is an $L$-algebra.  Modules are unital modules,
and $(A,B)$-bimodules  are assumed to be central $K$-bimodules.
We denote by $\bimod{A}{B}$ the category of all
$(A,B)$-bimodules; $\lmod{A}$ and $\rmod{B}$ denote, respectively,
the category of left $A$-modules and the category of right
$B$-modules. The corresponding hom-set functors will be denoted, respectively, by $\lhom{A}{-}{-}$ and $\rhom{B}{-}{-}$.  For every $(A,B)$-bimodule ${}_AX_B$, we denote by
$X^*=\rhom{B}{X}{B}$ its right dual, while by  ${}^*X=\lhom{A}{X}{A}$ its left dual,  which are considered as $(B,A)$-bimodules with the actions: 
\begin{equation}\label{Eq:phipsi}
b\varphi a: x \mapsto b\varphi(ax),\;\; b\psi a: x \mapsto \psi(xb) a, \text{ for all }x \in X, \varphi \in X^*, \psi \in {}^*X, a \in A, b \in B.  
\end{equation}

For every algebra $A$, 
let $add(A_A)$ denote the full sub-category of $\rmod{A}$ whose objects are
all  finitely generated and projective (\emph{fgp}~ for short)
right $A$-modules; we also use the terminology \emph{right $A$-profinite modules} for the objects in this sub-category. Given a morphism  of algebras $\eta:A \to R$, we denote by  $\eta_{*}: \rmod{R} \to \rmod{A}$ and ${}_{*}\eta: \lmod{R} \to \lmod{A}$ the corresponding restriction of scalars functors. It is clear that, if $\eta_{*}(R)_A$ is fgp  we then have a functor $\eta_{*}: add(R_R) \to add(A_A)$. 

An algebra $R$ is said to be an \emph{$A$-ring} if there is a morphism of algebras $ \eta: A \to R$ (also called a \emph{ring extension}). By restriction of scalars, $R$ becomes an $A$-bimodule and its multiplication factors through the tensor product over $A$, that is, its multiplication can be understood as a map $\mu: R\tensor{A}R \to R$. In this way  the triple $(R,\eta, \mu)$ can be seen as a \emph{monoid} in the monoidal category $\bimod{A}{A}$. Saying that $R$ is an $\Ae$-ring is equivalent to say  that there is  a  morphism of algebras $\Sf{s}: A \to R$ and  an anti-morphism of algebras $\Sf{t}: A \to R$  such that $\Sf{s}(a)\Sf{t}(a')=\Sf{t}(a')\Sf{s}(a)$, for every $a, a' \in A$.

Dually,  an $A$-\emph{coring} \cite{Sweedler:1975} is a \emph{comonoid} object in the category $\bimod{A}{A}$. That is, a triple $(\coring{C}, \varepsilon, \Delta)$, where 
$\coring{C}$ is an $A$-bimodule together with two $A$-bimodule maps, the comultiplication and the counit $\Delta: \coring{C} \longrightarrow
\coring{C}\tensor{A}\coring{C}$, $\counidad:
\coring{C} \longrightarrow A$, which satisfy the usual coassociativity and counitary properties. A \emph{right
$\coring{C}$-comodule} is a right $A$-module $M$ together with a
right $A$-linear map $\varrho_{M}: M \rightarrow
M\tensor{A}\coring{C}$ (called right coaction) such that
$$
(M\tensor{A}\Delta_{\coring{C}}) \circ \varrho_M =
(\varrho_M\tensor{A}\coring{C}) \circ \varrho_M, \;\;  \text{ and  }\;\;
(M\tensor{A}\counidad_{\coring{C}}) \circ \varrho_M=M. 
$$
A \emph{morphism} of right $\coring{C}$-modules $f: M \rightarrow
N$ (right colinear map) is right $A$-linear map such that $(f\tensor{A}\coring{C})
\circ \varrho_M = \varrho_N \circ f$. 
The category of all right
$\coring{C}$-comodules and their morphisms will be denoted by
$\rcomod{\coring{C}}$.

We recall from \cite{Schau:BONCRAASTFHB, Schau:DADOQGHA,  Swe:GOSA, Tak:GOAOAA} the definition and basic properties of (left, right) Hopf algebroids.  A good survey on the subject is the monograph \cite{Boe:HA}.

\subsection{Bialgebroids and Hopf algebroids}\label{ssec:bialgd}
Fix $A$ a ground $K$-algebra, and let $V$ be an $A^{\Sf{e}}$-ring via a ring homomorphism $\eta: A^{\Sf{e}} \to V$ whose  \emph{source} and \emph{target} maps are respectively denoted by  $\Sf{s}: A \to V$ and $\Sf{t}: A^{op} \to V$. The pair 
$(A, V)$ is said to be a \emph{left biagebroid} (or $V$ is a \emph{left} $\times_A$-\emph{bialgebra}) provided that its category of left $V$-modules  is a monoidal category and the restriction of scalars functor $\Scr{O}_l: \lmod{V} \to \lmod{A^{\Sf{e}}}$, $\Scr{O}_l={}_{*}(\Sf{s}\tensor{}\Sf{t})$, is a strict monoidal functor.
In particular, the underlying bimodule ${}_{A^{\Sf{e}}}V$ admits a structure of $A$-coring, and  if we denote by $\diama$ the given tensor product of $\lmod{V}$ and consider two objects $X, Y \in \lmod{V}$, then the underlying left $A^{\Sf{e}}$-module of $X \diama Y$ is the $A$-bimodule $X\tensor{A}Y$ whose  left $V$-action is given by the formula 
\begin{equation}\label{Eq:accion}
v. (x\tensor{A}y)\,\,=\,\, v_1x \tensor{A} v_2y,
\end{equation}
where $\Delta(v)= v_1\tensor{A}v_2$ is the comultiplication of the underlying $A$-coring of $V$ (we are using a simplified version of Sweedler's notation with sum understood).

It is known from \cite[Proposition 3.3]{Schau:DADOQGHA} that the category $\lmod{V}$ is right closed with right inner hom-functors 
\begin{equation}\label{Eq:LIH}
hom_{\lmod{V}}(X,Y)\,:=\, \lhom{V}{V\diama X}{Y},
\end{equation}
where the left $V$-action of the right hand term is given by the $V$-bimodule $V\diama X$ whose left $V$-action is \eqref{Eq:accion} and its right $V$-action is induced by that of the factor $V$. This means that for any object $X \in \lmod{V}$, the functor $-\diama X : \lmod{V} \to \lmod{V}$ is left adjoint  to the functor $\lhom{V}{V\diama X}{-} :\lmod{V} \to \lmod{V}$ called the \emph{right inner hom-functor} and denoted by $hom_{\lmod{V}}(X,-)$. This adjunction is easily seen once observed that $Y\diama X \cong (V\diama X) \tensor{V} Y$ is a left $V$-linear isomorphism via the map $x\tensor{A}y \mapsto (1\tensor{A}x) \tensor{V} y$, and it uses the usual Hom-tensor adjunction isomorphism to get
$$
\lhom{V}{Y\diama X}{Z} \,\cong\, \lhom{V}{(V\diama X) \tensor{V} Y}{Z}\, \cong \, \lhom{V}{Y}{\lhom{V}{V\diama X}{Z}}. 
$$

{
If we want to compute the right inner hom-functors  in the category of modules  $\lmod{V}$ via the forgetful functor $\Scr{O}_l$ to the monoidal category $\lmod{A^{\Sf{e}}}$,  then it is better to resume the previous situation in form of a diagram. So with the previous notations,  we have  a commutative diagram
\begin{footnotesize}
$$
\xymatrix@C=60pt{ \lhom{{V}}{Y\diama X}{Z} \ar@{->}^-{\cong}[rr] \ar@{_{(}->}_-{\Scr{O}_l}[d]& &   \lhom{{V}}{Y}{\lhom{V}{{V}\diama X}{Z}} \ar@{_{(}->}_-{\Scr{O}_l}[d] \\ 
\lhom{A^{\Sf{e}}}{Y\tensor{A} X}{Z} \ar@{-->}^-{}[rr]  \ar@{->}_-{\cong}[d]  & &    \lhom{A^{\Sf{e}}}{Y}{\lhom{{V}}{V\diama X}{Z}} \ar@{->}_-{\lhom{A^{\Sf{e}}}{Y}{\,T_Z}}[d]  
 \\   \lhom{A^{\Sf{e}}}{Y}{\lhom{A^{op}}{X}{Z}} \ar@{->}^-{\cong}[rr] &  &  \lhom{A^{\Sf{e}}}{Y}{\lhom{{V}}{{V}\tensor{A^{op}} X}{Z}},   }
$$
\end{footnotesize}
where the natural transformation $T_Z: \lhom{{V}}{{V}\diama X}{Z} \to \lhom{{V}}{{V}\tensor{A^{op}} X}{Z}$ is defined by sending $f \mapsto \left[ v\tensor{A^{op}}x \mapsto f(v_1\tensor{A}v_2x) \right]$. The dashed arrow is then a natural  isomorphism if and only if $T_{-}$ is a natural isomorphism, if and only if each map $\beta_X: {V}\tensor{A^{op}}X \to {V}\diama X$ sending $v\tensor{A^{op}}x \mapsto v_1\tensor{A}v_2x$ is an isomorphism of left ${V}$-modules, where in the tensor product ${V}\tensor{A^{op}}X$ we have used the bimodule ${}_{{V}}{V}_{1\tensor{}A^{op}}$. Now, one can easily see that  the maps $\beta_X$  are isomorphisms if and only if  $\beta_{{V}}$ is an isomorphism.  

}

Following \cite[Theorem 3.5]{Schau:DADOQGHA}, a left $A$-bialgebroid $V$ is said to be a \emph{left Hopf algebroid} provided that the functor $\Scr{O}_l$ preserves right inner hom-functors. As we have seen this  is equivalent to say that $\beta_{V}$  is an isomorphism. Its inverse induces then a well defined map  $\beta_{V}^{-1}(-\tensor{A}1): V \to V\tensor{A^{op}} V$ sending $v \mapsto v_{+}\tensor{A^{op}} v_{-}$ (the tensor product is defined using the bimodules ${}_{V}V_{1\tensor{}A^{op}}$ and ${}_{1\tensor{}A^{op}}V_{V}$).    This map is known as the anti-multiplication and its  definition for bialgebras over fields extensions goes  back  to W. D. Nichols \cite[Definition 4.]{Nichols:1985}.   Here the inverse map $\beta_{V}^{-1}(-\tensor{A}1)$ will be referred to as \emph{the translation map}.

 Analogously, an $A^{\Sf{e}}$-ring $ U$ with source and target $\Sf{s}: A \to  U$ and $\Sf{t}: A^{\Sscript{o}} \to  U$,   is said to be a  \emph{right} $A$\emph{-bialgebroid}  whenever its category of right modules is a monoidal category and the restriction of scalars functor $\Scr{O}_r: \rmod{ U} \to \rmod{A^{\Sf{e}}}$ ($\Scr{O}_r=(\Sf{s}\tensor{}\Sf{t})_{*}$) is a strict monoidal functor. In this case the category $\rmod{ U}$ is left closed with left inner hom-functors
\begin{equation}\label{Eq:RIH}
hom_{\rmod{ U}}(X,Y)\,:=\, \rhom{ U}{X\diama \, U}{Y}.
\end{equation}

For each object $X \in \rmod{ U}$, we have in this case that  the functor $X\diama -$ is left adjoint to 
$\rhom{ U}{X\diama\,  U}{-}$, and 
the adjunction is given by the natural isomorphism 
\begin{equation}\label{Eq:Radjunction}
\rhom{ U}{X\diama Y}{Z} \,\cong\, \rhom{ U}{Y \tensor{ U} (X\diama\,  U)}{Z}\, \cong \, \rhom{ U}{Y}{\rhom{ U}{X\diama \, U}{Z}}. 
\end{equation}
The corresponding $\beta$-maps in the category $\rmod{ U}$ are defined by $\beta_X: X\tensor{A^{op}} U \to X \diama \,  U$ sending $x \tensor{A^{op}} u \mapsto xu_1 \tensor{A} u_2$, where in the first tensor product we have used in its second factor the bimodule ${}_{1\tensor{}A^{op}}  U_{ U}$. One can check as before that $\Scr{O}_r$ preserves left inner hom-functors if and only if $\beta_{ U}$ is an isomorphism. So the right $A$-bialgebroid $ U$ is said to be a \emph{right Hopf} $A$\emph{-algebroid} provided that $\beta_{ U}$ is an isomorphism. In this case, \emph{the translation map} $\beta^{-1}_{ U}(1\tensor{A}-):  U \to  U\tensor{A^{op}} U$, where the first factor of the tensor product is the $A$-bimodule $ U_{A^{\Sf{e}}}$, will be denoted by $u \mapsto u_{-}\tensor{A^{op}} u_{+}$.  
As in \cite[Proposition 3.7]{Schau:DADOQGHA}, the map $\beta^{-1}_{ U}(1\tensor{A}-)$ enjoy a list of properties. Here we mention few of them  which we will need in the sequel:  First note that $\beta_{ U}^{-1} (u'\tensor{A} u)=u'u_-\tensor{A^{op}}u_+$, so we have
\begin{eqnarray}
u_{1,\, -} \tensor{A^{op}} u_{1,\,+}\tensor{A} u_2 &=& u_-\tensor{A^{op}} u_{+,\, 1}\tensor{A}u_{+,\, 2} \,\, \in \, ({U}_{1\tensor{}A^{op}} \,\tensor{A^{op}} \, {}_{1\tensor{}A^{op}}{U})\, \diama \, {U} \label{Eq:B5}  
 \\
\Sf{s}(b)\tensor{A^{op}} \Sf{s}(a) &=& \eta(a\tensor{}b^o)_{\, -} \tensor{A^{op}} \eta(a\tensor{}b^o)_{\, +} \,\, \in \,  {U}_{1\tensor{}A^{op}} \,\tensor{A^{op}} \, {}_{1\tensor{}A^{op}}{U}. \label{Eq:B6}
\end{eqnarray}

In case $A=K$ which, by our conventions, automatically  implies  that $\Sf{s}=\Sf{t}$ and that $V$ (or $ U$) is an ordinary bialgebra, it is well know that $V$ is Hopf algebra if and only if $\beta_{V}$ is bijective.

\begin{remark}\label{rem:S}
For a general left $A$-bialgebroid $V$, there is no hope of obtaining from the translation map  $\beta_{V}^{-1}(-\tensor{A}1)$ an endomorphism of $V$ which could play the role of the antipode as in the case of Hopf algebras. Nevertheless, if we  assume that $A$ and $V$ are commutative rings with $\Sf{s} \neq \Sf{t}$, which means that $V$ is a commutative $A$-bialgebroid\footnote{We use the terminology: $(A,V)$ is a \emph{commutative Hopf algebroid over $K$}.},  then the map $\Scr{S}=(\varepsilon \tensor{A} V) \circ \beta_{V}^{-1}(-\tensor{A}1)$ is well defined and gives the antipode for $V$. Conversely,   the inverse of $\beta_{V}$ is given by $\beta_{V}^{-1}(v\tensor{A}v')=v_1\tensor{A}\Scr{S}(v_2)v'$, whenever $V$ having $\Scr{S}$ as an antipode. Of course in this case the notions of bialgebroid and Hopf algebroid are obviously independent from the sides.  
\end{remark}

Let $(A,  U)$ be a right Hopf algebroid. Then the adjunction \eqref{Eq:Radjunction} and the analogue natural isomorphism to $T_{-}$, give a  right ${U}$-action on the right $A^{\Sf{e}}$-module $\rhom{A^{op}}{X}{Y}$, for every pair of right $ U$-modules $X$ and $Y$. This action is given by 
\begin{equation}\label{Eq:hom-action}
f.\, u : X \longrightarrow Y, (\, x \longmapsto f(xu_{-}) \, u_{+}\,), \text{ where, as before, } \beta_{ U}^{-1}(1\tensor{A}u)\,=\, u_{-}\tensor{A^{op}} u_+.
\end{equation}
Now, fix a right $ U$-module $X$ and consider its left dual ${}^*X$ also as right $ U$-module with the action \eqref{Eq:hom-action} by taking $Y_{ U}=A_{ U}$, the unit object of $\rmod{ U}$, with  right $ U$-action $a\,.\, u= \varepsilon( \Sf{s}(a)u)=\varepsilon(\Sf{t}(a) u)$. 
\begin{equation}\label{Eq:ev-db-1}
\xymatrix@R=0pt{X\diama\, {}^*X \ar@{->}^-{ev}[rr] & & A, \\  x\tensor{A} \varphi \ar@{|->}[rr] & & \varphi(x) }  \qquad \xymatrix@R=0pt{ A \ar@{->}^-{db'}[rr] & & \rhom{A^{op}}{X}{X} \\
1 \ar@{|->}[rr] & & 1_X,  }
\end{equation}

The following lemma  is the version on the right of  \cite[Proposition 4.41]{Boe:HA}.
\begin{lemma}\label{lema:duales}
Let $(A,  U)$ be a right Hopf algebroid. Then every right $ U$-module $X$ which is finitely generated and projective as a left $A$-module, admits ${}^*X$ as the right dual in the monoidal category $\rmod{ U}$. 
\end{lemma}
\begin{proof}
Using the maps of Eq.\eqref{Eq:ev-db-1}, the right dual of $X$ is given by the triple $({}^*X, ev, j_X^{-1} \circ\, db')$, where $$ j_X: {}^*X\diama\, X \longrightarrow  \rhom{A^{op}}{X}{X}, \, (\varphi\tensor{A}x  \longmapsto \left[ x' \mapsto \varphi(x') x\right]),$$ is the canonical isomorphism.  To show that $j_X$ is in  fact a right $ U$-linear map, one use the properties of the comultiplication of the underlying $A$-coring of $ U$  as well as Eq.\eqref{Eq:B5}.
\end{proof}

\subsection{Examples of (right) Hopf algebroids.}\label{ssec:examples}
Any Hopf algebra over a commutative ring is obviously a (left and right) Hopf algebroid. Below, we list some non trivial examples of Hopf algebroids, specially the ones with commutative base ring, which we will deal with in the forthcoming sections.

\begin{example}\label{exam:Hco}\cite{Ravenel:1986} 
Assume that the ground ring $K$ is a field and let $A$ be a commutative $K$-algebra which is a right $H$-comodule $K$-algebra with coation $\varrho_A: A \to A\tensor{K}H$, where $H$ is a commutative Hopf algebra. Consider $A\tensor{K}H$ as an $(A\tensor{K}A)$-ring via the following source and target:
$$ \Sf{s}: A \to A\tensor{K}H, (a \mapsto a\tensor{K}1), \quad \Sf{t}=\varrho_A: A \to  A\tensor{K}H, (a \mapsto a_0\tensor{}a_1).$$ Then $A\tensor{K}H$ is a commutative Hopf algebroid over $K$, with antipode $\Scr{S}: A\tensor{K}H \to A\tensor{K}H$ sending $a\tensor{}h \mapsto a_{0}\tensor{}a_{1}S(h)$, where $S$ is the antipode of $H$. 
\end{example}

\begin{example}\label{exam:LR} {
Let $A$ be a commutative algebra over a field $K$ of characteristic $0$, and denote by $\mathrm{Der}_{\Sscript{K}}(A)$ the Lie algebra of all $K$-linear derivations of $A$.  Consider an $A$--module endowed with a structure of $K$--Lie algebra $L$, and let $\omega: L \to \mathrm{Der}_{\Sscript{K}}(A)$ be a morphism of $K$-Lie algebras. 
The pair $(A,L)$ is called \emph{Lie-Rinehart algebra} with \emph{anchor} map $\omega$, provided
\begin{eqnarray*}
 (aX) (b) &=&  a(X(b)), \qquad  \\
{[ X, aY ]} &=& a{[X,Y]}+X(a)Y,
\end{eqnarray*}
for all $X, Y \in L$ and $a, b \in A$, where $X(a)$ stands for $\omega(X)(a)$. Consider the (left) $A$-module direct sum $A\oplus L$ as a $K$-Lie algebra with the bracket:
$$
\Big[(a,X), (b,Y)\Big]\,=\, \Big( X(b)-Y(a), [X,Y]  \Big),
$$ 
for any $a, b \in A$ and $X, Y \in L$. Denote by $\tau: A \oplus L \to U(A\oplus L)$ the canonical inclusion into  its universal enveloping $K$-algebra.

As it was expounded in \cite{Rin:DFOGCA, Hue:PCAQ, Moerdijk/Mrcun:2010},  associated to any Lie-Rinehart algebra $(A, L)$ there is a universal object denoted by $(A, \cat{U}(L))$. As an algebra,  $\cat{U}(L) = U(A\oplus L)/I$, where $I$ is the two-sided ideal generated by the set $\Big\{\tau(a',0).\tau(a,X) - \tau(a'a,a'X)|\, a, a' \in A,\, X \in L\Big\}$. There are canonical maps
$$
\iota_{\Sscript{A}}: A \to \cat{U}(L), \;\; \Big( a \mapsto a +I  \Big)\quad \text{ and } \quad \iota_{\Sscript{L}}:L \to \cat{U}(L),\;\; \Big(X \mapsto X + I \Big).
$$
The first one is an algebra map (whose image is not necessarily in the center), while the second is a $K$-Lie algebra map. Both maps are compatible in the sense that the following equations are fulfilled: 
$$
\iota_{\Sscript{A}}(a)\iota_{\Sscript{L}}(X) =\iota_{\Sscript{L}}(aX), \qquad [\iota_{\Sscript{L}}(X),\iota_{\Sscript{A}}(a)]=\iota_{\Sscript{A}}(X(a)),$$ 
for any $a \in A$ and $X \in L$. Such equations determine in fact the universality of $\cat{U}(L)$. 

As observed in \cite[\S4.2.1]{Kow:HAATCT},  the usual Hopf algebra structure of $U(A\oplus L)$ can be lifted to a structure of  cocommutative (right) Hopf $A$-algebroid on $\cat{U}(L)$. The source and target are equal: $\Sf{t}=\Sf{s} = \iota_{\Sscript{A}}$.  The $A$-coring structure is an $A$-coalgebra structure whose underlying $A$-bimodule uses the right $A$-module structure derived from $\iota_{\Sscript{A}}$, that is, the $A$-bimodule $\cat{U}(L)_{\Sscript{A}}$ with two-sided action $a.u.a'= u \, \iota_{\Sscript{A}}(aa')$, for every $a, a' \in A$ and $u \in \cat{U}(L)$  (recall that $A$ is commutative).   The comultiplication and the counit of $\cat{U}(L)_{\Sscript{A}}$ are given on generators by 
$$
\Delta(\iota_{\Sscript{L}}(X)) = \iota_{\Sscript{L}}(X)\tensor{A} 1_{ \Sscript{\cat{U}(L)}} + 1_{\Sscript{\cat{U}(L)}} \tensor{A} \iota_{\Sscript{L}}(X),\; \varepsilon(\iota_{\Sscript{L}}(X))=0, \quad \Delta(\Sf{s}(a))= \Sf{s}(a)\tensor{A}1_{ \Sscript{\cat{U}(L)}}=1_{ \Sscript{\cat{U}(L)}}\tensor{A}\Sf{s}(a), \; \varepsilon(\Sf{s}(a))=a.
$$
for any $a \in A$ and $X \in L$. The translation map $\beta_{\Sscript{\cat{U}(L)}}: \cat{U}(L) \to \cat{U}(L)_{\Sscript{A}} \,\tensor{A} \, {}_{\Sscript{A}}\cat{U}(L)$  is given on generators by
\begin{eqnarray*}
\Sf{s}(a)_{-}\tensor{A} \Sf{s}(a)_{+} & := & 1_{\Sscript{ \cat{U}(L)}} \tensor{A}\Sf{s}(a), \\ \iota_{\Sscript{L}}(X)_{-}\tensor{A}\iota_{\Sscript{L}}(X)_{+} &:=& 1_{\Sscript{ \cat{U}(L)}} \tensor{A} \iota_{\Sscript{L}}(X) -  \iota_{\Sscript{L}}(X)\tensor{A} 1_{\Sscript{\cat{U}(L)}}.
\end{eqnarray*} }
\end{example}

\begin{example}\label{Exam: bundles}
Here are some basic examples of Lie-Rinehart algebras.
\begin{enumerate}
\item The pair $(A, \mathrm{Der}_{K}(A))$ obviously admits the structure of a Lie-Rinehart algebra. 
\item\cite[Definition 3.3.1]{Mackenzie:2005} A \emph{Lie algebroid}  is a  vector bundle $\mathcal{E} \to \mathcal{M}$ over a smooth manifold, together with a map $\omega: \mathcal{E} \to T\mathcal{M}$ of vector bundles and a Lie structure $[-,-]$ on the  vector space  $\Gamma(\cat{E})$  of global smooth sections of $\mathcal{E}$, such that the induced map $\Gamma(\omega): \Gamma(\cat{E}) \to \Gamma(T\mathcal{M})$ is a Lie algebra homomorphism, and for all $X, Y \in \Gamma(\cat{E})$ and any $f \in \mathcal{C}^{\infty}(\mathcal{M})$  one has $[X,fY]\,=\, f[X,Y]+ \Gamma(\omega)(X)(f)Y$.  Then the pair $(\mathcal{C}^{\infty}(\mathcal{M}), \Gamma(\cat{E}))$ becomes a Lie-Rinehart algebra.
\item Let $(A, \partial_i)_{i=1,\cdots, n}$ be a partial differential commutative $K$-algebra, that is, $\partial_i$ are $K$-algebra derivations of $A$ such that $\partial_i \circ \partial _j \,=\, \partial_j \circ \partial_i$, for every $i \neq j$. Consider the $A$-module $L:= \oplus_i^n A. \partial_i $ as a Lie $K$--algebra with bracket 
$$
[a.\partial_i, a'.\partial_j] \,=\,    a\partial_i(a'). \partial_j - a'\partial_j(a) . \partial_i, \quad \text{ for every }\; i,j= 1, \cdots, n.
$$
Then $L$ admits canonically a structure of  Lie-Rinehart algebra whose anchor map is given by:  $$
\omega: L \longrightarrow  \mathrm{Der}_{K}(A), \qquad \big[a .\partial_i \longmapsto [a' \mapsto a\partial_i(a')]\big].
$$
\end{enumerate} 
\end{example}

\begin{example}\label{exam:situDual}
Under the assumptions of Example \ref{exam:LR}, it is clear that the anchor map $\omega: L \to {\rm Der}_{K}(A)$ can be extended to an algebra map  $\Omega: U(L) \to {\rm End}_{K}(A)$ which gives a left $U(L)$-action on the base algebra $A$. If we further assume that $A$, with this action, is a locally finite left module (i.e.~each element generates a left $U(L)$-module which is a finite dimensional $K$-vector space),  then $A$ becomes a right comodule algebra over the usual finite dual commutative  Hopf algebra $U(L)'$, see Example \ref{exam:dualcoalg}. In this way, as in Example \ref{exam:Hco} we obtain another (commutative) Hopf algebroid, namely, $A\tensor{K}U(L)'$.
\end{example}

{
\begin{example}[Base extension]\label{exam:Bext}
Let $(A,\cat{H})$ be a commutative Hopf algebroid and $\phi: A\to B$ be a morphism of commutative algebras. Then $(B,B\tensor{A}\cat{H}\tensor{A}B)$ admits a structure of commutative Hopf algebroid called \emph{the base ring extension} of $(A,\cat{H})$ by $\phi$.  The structure maps of this Hopf algebroid are given as follows:
\begin{multline*}
\Sf{s}(b) \,=\, b\tensor{A}1\tensor{A}1, \;\; \Sf{t}(b')\,=\, 1\tensor{A}1\tensor{A}b',\;\; \sS(b\tensor{A}h\tensor{A}b')\,=\, b'\tensor{A}\sS(h)\tensor{A}b, \\
\Delta(b\tensor{A}h\tensor{A}b')\,=\, \sum_{\Sscript{(h)}} (b\tensor{A}h_{1}\tensor{A}1)\, \tensor{B} \, (1 \tensor{A}h_{2}\tensor{A}b'), \qquad \varepsilon(b\tensor{A}h\tensor{A}b')\,=\,bb'\phi(\varepsilon(h)).\hspace{1cm}
\end{multline*}
\end{example}
}

\section{The reconstruction process,   Galois corings and principal bi-bundles.}\label{sec:ICC}
We recall in this section  the construction of the universal coring from a given fibre functor \cite{Bruguieres:1994}. We follow the presentation given in  \cite{ElKaoutit/Gomez:2004b}.
We also recall the notion of Galois coring. Roughly speaking, these are corings which can be reconstructed from their class of right comodules which are finitely generated and projective as right modules over the base algebra. 

\subsection{Infinite comatrix corings}\label{ssec:comatrix}
Let $\omega: \cat{A} \rightarrow add(A_A)$ be an additive faithful functor (referred to as  a \emph{fibre functor}), where $\cat{A}$ is an additive small category and $A$ is an algebra, which is not assumed to be commutative. Here, $add(A_A)$ denotes the category of all finitely generated and projective right $A$--modules. The image of an object $P$ of $\cat{A}$ under $\omega$ will be denoted
by ${}^{\omega}P:=\omega(P)$, or even by $P$ when no confusion may
be expected. Given $P,Q\in \cat{A}$, we denote by
$T_{PQ}=\hom{\cat{A}}{P}{Q}$ the $K$-module of all morphisms in
$\cat{A}$ from $P$ to $Q$. The symbol $T_P$ is reserved to the
ring of endomorphisms of $P$.  Clearly, $S_P = \rend{{}^{\omega}P}{A}$
is a ring extension of $T_P$ via $\omega$. In this way, every
image ${}^{\omega}P$ of an object $P \in \cat{A}$, becomes canonically
a $(T_P,A)$-bimodule, and this bi-action can be extended to the
following $(T_Q,A)$-bimodule map
\begin{equation}\label{omega1}
\xymatrix@R=0pt@C=40pt{ T_{PQ}\tensor{T_P}P \ar@{->}^-{\omega_{PQ}}[r] & Q \\
t\tensor{T_P}p \ar@{|->}[r] & tp=\omega(t)(p), }
\end{equation}
since $T_{PQ}$ is already a $(T_Q,T_P)$-bimodule. The dual bi-action is given by the following
$(A,T_P)$-bimodule map
\begin{equation}\label{omega2}
\xymatrix@R=0pt@C=40pt{ Q^* \tensor{T_Q} T_{PQ} \ar@{->}^-{\omega_{Q^*P^*}}[r] & q^* \\
q^*\tensor{T_Q} t \ar@{|->}[r] & q^* t = q^* \circ \omega(t). }
\end{equation}

For every object $P
\in \cat{A}$, one can define its associated finite comatrix
$A$-coring $P^*\tensor{T_P}P$ with the following comultiplication
and counit
\begin{equation}\label{Eq:Deltapq}
\xymatrix@R=0pt@C=40pt{\rcomatrix{T_P}{P}
\ar@{->}^-{\Delta_{\rcomatrix{T_P}{P}}}[r] & \rcomatrix{T_P}{P}
\tensor{A} \rcomatrix{T_P}{P}, \\ p^*\tensor{A}p \ar@{|->}[r] &
\sum_{\alpha_P} p^*\tensor{T_P}e_{\alpha_P} \tensor{A}
e_{\alpha_P}^*\tensor{T_P}p } \quad
\xymatrix@R=0pt@C=30pt{\rcomatrix{T_P}{P}\ar@{->}^-{\counidad_{\rcomatrix{T_P}{P}}}[r]
& A \\ p^*\tensor{T_P} p \ar@{|->}[r] & p^*(p) }
\end{equation}
where $\{(e_{\alpha_P},e_{\alpha_P}^*)\} \subset P \times
P^*$ is any right dual basis for $P_A$. The map
$\Delta_{\rcomatrix{T_P}{P}}$ does not depend on the choice of the
dual basis \cite[Remark 2.2]{ElKaoutit/Gomez:2003a}. Now consider
the following direct sum of $A$-corings
\begin{eqnarray*}
  \fk{B}\left(\cat{A}\right) &=& \bigoplus_{P \in
  \cat{A}} P^*\tensor{T_P}P
\end{eqnarray*}
and its $K$-submodule $\fk{J}_{\Sscript{\omega}}$ generated by
the set
\begin{equation}\label{Eq:JA}
\left\{\underset{}{} q^* \tensor{T_Q} tp - q^* t \tensor{T_P}p : \, q^* \in Q^*,\, p
\in P,\, t \in T_{PQ},\, P,Q \in \cat{A} \right\},
\end{equation}
where the products are defined by the pairings of Eqs.~\eqref{omega1} and \eqref{omega2}.
By \cite[Lemma 4.2]{ElKaoutit/Gomez:2004b}, $\fk{J}_{\Sscript{\omega}}$ is a
coideal of the $A$-coring $\fk{B}(\cat{A})$. Therefore, we can
consider the quotient $A$-coring
\begin{eqnarray}\label{comatInf}
  \Scr{R}\left( \cat{A}\right) :\,=\, \fk{B}(\cat{A})/\fk{J}_{\Sscript{\omega}} \,=\, \left(\underset{P\,\in\, 
\cat{A}}{\bigoplus} \rcomatrix{T_P}{P}\right)/ \fk{J}_{\Sscript{\omega}}
\end{eqnarray}
and this is the \emph{infinite comatrix $A$-coring} associated to the fibre functor
$\omega: \cat{A} \rightarrow add(A_A)$.  Furthermore, it is clear that any object $P \in \cat{A}$ admits (via the functor $\omega$) the structure of a right $\Scr{R}(\cat{A})$-comodule, which leads to a well defined  functor $\chi: \cat{A} \to \rcomod{\Scr{R}(\cat{A})}$.

\textsc{Notation}: We will denote by $\bara{\varphi \tensor{T_{P}} p} := \varphi \tensor{T_{P}} p + 
\cat{J}_{\Sscript{\omega}}$, for $\varphi \in \omega(P)^{*}, p \in \omega(P)$ and $P \in \cat{A}$,  a generator element  in the infinite comatrix $A$-coring $\rR(\cat{A})$. 

Under this notation the comultiplication of $\rR(\cat{A})$ is given by 
\begin{equation}\label{Eq:delta}
\Delta :  \rR(\cat{A}) \longrightarrow \rR(\cat{A}) \tensor{A} \rR(\cat{A}), \quad \overline{p^* \tensor{T_P} p} \longmapsto \sum_{\alpha_{P}}\overline{p^* \tensor{T_P} e_{\alpha,P} }\tensor{A} \overline{e_{\alpha,P}^* \tensor{T_P} p},
\end{equation}
where $\{ e_{\alpha,P}, e_{\alpha,P}^* \}$ denotes a finite dual basis for $P_A=\omega(P)_{A}$. The counit of $\rR(\cat{A})$ is
\begin{equation}\label{Eq:counit}
\varepsilon: \rR(\cat{A}) \longrightarrow A, \quad \overline{p^* \tensor{T_P}  p} \longmapsto p^*(p). 
\end{equation}

The above construction is in fact functorial, in the sense that if $\cat{F}: \cat{A} \to \cat{A}'$ is a $K$-linear functor  and $\xymatrix{\omega: \cat{A} \ar@{->}[r] &  add(A_A) & \ar@{->}[l] \cat{A}' :\omega'}$ are fibre functors such that $\omega' \circ \cat{F} = \omega$,  then there is a morphism of $A$-corings $\Scr{R}(\cat{F}) : \Scr{R}(\cat{A}) \to \Scr{R}(\cat{A}')$.

There are two typical situations where the above process can be applied. Namely, starting with a ring extension $\eta: A \to R$ and consider the category $\cat{A}_R$ of all right $R$-modules which by restriction of scalars are finitely generated and projective as right  $A$-modules, the fibre functor  is then the restriction of the forgetful functor $\eta_{*}: \rmod{R} \to \rmod{A}$, that is, $\eta_{*}: \cat{A}_R \to add(A_A)$. 
The associated $A$-coring  $\Scr{R}(\cat{A}_R)$ is  simply denoted by $\Scr{R}(R)$ or  by $R^{\bcirc}$, and referred to as \emph{the (right)  finite dual of} the $A$-ring $R$.  This in fact establishes a contravariant functor from the category of $A$-rings to the category of $A$-corings.  The basic properties of the finite dual of a ring extension are presented in Section \ref{sec:FDRE}.

The other situation which is somehow dual to the previous one,  is that of an $A$-coring $\coring{C}$, where the small category $\cat{A}$ is taken to be the category $\cat{A}^{\coring{C}}$ of all right $\coring{C}$-comodules whose underlying  right $A$-modules are finitely generated and projective. The fibre functor here is given by the restriction of the  forgetful functor $\Scr{O}: \rcomod{\coring{C}} \to \rmod{A}$, that is, $\Scr{O}: \cat{A}^{\coring{C}} \to add(A_A)$. This situation was studied in \cite{Bruguieres:1994,ElKaoutit/Gomez:2004b}.

\subsection{Review on Hopf algebroids constructed from fibred functors}\label{ssec:HAldfibre}
Starting with a $K$-linear monoidal (essentially small) category  $\cat{A}$ (the tensor product $\otimes$ of $\cat{A}$ is implicitly assumed to be a $K$-linear functor in both factors) with a monoidal faithful $K$-linear functor $\omega: \cat{A} \longrightarrow \bimod{A}{A} $ with image in $add(A_A)$. Then, one can endow, using multiplication induced form the monoidal structure  (see the formula \eqref{Eq:mo} below),  the associated infinite comatrix $A$-coring $\rR(\cat{A})$  of equation \eqref{comatInf} with a structure of $(A\tensor{T_{\II}}A^{\Sscript{o}})$-ring, where $T_{\II}$ is the (commutative) endomorphism $K$-algebra of the identity object $\II$, and $T_{\II} \to A$ is the injective $K$-algebra homomorphism induced by $\omega$. It turns out that $(A,\rR(\cat{A}))$  with this algebra structure is a (left) bialgebroid. If we further assume that $\cat{A}$ is a symmetric rigid monoidal category and that $A$ is a commutative $K$-algebra, then $(A,\rR(\cat{A}))$ has a structure of commutative Hopf algebroid over $T_{\II}$ (compare with  \cite[Section 7]{Bruguieres:1994}). For sake of completeness and for our needs, as well as for  a non expertise reader convenience,  we review in detail the structure maps of this construction.  This detailed construction is to be used later.

Consider the previous situation: $\omega: \cat{A} \longrightarrow add(A_A)$, where $A$ is a not necessarily commutative algebra, and assume that the resulting functor, after  composing $\omega$ with the embedding $ add(A_A) \hookrightarrow \rmod{A}$, factors throughout a (strict) monoidal faithful $K$-linear functor $\cat{A} \to \bimod{A}{A}$\footnote{This factorization condition is not needed when $A$ is a commutative $K$--algebra.}. To avoid some technical problems, we further assume that $\cat{A}$ is a Penrose category, in the sense that each  of the $K$-modules $T_{PQ}$ is a central $T_{\II}$-bimodule over the commutative $K$-algebra $T_{\II}$ (i.e., the left $T_{\II}$-action coincides with right one). This the case when $\cat{A}$ is for instance a braided or symmetric monoidal category.

Let us   first define an unitary and associative multiplication on $\rR(\cat{A})$.  As we have seen before, the infinite comatrix coring $\rR(\cat{A})$ is the  quotient   $A$-coring $\fk{B}(\cat{A})/\fk{J}_{\Sscript{\omega}}$, where
\begin{equation}\label{monoring}
\fk{B}(\cat{A}) = \bigoplus_{P \in \, \cat{A}} P^* \tensor{T_P} P 
\end{equation}
and $\fk{J}_{\Sscript{\omega}}$ is the $K$-submodule spanned by the set of  elements as in Eq.\eqref{Eq:JA}. Given $P, Q \in \mathcal{A}$, and $\varphi \in P^*, \psi \in Q^*$,  we define $(\varphi \star \psi) \in (P\tensor{A}Q)^*$ (remember that here we are denoting $\omega(P):=P$), by 
\begin{equation}\label{Eq:star}
 (\varphi \star \psi) : P\tensor{A}Q \longrightarrow A, \quad \lr{ x\tensor{\peque{A}}y \longmapsto \psi(\varphi(x)y) }.
\end{equation}

We use \eqref{Eq:star} to define componentwise an associative multiplication on $\fk{B}(\cat{A})$  by 
\begin{equation}\label{multicomponent}
\xymatrix@R=5pt{P^* \tensor{T_P} P \tensor{A} Q^* \tensor{T_Q} Q \ar[r] & (Q \tensor{A} P)^* \tensor{T_{Q \tensor{} P}} Q \tensor{A} P  \\
 p^* \tensor{T_P} p \tensor{A} q^* \tensor{T_Q} q \ar@{|->}[r] & q^* \star p^* \tensor{T_{Q \tensor{} P}} q \tensor{A} p
}
\end{equation}
To see that the map \eqref{multicomponent} is well defined, we used the fact that $\cat{A}$ is monoidal and that $\cat{A} \to {}_A\mathsf{Mod}_A$ is a (strict) monoidal functor.
Since $\fk{J}{\Sscript{\omega}}$ is easily checked to be an ideal of  $\fk{B}(\cat{A})$, we get that the following multiplication is well defined :
\begin{equation}\label{Eq:mo}
(\bara{p^*\tensor{T_P}p})\,.\, (\bara{q^*\tensor{T_Q}q}) \,\, =\,\, \bara{(q^* \star p^*) \tensor{T_{Q\tensor{}P}} (q\tensor{\peque{A}}p)}, 
\end{equation}
for every $\bara{p^*\tensor{T_P}p}$ and $\bara{q^*\tensor{T_Q}q}$ in  $\rR(\cat{A})$. A straightforward computation checks that this multiplication is associative. Let us show that there is a unit in $\rR(\cat{A})$ for this multiplication. To this end, observe that there is an injective $K$-algebra map $T_{\II} \to A$ induced by $\omega$. 
Therefore, one can consider the following well defined map 
\begin{equation}\label{Eq:Etao} 
\etaup: \Ae:=A\tensor{T_{\II}}A^{\Sscript{o}} \longrightarrow  \rR(\cat{A}), \quad \Big[ (a'\tensor{} a^o)\longmapsto   (\bara{\Sf{l}_{a'}  \tensor{T_{\II}} a}) \Big],
\end{equation}
where we denote by $\Sf{l}_{a'} : A \to A$ the left multiplication map sending $r \mapsto a'r$ for any $a' \in A$.
The element $\overline{id_A \tensor{T_{\II}} 1_{A}} \in  \rR(\cat{A})$ is clearly the unit for the multiplication \eqref{Eq:mo}.  As before the source  map $\Sf{s}: A \to \rR(\cat{A})$ sends $a \mapsto \etaup(a\tensor{T_{\II}}1^{\Sscript{o}})$ and the target $\Sf{t}$ sends $a \mapsto \etaup(1\tensor{T_{\II}} \ao)$.

The following result could be deduced from \cite[Exemple, pp. 5849]{Bruguieres:1994} and also from  \cite[Theorem 2.2.4]{P.H. Hai:2008}. However,  a direct application of the former or the latter will not lead to an explicit description of the structure maps of the constructed Hopf algebroid. For our needs  and for the reader convenience, we give here an elementary proof of these statements.

\begin{proposition}\label{prop:TBialgd}
Let $A$, $\cat{A}$ and $\omega$ be as above. Then the pair $(A,\rR(\cat{A}))$ admits a structure of left bialgebroid. Assume furthermore that $A$ is a commutative $K$-algebra, $\cat{A}$ is symmetric rigid monoidal $K$-linear and that $\omega$ is a symmetric monoidal faithful $K$-linear functor. Then $(A, \rR(\cat{A}))$ is a commutative Hopf algebroid over $T_{\II}$.
\end{proposition}
\begin{proof}
We need to endow the category of left $\rR(\cat{A})$-modules with a monoidal  structure such that the restriction of scalars functor ${}_{*}\etaup: \lmod{ \rR(\cat{A})} \to \lmod{A^{\Sf{e}}}$ becomes  strict monoidal. 
Take  two left $\rR(\cat{A})$-modules $X$, $Y$. The tensor product $X \tensor{A} Y$ is then a left module over $\rR(\cat{A})$ with the action given by
\begin{equation}\label{Eq:uxyO}
\xymatrix@R=0pt{ \lambda:   \rR(\cat{A}) \otimes X\tensor{A}Y \ar@{->}[r] & X\tensor{A}Y \\  (\bara{p^*\tensor{T_P}p}) \otimes (x\tensor{A}y) \ar@{|->}[r] &  \sum_{\alpha_{P}}(\bara{p^*\tensor{T_P}e_{\alpha,P}})x \tensor{A} (\bara{e_{\alpha,P}^* \tensor{T_P} p})y, }
\end{equation}
where we have  used the comultiplication displayed in Eq.~\eqref{Eq:delta} and the fact that this map lands in an appropriate Sweedler-Takeuchi's product. The restriction of this action to scalars from $A^e$ via $\etaup$ gives the canonical $A$--bimodule structure of $X \tensor{A} Y$, that is, the action
$$
(a'\tensor{T_{\II}}\ao) \,.\, x\tensor{A}y \,\,=\,\, (a'x)\tensor{A}(ya),\quad \text{ for every } a,a' \in A \text{ and }  x \in X, y \in Y.
$$
This in particular shows that the action is unital. 
The associativity property of the action given by Eq.~\eqref{Eq:uxyO}  is derived as follows. Take two elements of the form  $\bara{p^*\tensor{T_P}p}$ and $\bara{q^*\tensor{T_Q}q}$ in  $\rR(\cat{A})$, then 
\begin{eqnarray*}
\lambda\lr{(\bara{p^*\tensor{T_P}p}) \; (\bara{q^*\tensor{T_Q}q}) \tensor{} (x\tensor{A}y) } &=& \lambda\lr{\bara{(q^* \star p^*)\tensor{T_{Q\tensor{}P}}(q\tensor{A}p)} \tensor{} (x\tensor{A}y) } \\ &=& \sum_{\alpha_{P}, \, \beta_{Q}} \bara{(q^* \star p^*)\tensor{T_{Q\tensor{}P}}(e_{\beta, Q}\tensor{A}e_{\alpha, P})}x \tensor{A} 
\bara{(e_{\beta, Q}^* \star e_{\alpha, P}^*)\tensor{T_{Q\tensor{}P}}(q\tensor{A}p)} y
\\ &=& \sum_{\alpha_{P}, \, \beta_{Q}} \lr{ \bara{(p^* \tensor{T_P}e_{\alpha, P})}\; \bara{(q^*\tensor{T_Q}e_{\beta, Q})}x } \tensor{A} 
\lr{ \bara{(e_{\alpha, P}^* \tensor{T_P}p)}\; \bara{(e_{\beta,Q}^*\tensor{T_Q}q)} y} \\
&=& \lambda\lr{\bara{(p^*\tensor{T_{P}}p)}   \tensor{} \lambda\lr{\bara{(p^*\tensor{T_{P}}p)}  \tensor{A^{\Sf{e}}} (x\tensor{A}y)} }.
\end{eqnarray*}

Given $Z$ another left $\rR(\cat{A})$--module, the natural isomorphism $(X\tensor{A}Y)\tensor{A} Z \cong X\tensor{A}(Y\tensor{A}Z)$ is clearly  an isomorphism of left $\rR(\cat{A})$-modules. On the other hand, for any two arrows $f,g$ in $\lmod{\rR(\cat{A})}$, it is easily seen, using the above action,   that $f\tensor{A}g$ is a morphism in $\lmod{\rR(\cat{A})}$. 

Now,  consider the monoidal unit $_{A^{\Sf{e}}}A$ as a left $\rR(\cat{A})$-module via the action 
$$
\bara{p^*\tensor{T_P}p} \; a \,=\, p^*(ap) \,=\, p^*a(p) \, =\, \varepsilon \lr{\bara{p^*\tensor{T_P}p}\,  \Sf{s}(a)} \,=\,  \varepsilon \lr{\bara{p^*\tensor{T_P}p} \,  \Sf{t}(a)},
$$ 
where  $\varepsilon$ is the counit of $\rR(\cat{A})$. For a left $\rR(\cat{A})$-module $X$, let $\iota: A\tensor{A}X \cong X$ sending $a\tensor{A}x \mapsto \Sf{s}(a) x$ be the canonical isomorphism. Using the action of Eq.\eqref{Eq:uxyO}, we have 
\begin{eqnarray*}
\iota\lr{\bara{p^*\tensor{T_P}p} \, (a\tensor{A}x) }&=& \sum_{\alpha_{P}} \iota\lr{ p^*(ae_{\alpha,P}) \tensor{A} \bara {(e_{\alpha,P}^* \tensor{T_P}p) }x } \\ &=& \sum_{\alpha_{P}} \bara{ p^*(ae_{\alpha,P})e_{\alpha,P}  \tensor{T_P} p} \, x \\ &=& \bara{p^*a \tensor{T_P} p} \, x \\ &=& \bara{p^* \tensor{T_P} p} \, \Sf{s}^{\bcirc}(a)  x \\ &=& \bara{p^* \tensor{T_P} p} \;\; \iota(a\tensor{A}x ),
\end{eqnarray*}
which shows that $\iota$ is left $ \rR(\cat{A})$-linear.  Similarly, one can show that the isomorphism $X\tensor{A}A \cong X$ is also left $ \rR(\cat{A})$--linear. Thus, ${}_{\rR(\cat{A})}A$ is the monoidal unit for the tensor product in $\lmod{\rR(\cat{A})}$.

Summarizing, we obtain that  the category of left $\rR(\cat{A})$-modules is  a monoidal category with unit the left module $_{\rR(\cat{A})}A$, and the restriction of scalars functor ${}_{*}\etaup: \lmod{\rR(\cat{A})} \to \lmod{\Ae}$ is a strict monoidal. Therefore, $(A,  \rR(\cat{A}))$ is left bialgebroid. 

Regarding the commutative case, we need first to introduce some notations: let us denote  by $\tau_{P,Q}: P\tensor{} Q \to Q\tensor{} P$ the symmetry of $\cat{A}$ and by $(-)^{\Sscript{\vee}}: \cat{A} \to \cat{A}$ the $K$-linear contravariant functor which sends any object $P \in \cat{A}$ to its dual $\dP$ and any morphism $f:P \to Q$ to it dual $f^{\Sscript{\vee}}: \dQ \to \dP$.  Under assumption, there is a natural isomorphism $\gammaup_{-}:  (-)^{*} \circ \omega \to \omega \circ   (-)^{\Sscript{\vee}} $, which leads to a natural isomorphism $\phiup_{P}: \omega(P) \to \omega(\dP)^{*}$. Using the action of \eqref{omega1} and \eqref{omega2}, the naturality of bothe $\gammaup$ and $\phiup$, reads
\begin{gather}
\phiup_{Q}(\omega(f)\, p) \,\,=\,\, \phiup_{P}(p) \,  \omega(f^{\Sscript{\vee}})  \label{Eq:GP1}\\ 
\gammaup_{P}(\varphi \,  \omega(f)) \,\,=\,\, \omega(f^{\Sscript{\vee}}) \,  \gammaup_{Q}(\varphi), \label{Eq:GP2}
\end{gather} 
for every morphism $f: P \to Q$ in $\cA$ ,   $p \in \omega(P)$ and $\varphi \in \omega(Q)^*$. 
On the other hand $\gammaup$ and $\phiup$ render commutative the following diagrams
\begin{equation}\label{Eq:PG}
\begin{gathered}
\xymatrix{   (\omega(\dP)^{*})^{*} \ar@{->}^-{\phiup^{*}}[r] & \omega(P)^{*}  \ar@{->}^-{\gammaup}[r] & \omega(\dP)  \\ \omega(\dP ),   \ar@{->}^-{\psiup}[u]  \ar@{=}^-{}[rru] & &  } \qquad \xymatrix{   (\omega(P)^{*})^{*}  & \ar@{->}_-{\gammaup^{*}}[l] \omega(\dP)^{*}   & \ar@{->}_-{\phiup}[l] \omega(P)  \\ \omega(P ),   \ar@{->}^-{\psiup}[u]  \ar@{=}^-{}[rru] & &  }
\end{gathered}
\end{equation}
where the vertical map is the canonical isomorphism of $A$-modules, sending $q \mapsto\left[q^* \mapsto q^*(q)\right]$. The natural isomorphism $\gammaup$ rends also the following diagrams
\begin{equation}\label{Eq:evpluadb}
\begin{gathered}
\xymatrix@R=20pt{   & \omega(P)^* \tensor{A} \omega(P) \ar@{->}^-{ev_{\omega(P)}}[r]  \ar@{->}_-{\gammaup \tensor{A} \omega(P)}[ld]  & A \\ \omega(\dP)\tensor{A} \omega(P) \ar@{->}_-{\cong}[rd]  & &   \\  & \omega\big( \dP \tensor{}P \big)  \ar@{->}^-{\omega(ev_{P})}[r]   &  \omega(\II) \ar@{->}^-{\cong}[uu] } \;   \xymatrix@R=20pt{    A   \ar@{->}^-{db_{\omega(P)}}[r]    &  \omega(P) \tensor{A} \omega(P)^*  \ar@{->}^-{ \omega(P) \tensor{A} \gammaup }[rd] & \\ & & \omega(\dP)\tensor{A} \omega(P)      \\ \omega(\II) \ar@{->}^-{\cong}[uu] \ar@{->}^-{\omega(db_{P})}[r]  & \omega\big( \dP \tensor{}P \big)  \ar@{->}_-{\cong}[ru]    &   }
\end{gathered}
\end{equation}
commutative, where $ev_{-}$ and $db_{-}$ symbolize  the evaluation and the dual basis morphisms attached to a dualizable object  in a rigid monoidal category, respectively. 

Furthermore, $\phiup$ is compatible with the tensor product, that is, for any pair of objects $P, Q \in \cat{A}$, we have a commutative diagram:
\begin{equation}\label{Eq:Phi}
\begin{gathered}
\xymatrix@C=35pt{\omega(P\tensor{}Q) \ar@{->}^-{\cong}[r]   \ar@/_1.5pc/@{->}_-{\phiup_{P\tensor{}Q}}[rrrdd]   &  \omega(P)\tensor{A} \omega(Q)  \ar@{->}^-{\phiup_{P}\tensor{A}\phiup_{Q}}[r] & \omega(\dP)^{*}\tensor{A} \omega(\dQ)^{*}  \ar@{->}^-{\cong}[r] & (\omega(\dQ) \tensor{A} \omega(\dP))^{*} \ar@{->}^-{\cong}[d]    \\  & & & (\omega(\dQ\tensor{} \dP))^{*} \ar@{->}^-{\omega(\iota_{P ,\, Q})^{*}}[d]  \\ & & & \omega\big((P\tensor{}Q)^{\Sscript{\vee}}\big)^{*},   }
\end{gathered}
\end{equation}
where $\iota_{Q,\, P}: (Q\tensor{}P)^{\Sscript{\vee}}  \cong \dP\tensor{}\dQ   $
is the canonical natural isomorphism in the category  $\cat{A}$.  The natural transformation $\gammaup$ is also compatible with the tensor product, that is, we have another commutative diagram:
\begin{equation}\label{Eq:Gamma}
\begin{gathered}
\xymatrix@C=45pt{\omega(P\tensor{}Q) \ar@{->}^-{\gammaup_{P\tensor{}Q}}[rr]  \ar@{->}_-{\cong}[d]    &   & \omega\big((P\tensor{}Q)^{\Sscript{\vee}}\big)     \\  \omega(P)^{*}\tensor{A}\omega(Q)^{*} \ar@{->}^-{\gammaup_{P}\tensor{A}\gammaup_{Q}}[r] & \omega(\dP) \tensor{A} \omega(\dQ) \ar@{->}^-{\cong}[r] & \omega(\dQ\tensor{} \dP)  \ar@{->}_-{\omega(\iota_{P,\, Q}^{-1})}[u]    }
\end{gathered}
\end{equation}

On elements, the commutativity of both diagrams of equations \eqref{Eq:Phi} and \eqref{Eq:Gamma}, are expressed by the following first two equalities:
\begin{eqnarray}
 \iota_{Q,\, P} \,\,  \gammaup_{Q\tensor{}P}\big( q^* \star p^*\big)  & =&  \gammaup_{P}(p^*) \tensor{A} \gammaup_{Q}(q^*)  , \label{Eq:iota} \\
\phiup_{Q\tensor{}P}\big( q \tensor{A} p\big)  & =&     \Big(  \phiup_{P}(p) \star \phiup_{Q}(q) \Big) \,\,  \iota_{Q,\, P}, \label{Eq:phi} \\ 
 (q^*\star p^*) \, \,  \tau_{P,\, Q} &=& p^*\star q^*,  \label{Eq:tau}
\end{eqnarray}
for every  $P, Q \in \cat{A}$, $p \in \omega(P)$, $q \in \omega(Q)$  and $p^* \in \omega(P)^*, q^* \in \omega(Q)^*$, where we have used the actions of equations \eqref{omega1} and \eqref{omega2}.

Take two generic elements $(\bara{p^*\tensor{T_P}p}) , (\bara{q^*\tensor{T_Q} q})  \in \rR(\cat{A})$ and using equality \eqref{Eq:tau}, we compute
\begin{eqnarray*}
(\bara{p^*\tensor{T_P}p}) \, (\bara{q^*\tensor{T_Q} q})  &=& \bara{(q^*\star p^*) \tensor{T_{Q\tensor{}P}} (q\tensor{A}p)} \\ &=& \bara{(q^*\star p^*) \tensor{T_{Q\tensor{}P}} \tau_{P,\, Q} (p\tensor{A}q)}, \,\, \tau_{P,\, Q} \text{ is an arrow in }\cat{A} \\ &=& \bara{(q^*\star p^*)\tau_{P,\, Q} \tensor{T_{P\tensor{}Q}} (p\tensor{A}q) }\\ &=&  \bara{(p^*\star q^*) \tensor{T_{P\tensor{}Q}} (p\tensor{A}q)} \\ &=& (\bara{q^*\tensor{T_Q} q}) \, (\bara{p^*\tensor{T_P}p}),
\end{eqnarray*}
which means that the multiplication \eqref{Eq:mo} is commutative, and so $ \rR(\cat{A})$ is a commutative $A\tensor{T_{\II}}A$-algebra.  The compatibility of the comultiplication  of Eq.\eqref{Eq:delta} and the counit of Eq.\eqref{Eq:counit} with the multiplication of Eq.\eqref{Eq:mo}, are routine computation using the dual basis properties. 

The following map 
\begin{equation}\label{Eq:So}
\Scr{S}:   \rR(\cat{A}) \longrightarrow  \rR(\cat{A}), \qquad \Big(\bara{p^*\tensor{T_P} p} \longmapsto \bara{\phiup(p) \tensor{T_{\dP}} \gammaup(p^*)}\Big), 
\end{equation}
which is well defined thanks to the naturality of $\gammaup_{-}$ and $\phiup_{-}$, is our candidate for the antipode map of  the commutative bialgebroid $(A,\rR(\cat{A}))$.
Let us check that $\Scr{S}$ transforms the source to the target and vice-versa, and that it is an algebra map as well. So, for two elements $a, b \in A$, we have 
$$\Scr{S}(\etaup(a\tensor{T_{\II}}b)) \,=\, \Scr{S}(\bara{\Sf{l}_a\tensor{T_{\II}}b}) \,=\, \bara{\phiup(b) \tensor{T_{\II^{\Sscript{\vee}}}} \gammaup(\Sf{l}_a)} \,\overset{\eqref{Eq:PG}}{=}\, \bara{\Sf{l}_b\tensor{T_{\II}}a}\,=\, \etaup(b\tensor{T_{\II}}a).
$$ 
Hence $\Scr{S} \circ \Sf{s}\,=\, \Sf{t}$ and $\Scr{S} \circ \Sf{t}\,=\, \Sf{s}$. The fact that $\Scr{S}$ is a multiplicative map follows from the following computation: For any two generic elements $(\bara{p^*\tensor{T_P} p}) , (\bara{q^*\tensor{T_Q}q}) \in \rR(\cat{A})$, we have
\begin{eqnarray*}
\Scr{S}\Big((\bara{p^*\tensor{T_P} p}) \, (\bara{q^*\tensor{T_Q}q})\Big) &=& \Scr{S}\Big(\bara{(q^*\star p^*) \tensor{T_{Q\tensor{}P}}(q\tensor{A}p)} \Big) \\ &=& \bara{\phiup(q\tensor{A}p) \tensor{T_{(Q\tensor{}P)^{\Sscript{\vee}}}} \gammaup(q^*\star p^*) }\\ &\overset{\eqref{Eq:iota}}{=}&  \bara{\phiup(q\tensor{A}p) \tensor{T_{(Q\tensor{}P)^{\Sscript{\vee}}}}  \iota_{Q,\, P}^{-1}  \, \gammaup(p^*)\tensor{A} \gammaup(q^*)} \\ &=&   \bara{\phiup(q\tensor{A}p)  \, \iota_{Q,\, P}^{-1} \tensor{T_{\dP\tensor{}\dQ}} \gammaup(p^*)\tensor{A} \gammaup(q^*)} \\ &\overset{\eqref{Eq:phi}}{=}& \bara{ \phiup(p)\star \phiup(q) \tensor{T_{\dP\tensor{}\dQ}} \gammaup(p^*)\tensor{A} \gammaup(q^*)} \\ &\overset{\eqref{Eq:mo}}{=}& (\bara{\phiup(q) \tensor{T_{\dQ}} \gammaup(q^*)}) \, \, (\bara{\phiup(p)\tensor{T_{\dP}} \gammaup(p^*)}) \\ &=&  \Scr{S}(\bara{q^*\tensor{T_Q}q}) \, \, \Scr{S}(\bara{p^*\tensor{T_P}p}),
\end{eqnarray*}
which shows that $\Scr{S}$ is a morphism of algebras.

We still need to check that, for every element $h \in  \rR(\cat{A})$, we  have
$$ h_1 \Scr{S}(h_2)\,\,=\,\, \etaup(\varepsilon(h)\tensor{}1), \quad \Scr{S}(h_1)h_2\,\,=\,\, \etaup(1\tensor{}\varepsilon(h)).$$  Take $h \in \rR(\cat{A})$ of the form $h=\bara{p^*\tensor{T_P}p}$, where $p^* \in \omega(P)^*$, $p \in \omega(P)$, for some object $P \in \cat{A}$, then we have  (summation understood)
\begin{eqnarray*}
h_1\Scr{S}(h_2) &=& (\bara{p^*\tensor{T_P}e_{\alpha, P}}) \Scr{S}(\bara{e_{\alpha,P}^*\tensor{T_P}p}) \\ &=&  (\bara{p^*\tensor{T_P}e_{\alpha, P}}) \, (\bara{\phiup(p)\tensor{T_{\dP}}\gammaup(e_{\alpha,P}^*)}) 
\\ &=&   \bara{(p^*\star \phiup(p))\tensor{T_{P\tensor{}\dP}}(e_{\alpha,P}\tensor{A}\gammaup(e_{\alpha,P}^*))}
\\ &\overset{\eqref{Eq:evpluadb}}{=}& \bara{(p^*\star \phiup(p))\tensor{T_{P\tensor{}\dP}}\omega(db_P)(1_A)}
\\ &=&   \bara{(p^*\star \phiup(p))\omega(db_P)  \tensor{T_{\II}}1_A}
\\ &=&   \bara{\big( \omega(db_P)^* \circ (p^*\tensor{A} \phiup(p)) \big)  \tensor{T_{\II}}1_A}
\\ &\overset{\eqref{Eq:evpluadb}}{=}& \bara{\big( db_{\omega(P)}^* \circ (\omega(P)^* \tensor{A}\gammaup^*)  \circ (p^*\tensor{A} \phiup(p)) \big)  \tensor{T_{\II}}1_A}
\\ &=& \bara{\big( db_{\omega(P)}^*  \circ (p^*\tensor{A} \gammaup^*\phiup(p)) \big)  \tensor{T_{\II}}1_A}
\\ &\overset{\eqref{Eq:PG}}{=}&  \bara{\big( db_{\omega(P)}^*  \circ (p^*\tensor{A} \psiup(p)) \big)  \tensor{T_{\II}}1_A}
\\ &=&  \bara{ (p^*\tensor{A} \psiup(p)) \big( db_{\omega(P)}(1_A) \big)  \tensor{T_{\II}}1_A}
\\  &=& \bara{\Sf{l}_{p^*(p)} \tensor{T_{\II}} 1_A}\,\,=\,\, \etaup(\varepsilon(h)\tensor{}1_A),
\end{eqnarray*}
which shows the first desired equality. The second one is similarly obtained (or can be obtained from the  equality $\Scr{S}^2=id$ which is shown below). Lastly, since the natural isomorphisms $\fk{d}_P: P \to (\dP)^{\Sscript{\vee}}$ in  $\cat{A}$, satisfy the following equalities
\begin{equation}\label{Eq:fkd}
\gammaup_{\dP} \circ \phiup_P \,\,=\,\, \omega(\fk{d}_P)\qquad \omega(\fk{d}_P)^* \circ \phiup_{\dP} \circ \gammaup_P \,\, =\,\, id_{\omega(P)^*}, 
\end{equation}
it is not difficult to show that $\Scr{S}^2=id$, and this finishes the proof.
\end{proof}

\begin{remark}\label{rem:generalcase}
Given two symmetric rigid monoidal $K$-linear categories  with symmetric monoidal faithful $K$-linear functors: $\omega: \cat{A} \to add(A)$ and $\omega':\cat{A}'\to add(A)$, and  $\cat{F}: \cat{A} \to \cat{A}'$ a monoidal $K$-linear functor such that 
$$
\xymatrix@R=15pt@C=40pt{  \cat{A} \ar@{->}^-{\cat{F}}[rr]  \ar@{->}_-{\omega}[rd] & & \cat{A}'  \ar@{->}^-{\omega'}[ld] \\ & add(A) & }
$$
is a commutative diagram, then there exists a morphism $\phiup=:\rR(\cat{F}): (A,\rR(\cat{A})) \to (A,\rR(\cat{A}'))$ of Hopf algebroids making commutative the following diagram of functors:
$$
\xymatrix@R=15pt{   & \cat{A}^{\Sscript{\rR(\cat{A})}} \ar@{->}^-{\phi_*}[rrr] \ar@/_1pc/@{->}^-{\oO}[rddd]  & & & (\cat{A}')^{\Sscript{\rR(\cat{A}')}} \ar@/^1.5pc/@{->}^-{\oO}[llddd]  \\  \cat{A} \ar@{->}^-{\cat{F}}[rrr] \ar@{->}^-{\omegaup}[ru] \ar@/_1pc/@{->}^-{\omega}[rrdd] & & & \cat{A}' \ar@{->}^-{\omegaup'}[ru] \ar@/^1pc/@{->}_-{\omega}[ldd] & \\ & & & &   \\ & & add(A) & & }
$$
where $\oO$ is the forgetful functor and $\phiup_{*}$ is the restriction of the induced functor $\phiup_{*}: \rcomod{\rR(\cat{A})} \to \rcomod{\rR(\cat{A}')}$ sending any right $\rR(\cat{A})$-comodule $(M, \varrho_M)$ to the right $\rR(\cat{A}')$-comodule $$\xymatrix@C=30pt{M \ar@{->}^-{\varrho_M}[r]& M\tensor{A}\rR(\cat{A}) \ar@{->}^-{M\tensor{A}\phiup}[r]  & M \tensor{A}\rR(\cat{A}'), }$$ and acting obviously on morphisms. 
\end{remark}

\subsection{Reconstruction and Galois corings}\label{ssec:Galois}
It is well known that any coalgebra over a field can be reconstructed from its category of finite-dimensional right comodules. That is, it is isomorphic to an infinite comatrix coalgebra. Over a general base ring, one only obtains an $A$-coring morphism which is not always an isomorphism. Precisely, let $\Scr{R}(\cat{A}^{\coring{C}})$ be the $A$-coring associated to the fibre functor $\Scr{O}: \cat{A}^{\coring{C}} \to add(A_A)$ for some $A$-coring $\coring{C}$. Then there is a homomorphism of $A$-corings known as the \emph{canonical map}:
\begin{equation}\label{Eq:can}
\Sf{can}_{\cat{A}^{\coring{C}}} : \Scr{R}(\cat{A}^{\coring{C}})  \longrightarrow \coring{C}, \quad \LR{p^*\tensor{T_P} p + \fk{J}_{\cat{A}} \longmapsto (p^*\tensor{A}\coring{C}) \circ \varrho_P(p)}.
\end{equation}
Obviously any $K$-linear functor $\cat{A} \to \cat{A}^{\coring{C}}$ induces an analogue morphisms of corings $\Sf{can}_{\cat{A}}$.

Following the terminology used in \cite{ElKaoutit/Gomez:2004b,Gomez/Vercruysse:2007}:

\begin{definition}\label{def:Galois}
An $A$-coring $\coring{C}$ is said to be an $\cat{A}^{\coring{C}}$-\emph{Galois coring}   (or $\Sigma=\oplus_{p \in\cat{A}^{\coring{C}} }P$ is a right  \emph{Galois comodule}) provided $\Sf{can}_{\cat{A}^{\coring{C}}}$ is an isomorphism of $A$-corings.  
\end{definition}

This is the case when $\rcomod{\coring{C}}$ is an abelian category having $\cat{A}^{\coring{C}}$ as a set of generators \cite[Theorem 4.8]{ElKaoutit/Gomez:2004b}, which is exactly the situation in the aforementioned coalgebra case (see \cite[Proposition 4.13]{Deligne:1990} for the  case of corings over fields).  The case of corings underlying certain commutative Hopf algebroids is of special interest as the following example shows.

\begin{example}[\emph{Geometrically transitive Hopf algebroids} \cite{Bruguieres:1994, ElKaoutit:2015}]\label{exam:GT}
Roughly speaking, a commutative flat Hopf algebroid $(A,\cH)$ over a base field $K$, is said to be \emph{geometrically transitive} (GT for short), provided that the  algebra map $\etaup=\Sf{s}\tensor{K}\Sf{t}: A\tensor{K}A \to \cH$ is a faithfully flat extension. The fact that $K$ should be a field is a crucial condition here.  It turns out that GT Hopf algebroid $(A,\cH)$ is reconstructed from  its $A$-profinite comodules, that is, from the category $\cA^{\cH}$ of the underlying $A$-coring $\cH$, as (the skeleton of) this  category form a `set' of generators, see \cite[Corollary A (2)]{ElKaoutit:2015}. Indeed, every $\cH$-comodule is an inductive limit of comodules from $\cA^{\cH}$. Besides, any comodule is faithfully flat as an $A$-module  \cite[Lemma 4.3 (b)]{ElKaoutit:2015}. In summary,  by applying  \cite[Theorem 4.8]{ElKaoutit/Gomez:2004b}, the canonical map of Eq.~ \eqref{Eq:can} is bijective and $\cH$ is a Galois $A$-coring.
\end{example}

\subsection{Principal bi-bundles attached to two different fibre functors}\label{ssec:PB}
Apart from its own general interest, the material of this section will be used to clarify the construction of Picard-Vessiot extensions (over the affine line) in Section \ref{sec:Weyl}.  In this subsection all algebras are  commutative $K$-algebras over a base field $K$ and are assumed to have $K$-points. The category of all  commutative algebras over a commutative algebra $A$, is denoted by $\Alg{A}$.

We keep the notation of subsection \ref{ssec:HAldfibre} and fix a symmetric rigid monidal $K$-linear category $\cA$.  For our needs it is convenient to assume further that the endomorphism algebra of the identity object of $\cA$ is isomorphic to the base field, that is, $T_{\II} \cong K$.   For a given (non trivial) monoidal symmetric $K$-linear faithful functor $\omega: \cA \to add(A)$ (i.e., ~ a fibre functor),  instead of denoting by $\rR(\cA)$  the resulting Hopf algebroid from Proposition \ref{prop:TBialgd}, we will use the notation $\RAA{\omega}:=\rR(\cA)$ in order to specify which fibre functor we are using and over which base algebra we are working.  

Given an algebra map $\xi: A \to C$, we denote by $\omega\tensor{A} {}_{\Sscript{\xi}}C: \cA \to add(C)$ the extended fibre functor, which sends any object $P \in \cA$ to the finitely generated and projective $C$-module $\omega(P) \tensor{A}{}_{\Sscript{\xi}}C$.

Assume we are given two fibre functors $\omega_i: \cA \to add(A_i)$, $i=1,2$. 
If $R$ is an $A_1\tensor{}A_2$-algebra, that is, we have an extension $s\tensor{}t: A_1\tensor{}A_2 \to R$, then we have two extended fibre functors, namely, $\omega_1\tensor{A_1} {}_{\Sscript{s}}R$ and $\omega_2\tensor{A_2} {}_{\Sscript{t}}R$.
In this way, we can define the following $(A_1, A_2)$-bimodule:
\begin{equation}\label{Eq:L12}
\Rdos{\omega_1}{\omega_2} : = \frac{\bigoplus_{P \, \in \, \cA}\omega_1(P)^*\tensor{T_P}\omega_2(P)}{\cJ_{\Sscript{\omega_1,\, \omega_2}}}
\end{equation}
where $\cJ_{\Sscript{\omega_1,\, \omega_2}}$ is the subbimodule generated by the set 
\begin{equation}\label{Eq:JAA}
\left\{\underset{}{} q^* \tensor{T_Q} \omega_2(t)(p) - (q^* \circ \omega_1(t)) \tensor{T_P}p | \;\,  q^* \in \omega_1(Q)^*,\, p\in \omega_2(P),\, t \in T_{PQ},\, P,Q \in \cat{A} \right\}.
\end{equation}
Then there are two bimodule  maps 
\begin{equation}\label{Eq:AB}
\bo{\alpha}: A_1 \longrightarrow \Rdos{\omega_1}{\omega_2},   \Big(  a_1 \longmapsto \bara{\Sf{l}_{a_1} \tensor{K} 1_{A_2}} \Big) \text{ and }\,  
\bo{\beta}: A_2 \longrightarrow \Rdos{\omega_1}{\omega_2},   \Big(  a_2 \longmapsto \bara{1_{A_1} \tensor{K} {\Sf{l}_{a_2}} } \Big).
\end{equation}

Our next aim is to show that the triple $(\Rdos{\omega_1}{\omega_2}, \bo{\alpha}, \bo{\beta})$ admits a structure of $(\RAAu{\omega_1}, \RAAd{\omega_2})$-bicomodule algebra which becomes a principal bibundle (see \cite[Section 4]{Kaoutit/Kowalzig:2014} for the pertinent definition) when the involved Hopf algebroids are geometrically transitive (see Example \ref{exam:GT} and \cite[Theorem A]{ElKaoutit:2015} for more details).

The commutative algebra structure of $\Rdos{\omega_1}{\omega_2}$ is defined in similar way as in Eq.~\eqref{Eq:mo}. That is, for a given two object $P, Q \in \cA$ and elements $p_1^* \in \omega_1(P)^*$, $q_1^* \in \omega_1(Q)^*$ and  $p_2 \in \omega_2(P)$, $q_2 \in \omega_2(Q)$, we have that 
\begin{equation}\label{Eq:modos}
(\bara{p_1^*\tensor{T_P}p_2})\,.\, (\bara{q_1^*\tensor{T_Q}q_2}) \,\, :=\,\, \bara{(p_1^* \star q_1^*) \tensor{T_{P\tensor{}Q}} (p_2\tensor{\peque{A}}q_2)}, 
\end{equation}
is a well defined multiplication, such that the map $\bo{\alpha}\tensor{}\bo{\beta}: A_1\tensor{K}A_2 \to \Rdos{\omega_1}{\omega_2}$ is a $K$-algebra  map. Clearly, we have $\Rdos{\omega_1}{\omega_2}=\RAA{\omega}$, whenever $A_1=A_2=A$ and $\omega_1=\omega_2=\omega$. In  case we only assume that $A_1=A_2=A$,  we denote by $\langle \bo{\alpha}- \bo{\beta} \rangle$ the ideal of $\rR_{\Sscript{A\tensor{}A}}(\omega_1,\omega_2)$ generated by the set of elements 
$\big\{  \bo{\alpha}(a)- \bo{\beta}(a)|\,\,  a \in A  \big\}$, and by 
\begin{equation}\label{Eq:Ra}
\rR_{\Sscript{A}}(\omega_1,\omega_2) := \frac{\rR_{\Sscript{A\tensor{}A}}(\omega_1,\omega_2)}{\langle \bo{\alpha}- \bo{\beta} \rangle}
\end{equation}
its quotient $A$-algebra with extension denoted by  $\biota: A \to \rR_{\Sscript{A}}(\omega_1,\omega_2)$.   The equivalence class of a generic element $\bara{p^*\tensor{T_P}p}$ in this quotient algebra $\rR_{\Sscript{A}}(\omega_1,\omega_2)$ is denoted by $[\bara{p^*\tensor{T_P}p}]$.  When $\omega_1=\omega_2$, we denote $\rR_{\Sscript{A}}(\omega,\omega):=\RA{\omega}$ (here $\balpha$ and $\bbeta$ are the source and the target  maps, respectively, of this Hopf algebroid). Specifically,  $\RA{\omega}$  inherits canonically a structure of  commutative Hopf $A$-algebra, that is, a commutative Hopf algebroid with source is equal to the target.

In the general situation we will denote by $\IsomF{2}{1}$ the functor
\begin{equation}\label{Eq:IsomF}
\IsomF{2}{1}: {\rm Alg}_{\Sscript{A_1\tensor{}A_2}} \longrightarrow {\rm Sets}, \qquad \Big(  R \longrightarrow {\rm Isom}^{\Sscript{\otimes}} (\omega_2\tensor{A_2} \tR, \omega_1\tensor{A_1} \sR)  \Big)
\end{equation}
from the category of $(A_1\tensor{}A_2)$-algebras to the category of sets, which sends any $(A_1\tensor{}A_2)$-algebra $R$ to the set of all tensorial natural transformations from the functor $\omega_1\tensor{A_1} \sR$ to $\omega_2\tensor{A_2}\tR$ (recall that the category $\cA$ is assumed to be rigid, so that each of these natural transformations is already a natural isomorphism).  We will employ the same notations for the case $A_1=A_2=A$, that is, we will denote by $\Isomf{2}{1}: \Alg{A} \to {\rm Sets}$ the functor which sends any $A$-algebra $C$ to the set  ${\rm Isom}^{\Sscript{\otimes}} (\omega_2\tensor{A} C, \omega_1\tensor{A} C) $.

\begin{proposition}\label{prop:IsomF}
Keep the above notations. Then the functor $\IsomF{2}{1}$ of equation \eqref{Eq:IsomF} is represented by the $(A_1\tensor{}A_2)$-algebra $\Rdos{\omega_1}{\omega_2}$. Assume that $A_1=A_2=A$, then up to natural isomorphisms,  there is a commutative diagram 
$$
\xymatrix@C=40pt{  \Alg{A} \ar@{->}_{{\underline{\rm Isom}^{\Sscript{\otimes}}}{}_{\Sscript{A}}(\omega_2,\, \omega_1)}[dr] \ar@{->}_{}[rr]  & & \Alg{A\tensor{}A} \ar@{->}^-{{\underline{\rm Isom}^{\Sscript{\otimes}}}{}_{\Sscript{A\tensor{}A}}(\omega_2,\, \omega_1) }[dl]  \\ & {\rm Sets} &  }
$$
where the horizontal functor is the canonical one. Furthermore the functor $\Isomf{2}{1}$ is represented by the quotient $A$-algebra $\rR_{\Sscript{A}}(\omega_1,\omega_2)$ of equation \eqref{Eq:Ra}.
\end{proposition}
\begin{proof}
The first statement of the proposition is \cite[Proposition 6.6]{Deligne:1990}. The rest of the proof is not difficult  and left to the reader. 
\end{proof}

Recall that a pair $(\cP,\alpha)$ is said to be a \emph{left comodule algebra} over a commutative Hopf algebroid $(A, \cH)$, when $\alpha: A \to \cP$ is an extension of algebras and $P$ is endowed  with a structure of left $\cH$-comodule whose underlying left $A$-module is ${}_{\Sscript{\alpha}}\cP$ such that the coaction $\bo{\lambda}: \cP \to \cH_{\Sscript{s}} \tensor{A} {}_{\Sscript{\alpha}}\cP$ is an $A$-algebra map.  Right comodule algebras are similarly defined, and bicomodule algebras are naturally introduced. Thus an $(\cH, \cK)$-bicomodule algebra over two commutative Hopf algebroids $(A,\cH)$ and $(B,\cK)$,  is a triple $(\cP,\alpha, \beta)$ such that $(\cP, \alpha)$ is a left $\cH$-comodule algebra, $(\cP, \beta)$ is a right $\cK$-comodule algebra and  the left $\cH$-coaction is a morphism of right $\cK$-comodules (or equivalently the right $\cK$-coaction is left $\cH$-comodule morphism).

\begin{definition}(\cite[Defintion 4.1]{Kaoutit/Kowalzig:2014})\label{def:PB}
A \emph{left principal $(\cH, \cK)$-bundle} $(\cP, \alpha, \beta)$ 
for two Hopf algebroids $(A, \cH)$ and $(B, \cK)$ 
is an $(\cH,\cK)$-bicomodule algebra, that is,  $\cP$  is equipped with a left $\cH$-comodule algebra and a right $\cK$-comodule algebra structures with respect to the algebra maps $\alpha: A \to \cP$ resp.\ $\beta: B \to \cP$ such that  
\begin{enumerate}
\item $\beta$ is a faithfully flat extension; 
\item the canonical map
\begin{equation}\label{Eq:canl}
\can{}^{l}: \cP \tensor{B} \cP \to \cH\tensor{A}\cP, \quad p\tensor{B}p' \mapsto p_{(-1)} \tensor{A} p_{(0)} p'
\end{equation}
is bijective.
\end{enumerate}
A triple $(P, \alpha, \beta)$ which only satisfies condition (2) is called \emph{left pseudo $(\cH,\cK)$-bundle}.  \emph{Right (pseudo) principal bundles} are similarly defined and we use the notation $\can{}^{r}$ for the corresponding canonical map. For instance, the opposite right principal bundle of a given left principal bundle $(\cP,\alpha, \beta)$, is the bundle $(\cP^{\Sscript{op}}, \beta, \alpha)$ where the algebra $\cP^{\Sscript{op}}=\cP$ and the coactions are switched (interchanging the source and the target maps). A \emph{(pseudo) principal bibundle} is a left and right (pseudo) principal bundle. That is, for pseudo bibundle, both $\can{}^{l}$ and $\can{}^{r}$ are required to be bijective.  If, furthermore, both $\alpha$ and $\beta$ are faithfully flat, then the bibundle is principal.  For a given flat Hopf algebroid $(A,\cH)$ it is easily shown that the triple $(\cH, s, t)$ is a principal $(\cH, \cH)$-bibundle, see \cite[\S 4]{Kaoutit/Kowalzig:2014} for more examples and details.
\end{definition}

Now  we come back to our situation. Consider as above  the algebras $\Romaa{i}{j}$, for $i=1,2$. For an object $P \in \cA$ we fix a dual basis for $\omega_i(P) \in add(A_i)$ using the notation $\{e_{\alpha_{i,\, P}}, e^*_{\alpha_{i,\, P}}\}_{\alpha_{i,\, P}}$, $i=1,2$. We also make use of the notation $p_i \in \omega_i(P)$ and $p^*_i \in \omega_i(P)^*$, for any object $P \in\cA$.   In this way we have the following well defined  maps:
\begin{gather}\label{Eq:LRA}
\xymatrix@R=0pt{  \Romaa{1}{2} \ar@{->}^-{\bo{\lambda}}[r]  & \Romaa{1}{2} \tensor{A_2} \Roma{2} \\   \Big( \bara{p^*_1 \tensor{T_P} p_2} \ar@{|->}[r] &  \sum_{\alpha_{2,\, P}} \bara{p^*_1\tensor{T_p}e_{\alpha_{2,\, P}}} \, \tensor{A_2} \, \bara{e^*_{\alpha_{2,\, P}}\tensor{T_P} p_2} \Big),  }
\end{gather}
where the left $A_2$-action on $\Roma{2}$ is given by the algebra map $\balpha$ of equation \eqref{Eq:AB}, and 
\begin{gather}\label{Eq:RRA}
\xymatrix@R=0pt{  \Romaa{1}{2} \ar@{->}^-{\bo{\varrho}}[r]  & \Roma{1} \tensor{A_1} \Romaa{1}{2}  \\   \Big( \bara{p^*_1 \tensor{T_P} p_2} \ar@{|->}[r] &  \sum_{\alpha_{1,\, P}} \bara{p^*_1\tensor{T_p}e_{\alpha_{1,\, P}}} \, \tensor{A_2} \, \bara{e^*_{\alpha_{1,\, P}}\tensor{T_P} p_2} \Big),  }
\end{gather}
where the right $A_1$-action on $\Roma{1}$ is given by the algebra map $\bbeta$ of equation \eqref{Eq:AB}. 

The following Lemma and its subsequent Corollary are inspired from  Proposition \ref{prop:IsomF}. 

\begin{lemma}\label{lema:Rcomalg}
Consider $(A_i, \Roma{i})$, for $i=1,2$, as commutative Hopf algebroids with structure maps given as in subsection \ref{ssec:HAldfibre}. Then the triple $(\Romaa{1}{2}, \balpha, \bbeta)$ is an $(\Roma{1},\Roma{2})$-bicomodule algebra with left and right coactions given by equations  \eqref{Eq:LRA} and \eqref{Eq:RRA}, respectively. Furthermore,    $\Romaa{1}{2}$ is a pseudo $(\Roma{1},\Roma{2})$-bibundle.
\end{lemma}
\begin{proof}
The first statement is a straightforward verification.  As for the last one, recall the natural transformations $\gammaup$ and $\phiup$ from equations  \eqref{Eq:GP1} and \eqref{Eq:GP2}. We only give the translation map which leads to the inverse of the canonical map $\can{}$, the pertinent verification are routine computations which use the properties of $\gammaup$ and $\phiup$ already mentioned in subsection \ref{ssec:HAldfibre}. For the left one, that is, for $\can{}^{l}$ the translation map is given by:
$$
\xymatrix@R=0pt{ \RA{\omega_1}  \ar@{->}^-{\tauup^{l}}[r] & {}_{\Sscript{\balpha}}\Romaa{1}{2}_{\Sscript{\bbeta}} \, \tensor{A_2} \, \Romaa{1}{2}_{\Sscript{\bbeta}}
\\ \bara{p^*_1\tensor{T_P} p_1} \ar@{|->}[r] & \sum_{\alpha_{2,\, P}} \bara{p^*_1 \tensor{T_p} e_{\alpha_{2,\, P}}} \tensor{A_2}  \bara{\phiup(p_1) \tensor{T_{\dP}} \gammaup(e^*_{\alpha_{2,\, P}}) }  } 
$$ 
While for the right canonical map $\can{}^{r}$ the translation map is given by:
$$
\xymatrix@R=0pt{ \RA{\omega_2}  \ar@{->}^-{\tauup^{r}}[r] & {}_{\Sscript{\balpha}}\Romaa{1}{2} \, \tensor{A_1} \, {}_{\Sscript{\balpha}}\Romaa{1}{2}_{\Sscript{\bbeta}}
\\ \bara{p^*_2\tensor{T_P} p_2} \ar@{|->}[r] & \sum_{\alpha_{1,\, P}}   \bara{ \gammaup(e_{\alpha_{1,\, P}}) \tensor{T_{\dP}}\phiup(p^*_2) } \tensor{A_1} \bara{e^*_{\alpha_{1,\, P}} \tensor{T_p}p_2 }   } 
$$ 
This finishes the proof.
\end{proof}

Over the same algebra, the coaction of equations \eqref{Eq:LRA} and \eqref{Eq:RRA} are canonically lifted to the quotient algebras of Eq. \eqref{Eq:Ra} and the Lemma \ref{lema:Rcomalg} still working in this case as well.

\begin{corollary}\label{coro:Racomalg}
Assume that $A_{1}=A_{2}=A$.  Then the bicomodule  algebra structure of Lemma \ref{lema:Rcomalg}, induces a structure of $(\Ra{1}, \Ra{2})$-bicomdule algebra on $\Raa{1}{2}$.  Furthermore, $\Raa{1}{2}$ is a pseudo $(\Ra{1}, \Ra{2})$-bibundle.
\end{corollary}

The following  is the main result of this section, it is the affine groupoid schemes version  of the similar  well known result for affine group schemes (over an affine scheme). As before, the ground ring $K$ is assumed to be a field. For more notions and properties of geometrically transitive Hopf algebroids we refer to \cite{Bruguieres:1994, ElKaoutit:2015}.

\begin{theorem}\label{thm:PB}
Let $\cA$ be a symmetric rigid monoidal $K$-linear category with  the endomorphism algebra of the identity object $T_{\II}=K$. Assume we have two fibre functors $\omega_i: \cA \to  add(A_i)$, $i=1,2$ (i.e., ~faithful symmetric monoidal $K$-linear functors) and consider the associated $(\Roma{1},\Roma{2})$-bicomodule algebra $(\Romaa{1}{2}, \balpha, \bbeta)$. Then,
\begin{enumerate}[(i)]
\item if the Hopf algebroid $\Roma{1}$ (resp.,~$\Roma{2}$) is geometrically transitive, then $\Romaa{1}{2}$ is a left (resp.,~ right) principal $(\Roma{1}, \Roma{2})$-bundle.
\item if $\cA$ is an abelian locally finite $K$-linear category  and both functors $\omega_i$ are exact, then $\Romaa{1}{2}$ is a principal $(\Roma{1}, \Roma{2})$-bibundle with opposite bibundle isomorphic to $\Romaa{2}{1}$. 
\item if, in addition to the  assumptions of item $(ii)$, we also assume that $\omega_{1} =\omega \tensor{K} A$, for some fibre functor $ \omega: \cA \to \vect{K}$, then $\Raa{1}{2}$ is a principal $(\Ra{1}, \Ra{2})$-bibundle.
\end{enumerate}
\end{theorem}
\begin{proof}
$(i)$ If $\Roma{1}$ is geometrically transitive, then by \cite[Lemma 4.3(b)]{ElKaoutit:2015} every (left) $\Roma{1}$-comodule is faithfully flat as an $A_1$-module. In particular $\Romaa{1}{2}$ is faithfully flat as an $A_1$-module. Combining this with part $(i)$, we have that  $\Romaa{1}{2}$ is a left principal $(\Roma{1}, \Roma{2})$-bundle. Analogously, if we assume that $\Roma{2}$ is a geometrically transitive Hopf algebroid, then   $\Romaa{1}{2}$ is a right principal $(\Roma{1}, \Roma{2})$-bundle.

$(ii)$ Under these assumptions the reconstructed Hopf algebroids $\Roma{1}$ and $\Roma{2}$ are,  by \cite[Th\'eor\`eme 1.12]{Deligne:1990} (see also \cite[Th\'eor\`emes 5.2 et 7.1]{Bruguieres:1994}), both geometrically transitive. Therefore, part $(ii)$ implies that $\Romaa{1}{2}$ is a left and right principal bundle and so a principal bibundle. The opposite bibundle of $\Romaa{1}{2}$ is clearly isomorphic to $\Romaa{2}{1}$.

$(iii)$ We already know form Corollary \ref{coro:Racomalg} that $\Raa{1}{2}$ is a pseudo bibundle between  the Hopf $A$-algebras $\Ra{1}$ and $\Ra{2}$. The only remaining  condition is the faithfully flatness  of the involved $A$-algebras. Up to the canonical symmetric monoidal equivalence of $K$-linear categories between $\cA$ and  finite-dimensional comodules over the Hopf $K$-algebra $\rR_{\Sscript{K}}(\omega)=\rR_{\Sscript{K\tensor{}K}}(\omega)$, we know that $\Ra{1}$ is a flat $A$-module and so  it is faithfully flat (as the unit splits by the counit). To show that $\Raa{1}{2}$ is a faithfully flat over $A$, we proceed as in the proof of \cite[Th\'eor\`eme 4.2.2, page 155]{Saavedra:1972}, by realising $\Raa{1}{2}$ as an inductive limit of an inductive system of finitely generated and projective  $A$-modules whose transition morphisms are  split monomorphisms.   
\end{proof}

\section{The finite dual of a ring extension.}\label{sec:FDRE}
In this section we specialize the construction of Section \ref{sec:ICC} to the case of a ring homomorphism $A \to R$, which leads to the \emph{(right) finite dual $A$--coring $R^{\bcirc}$}, and of a functor $\chi: \cat{A}_R \longrightarrow \cat{A}^{R^{\circ}}$. We will show that, in the case of coalgebras over a Dedekind domain, our $R^{\bcirc}$ is isomorphic to the usual finite dual defined as the subspace of $R^*$ of all linear forms whose kernel contains a cofinite ideal of $R$. When $R_A$ is finitely generated and projective, we show that $R^{\bcirc}$ is isomorphic to $R^*$. We also study when $R^{\bcirc}$ is Galois and, what is more important, when $\chi$ is an isomorphism of categories. All these results are related to the injectivity of a map $R^{\bcirc} \to R^*$ to be defined below. 

\subsection{The finite dual  coring $R^{\bcirc}$}\label{ssec:Rcero}
Given an $A$--ring $R$, consider the category $\cat{A}_R$ of all right $R$--modules that are finitely generated and projective as right $A$--modules. We define the \emph{right finite dual} of the extension $A \to R$ as the $A$-coring $R^{\bcirc} = \recons{\cat{A}_R}$ described in Section \ref{sec:ICC}. 

We know from equation \eqref{comatInf} that $R^{\bcirc}$ is the factor $A$-coring $\fk{B}(\cat{A}_R)/\fk{J}(\cat{A}_R)$, where
\[
\fk{B}(\cat{A}_R) = \bigoplus_{P \in \cat{A}_R} P^* \tensor{T_P} P 
\]
and $\fk{J}(\cat{A}_R)$ is the $K$-submodule spanned by the set of  elements as in Eq.\eqref{Eq:JA}.
The  elements 
\[
\overline{p^* \tensor{T_P} p} = p^* \tensor{T_P} p + \fk{J}_{\cat{A}_R}, \quad p^* \in P^*, p \in P, P \in \cat{A}_R
\]
form a set of generators of $R^{\bcirc}$ as a $K$-module (and, of course, as an $A$-bimodule). Recall that 
the comultiplication of $R^{\circ}$ is given by 
\begin{equation}\label{Eq:deltao}
\Delta^{\bcirc} : R^{\bcirc} \longrightarrow R^{\bcirc} \tensor{A} R^{\bcirc}, \quad \overline{p^* \tensor{T_P} p} \longmapsto \overline{p^* \tensor{T_P} e_{\alpha,P} }\tensor{A} \overline{e_{\alpha,P}^* \tensor{T_P} p},
\end{equation}
where $\{ e_{\alpha,P}, e_{\alpha,P}^* \}$ denotes a finite dual basis for $P_A$. The counit of $R^{\bcirc}$ is
\begin{equation}\label{Eq:counito}
\varepsilon^{\bcirc} : R^{\bcirc} \longrightarrow A, \quad \overline{p^* \tensor{T_P}  p} \longmapsto p^*(p). 
\end{equation}

\subsection{The functors $\chi$ and $L$}\label{ssec:xiyL}
Consider the homomorphism of $A$-bimodules $\mathbf{b}:R^{\bcirc} \tensor{R} R \to A$ defined as the composite
\[
\xymatrix{\mathbf{b}:R^{\bcirc} \tensor{R} R \cong R^{\bcirc} \ar^-{\varepsilon^{\bcirc}}[r] & A}
\]
and let $\eta_R : R \to {}^*(R^{\bcirc})$ be its image under the adjunction isomorphism 
$$\lhom{A}{R^{\bcirc} \tensor{R} R}{A} \cong \lhom{R}{R}{{}^*(R^{\bcirc})}.$$
Explicitly, 
\begin{equation}\label{Eq:eta}
\eta_R (r) (\overline{p^* \tensor{T_P} p}) = p^*(pr).
\end{equation}
Recall that, since $R^{\bcirc}$ is an $A$-coring, we know that ${}^*(R^{\bcirc})$ is a ring with the convolution product. A straightforward computation shows that $\eta_R$ is an anti-homomorphism of rings. 

From the well-known (see, e.g \cite[19.1]{Brzezinski/Wisbauer:2003}) functor $\fk{l}:\rcomod{R^{\circ}} \to \lmod{{}^*(R^{\bcirc})}$ we get, after composing with the restriction of scalars functor associated to $\eta_R$, a functor  
$$
L : \cat{A}^{R^{\bcirc}} \longrightarrow  \cat{A}_R.
$$ 
This functor is explicitly given on objects as follows: given a right $R^{\bcirc}$-comodule $\varrho_M : M \to M \tensor{A} R^{\bcirc}$,  and using a Sweedler-type notation (summation understood),  we define the following right $R$-action on $M$
\begin{equation}\label{mr}
m \cdot r = m_0 m_1^*(m_1r), \quad r \in R, \; \varrho_M(m) = m_0 \tensor{A} \overline{m_1^* \tensor{T_P} m_1}.
\end{equation}

Conversely, every object $P \in \cat{A}_R$ is a right $R^{\bcirc}$-comodule with the coaction 
\begin{equation}\label{Eq:chip}
\chi_P : P \longrightarrow P \tensor{A} R^{\circ}, \quad p \longmapsto e_{\alpha,P} \tensor{A} \overline{e^*_{\alpha,P} \tensor{T_P} p}. 
\end{equation}
This gives the object map of a functor $\chi: \cat{A}_R \longrightarrow \cat{A}^{R^{\circ}}$. A straightforward computation shows that $L \circ \chi = id_{\cat{A}_R}$.

\subsection{The injectivity of the map $\zeta: R^{\bcirc} \to R^*$}\label{ssec:zeta}
Consider the image $\zeta: R^{\bcirc} \to R^*$ of the $A$-bilinear map $\mathbf{b}: R^{\bcirc} \tensor{R} R \to A$ under the adjunction isomorphism
\[
\rhom{A}{R^{\bcirc} \tensor{R} R}{A} \cong \rhom{R}{R^{\bcirc}}{R^*}
\]
This homomorphism of $(A,R)$-bimodules is explicitly given by
\begin{equation}\label{Eq:zeta}
\zeta : R^{\bcirc} \longrightarrow R^*, \quad \overline{p^* \tensor{T_P} p} \longmapsto (r \mapsto p^*(pr)). 
\end{equation}
Clearly, we have $\zeta(\overline{p^* \tensor{T_P} p})(r) =\eta_R(r)(\overline{p^* \tensor{T_P} p})$, where $\eta_R$ is the map defined in Eq.\eqref {Eq:eta}.

For each module $M_A$, define $\beta_M : M \tensor{A} R^{\bcirc} \to \rhom{A}{R}{M}$ by 
\[
\beta_M(m \tensor{A} \overline{{p^* \tensor{T_P} p}}) (r) = mp^*(pr).
\]
\begin{lemma}\label{betainy}
$\beta_M$ is injective for every $M_A$ finitely generated and projective if and only if $\zeta : R^{\bcirc} \to R^*$ is injective. 
\end{lemma}
\begin{proof}
Observe that, up to the canonical isomorphism $A \tensor{A} R^{\bcirc} \cong R^{\bcirc}$, we have that $\beta_A = \zeta$, which gives the direct implication. The converse is clearly deduced from the fact that the class of modules $M_A$ for which $\beta_M$ is injective is closed under finite direct sums and direct summands. 
\end{proof}

\begin{proposition}\label{isomor}
If $\zeta : R^{\bcirc} \to R^*$ is injective, then the funtor $\chi : \cat{A}_R \to \cat{A}^{R^{\bcirc}}$ is an isomorphism of categories. 
\end{proposition}
\begin{proof}
Since we already know that $L \circ \chi = id_{\cat{A}_R}$, we  only  need to prove that the composition $\chi \circ L$ is also the identity functor.  Let $M$ be an object of $\cat{A}^{R^{\bcirc}}$ with coaction $\varrho_M : M \to M \tensor{A} R^{\bcirc}$.  Then $L$ sends $M$ to a right $R$-module with the action \eqref{mr}.  With this structure, $M \in \cat{A}_R$, and, therefore, by applying $\chi$, we obtain a comodule $\chi_M : M \rightarrow M \tensor{A} R^{\bcirc}$. We need to check that $\varrho_M = \chi_M$ and, since $\beta_M$ is injective by Lemma \ref{betainy}, it is enough if we prove that $\beta_M \circ \varrho_M = \beta_M \circ \chi_M$. This follows from the following computation:
\[
(\beta_M \circ \varrho_M) (m)(r) = m_0m^*_1(m_1r) = mr
\]
\[
(\beta_M \circ \chi_M)(m)(r) = e_{\alpha,M}e_{\alpha,M}^*(mr) =mr.
\]
\end{proof}

\begin{proposition}\label{prop:R0Galois}
If $\zeta : R^{\circ} \to R^*$ is injective, then $R^{\bcirc}$ is a Galois $A$-coring. 
\end{proposition}
\begin{proof}
Recall that the canonical map $\mathsf{can}_{R^{\bcirc}} : \recons{R^{\bcirc}} \to R^{\bcirc}$ is defined by
\[
\mathsf{can}_{R^{\bcirc}} (\overline{m^* \tensor{T_M} m}) = m^*(m_0)\overline{m_1^* \tensor{T_{M_1}}m_1}.
\]
Therefore
\[
(\mathsf{can}_{R^{\bcirc}} \circ \recons{\chi})(\overline{p^* \tensor{T_P} p}) = \mathsf{can}_{R^{\bcirc}}(\overline{p^* \tensor{T_{\chi (P)}} p}) = p^*(e_{\alpha,P})\overline{e_{\alpha,P}^* \tensor{T_P} p} = \overline{p^* \tensor{T_P} p}
\]
In this way, $\mathsf{can}_{R^{\bcirc}} \circ \recons{\chi} = id_{R^{\bcirc}}$. By Proposition \ref{isomor}, $\chi$ is an isomorphism of categories and, henceforth, $\recons{\chi}$ is bijective. Therefore, $\mathsf{can}_{R^{\bcirc}}$ is bijective and $R^{\bcirc}$ is Galois.  
\end{proof}

\begin{proposition}\label{QP}
Assume that for every homomorphism of right $R$-modules $f : P \to R^*$, where $P \in \cat{A}_R$, there exists $\mathcal{Q} \subseteq \cat{A}_R$ and an exact sequence of right $R$-modules $\xymatrix@C=15pt{\bigoplus_{Q \in \mathcal{Q}} Q \ar^-{g}[r] & P \ar^-{f}[r] & R^*.}$ Then $\zeta : R^{\bcirc} \rightarrow R^*$ is injective. 
\end{proposition}
\begin{proof}
For every $P \in \cat{A}_R$ we consider the adjunction isomorphism of $(A,T_P)$-bimodules
\[
\rhom{R}{P}{R^*} \cong P^*, \quad f \mapsto (p \mapsto f(p)(1))
\]
with inverse
\[
P^* \longrightarrow \rhom{R}{P}{R^*}, \quad \varphi \mapsto (p \mapsto (r \mapsto \varphi(pr)))
\]
Writing $P^\star = \rhom{R}{P}{R^*}$, we obtain an isomorphism of $A$-bimodules 
\[
R^{\bcirc} \cong \frac{\bigoplus_{P \in \cat{A}_R}P^\star \tensor{T_P} P}{K\left\{ f \tensor{T_Q} tp - f t \tensor{T_P} p : \, f \in Q^\star,\, p
\in P,\, t \in T_{PQ},\, P,Q \in \cat{A}_R \right\}}
\]
Up to this isomorphism, $\zeta : R^{\bcirc} \to R^*$ is given by $\zeta(\overline{f \tensor{T_P} p}) = f(p)$, for $f \in P^\star$ and $p \in P$.  
Assume $\sum_i \overline{f_i \tensor{T_{P_i}} p_i} \in \ker \zeta$, that is, $\sum_i f_i(p_i) = 0$. Write $P = \oplus_i P_i$ and define $f :P \to R^*$ and $p \in P$ uniquely by the conditions $f \iota_i = f_i$ and $\pi_i(p) = p_i$ for every $i$, where $\iota_i : P_i \to P$ is the $i$-th canonical injection. We use the notation $\pi_i : P \to P_i$ for the canonical projections.  Since $p \in \ker f$, there exist a homomorphim of right $R$-modules $g : Q \to P$ and $q \in Q$ such that $p = g(q)$.  The following computation
\[
\sum_i \overline{f_i \tensor{T_{P_i}} p_i} = \sum_i \overline{f_i \tensor{P_i} \pi_i(g(q))} = \sum_i \overline{f_i \pi_i g \tensor{T_Q} q } = \overline{fg \tensor{T_Q} q} = 0.
\]
 finishes the proof.
\end{proof}

When $R_A$ is finitely generated and projective, it is well-known that $R^*$ is an $A$--coring (see e.g. \cite[Proposition 2.11]{Boe:HA}). 

\begin{corollary}\label{Exm:R+}
If $R_A$ is finitely generated and projective, then $\zeta : R^{\bcirc} \to R^*$ is an isomorphism of $A$--corings. 
\end{corollary}
\begin{proof}
Our map $\zeta : R^{\bcirc} \to R^*$ is injective by Proposition \ref{QP}, since $R$ is a generator of $\cat{A}_R$. In this case, $\zeta$ is obviously surjective. 
\end{proof}

Recall (see, e.g.~\cite[page 22]{Stenstrom:1975}) that a ring is said to be right hereditary if every right ideal is projective as a right module. 

\begin{corollary}\label{hereditary}
If $A$ is a right hereditary right noetherian ring, then $\zeta : R^\circ \to R^*$ is injective. 
\end{corollary}
\begin{proof}
Recall that, over a right hereditary ring, submodules of projective right modules are projective modules (see, e.g., \cite[Proposition I.9.5]{Stenstrom:1975}). Thus, given a homomorphism of right $R$-modules $f : P \to R^*$, where $P \in \cat{A}_R$, then $Q : = \ker f$ is projective as a right $A$--module, and, of course, it is finitely generated over $A$, since we are assuming that $A$ is right noetherian. Now, apply Proposition \ref{QP}.
\end{proof}

\begin{example}\label{exam:Dedekind}
Of course, every commutative Dedekind domain $A$ fulfills the hypotheses of Corollary \ref{hereditary}. 
\end{example}

\begin{example}\label{exam:ssA}
Obviously, if $A$ is semi-simple Artinian, then $A$ fulfills the hypotheses of Corollary \ref{hereditary}. 
\end{example}

\begin{problem}\label{problemauno}
Corollaries \ref{Exm:R+} and \ref{hereditary}  require ``extreme'' conditions to guarantee, by virtue of Proposition \ref{QP}, that $\zeta$ is injective. More precisely, Corollary \ref{Exm:R+} imposes a strong condition on the ring extension $A \to R$, while Corollary \ref{hereditary} restricts the kind of ground ring $A$ we are allowed to work with. It would be interesting, in view of the consequences of the injectivity of $\zeta$ (see Theorem \ref{thm:lduality} and Corollary \ref{LRgrupoid} below), to investigate more general hypotheses (presumably, module-theoretical conditions) that would imply it. 
\end{problem}

\subsection{Dual coalgebras}
Let $R$ be an algebra over a  commutative noetherian hereditary ring $A$. As a generalization to the case of algebras over fields, it is possible to define a structure of  $A$--coalgebra over the $A$--submodule $R'$ of $R^*$ consisting of those $\varphi \in R^*$ such that $\ker \varphi$ contains an $A$--cofinite ideal of $R$ (see \cite[Theorem 2.8, Proposition 2.11]{Abuhlail/Gomez-Torrecillas/Wisbauer:2000}).  Here, an ideal $I$ of $R$ is \emph{$A$--cofinite} if $R/I$ is a finitely generated $A$--module.  

The map $\zeta : R^{\bcirc} \to R^*$ factorizes through $R'$. Indeed, for any generator $\overline{p^* \otimes_{T_P} p} \in R^{\bcirc}$,  we need to prove that the cyclic right $R$--module $\varphi R$, where $\varphi = \zeta (\overline{p^* \otimes_{T_P} p})$, is finitely generated as an $A$--module (see \cite[2.3]{Abuhlail/Gomez-Torrecillas/Wisbauer:2000}). Let $r, s \in R$. Since $\zeta$ is right $R$--linear, $A$ is central in $R$, and by using a suitable dual basis, we get
$$(\varphi r)(s) = p^*(prs) = \sum_{\alpha_P}p^*(e_{\alpha_P}e_{\alpha_P}^*(pr)s) = \sum_{\alpha_P}p^*(e_{\alpha_P}se_{\alpha_P}^*(pr)) = \sum_{\alpha_P}p^*(e_{\alpha_P}s)e_{\alpha_P}^*(pr).$$  This implies that $\varphi r$ belongs to the $A$--submodule of $R^*$ generated by the finite set of all $p^*(e_{\alpha_P}-)$'s and, hence, $\varphi R$ is finitely generated as an $A$--module. 

It follows from Example \ref{exam:Dedekind} 
that $\zeta: R^{\bcirc} \to R^*$ is injective. Let $\tau : R^{\bcirc} \otimes R^{\bcirc} \to R^{\bcirc} \otimes R^{\bcirc}$ denote the flip map. A straightforward computation shows that the following diagram is commutative.
\[
\xymatrix@R=20pt{R^{\bcirc} \ar^-{\tau \Delta^{\bcirc}}[rr] \ar_-{\zeta}[d] & & R^{\bcirc} \tensor{} R^{\bcirc} \ar^{\zeta \tensor{} \zeta}[d]\\
R^* \ar^-{m^*}[dr] & & R^* \tensor{} R^* \ar[dl] \\
 & (R \tensor{} R)^* & }
\]
In resume, we have that, if $R$ is an  algebra over a commutative noetherian hereditary ring $A$, then $\zeta : R^{\bcirc} \to R'$ is an injective anti-homomorphism of $A$--coalgebras.

\begin{proposition}\label{exam:dualcoalg}
Let $R$ be an algebra over a commutative Dedekind domain $A$. Then $\zeta : R^{\bcirc} \to R'$ is an anti-isomorphism of $A$--coalgebras. 
\end{proposition}
\begin{proof}
We need just to show that $\zeta : R^{\bcirc} \to R'$ is surjective. Let $\varphi \in R'$. We know that the cyclic right $R$--module $\varphi R$ is a finitely generated $A$--module. On the other hand, $\varphi R \subseteq R^*$ and, therefore, $\varphi R$ is a torsion-free $A$--module. Since $A$ is a Dedekind domain, we get that $\varphi R$ is a fgp $A$--module. Now, $I : = \{ r \in R : \varphi r = 0 \}$ is a right ideal of $R$ and, since $\varphi R \cong R/I$, we get that $P : = R/I$ is fgp as an $A$--module. On the other hand, $I \subseteq \ker \varphi$, which implies that there is $\overline{\varphi} \in P^*$ such that $\overline{\varphi}(r + I) = \varphi (r)$ for all $r \in R$. Hence, $\varphi = \zeta ( \overline{\overline{\varphi} \otimes_{T_P} (1 + I)})$. 
  \end{proof}
  
\begin{remark}
The proof of Proposition \ref{exam:dualcoalg} works to prove that if $R$ is an algebra over a semisimple commutative ring $A$, then $\zeta : R^{\bcirc} \to R'$ is an anti-isomorphism of coalgebras. 
\end{remark}

\section{The finite dual of a right bialgebroid, and of a cocommutative Hopf algebroid.}\label{sec:coHalgd}

We show that the right finite dual $U^{\bcirc}$ of a right bialgebroid $(A,U)$ is a left bialgebroid. When $A$ is commutative, this fact can be deduced from \cite[Exemple, pp. 5849]{Bruguieres:1994}, although that construction cannot be directly extended to the setting of a non-commutive basis $A$ because $add(A_A)$ is not a monoidal category.  We also prove that the map $\zeta : U^{\bcirc} \to U^*$ is a homomorphism of $A^e$--rings, when $U^*$ is endowed with the convolution product.
Theorem \ref{thm:lduality} states that the finite dual of a  cocommutative Hopf algebroid with commutative base ring  is a commutative Hopf algebroid. This could also have been deduced from \cite[Example, pp. 5849]{Bruguieres:1994}, we include an elementary proof. Our approach could be useful to treat specific examples. This will be illustrated in the last section.  Theorem \ref{thm:lduality} also contains our main contribution in this section, namely, a sufficient condition to get a monoidal equivalence between the category of the $A$--profinite modules over a cocommutative Hopf algebroid over $A$ and the $A$--profinite representations of its associated affine groupoid via the finite dual construction. This equivalence works in a non necessarily tannakian context.

\subsection{Duality for bialgebroids}\label{ssec:Dbialgd}

Let $ U$ be a right bialgebroid over a  (possibly non commutative) $K$--algebra $A$, and consider $ U$ as an $A$--ring via its source map $\Sf{s}: A \to  U$.  
Consider the fibre functor $\Sf{s}_* : \cat{A}_{ U} \to add(A_A)$ as in Section \ref{sec:ICC}, and the corresponding right finite dual $A$--coring $ U^{\bcirc}$ as in Subsection \ref{ssec:Rcero}. Our aim is to endow $ U^{\bcirc}$ with the structure of a left bialgebroid over $A$ in such a way that the map $\zeta:  U^{\bcirc} \to  U^*$ defined in Eq.\eqref{Eq:zeta}  becomes a morphism of $\Ae$-rings. Here, the right convolution ring $ U^*$ is an $\Ae$--ring via the homomorphism of rings 
\begin{equation}\label{Eq:xi}
\xi: \Ae \longrightarrow   U^*, \quad \Big(a\tensor{}b^o\longmapsto \big[ u \mapsto a\varepsilon(\Sf{t}(b)u)\big]\Big).
\end{equation}

\begin{lemma}\label{lema:multiplication} 
Let $(A,  U)$ be a right bialgebroid. Then $ U^{\bcirc}$ admits a structure of $\Ae$-ring such that $\zeta:  U^{\bcirc} \to  U^*$ is a homomorphism of $\Ae$-rings.
\end{lemma}
\begin{proof}

Observe first that $A \in \cat{A}_{ U}$, where the right $ U$--module structure is given by the action 
\begin{equation}\label{Eq:AU}
a.u \,=\, \varepsilon(\Sf{t}(a) u)\,=\varepsilon(\Sf{s}(a)u).
\end{equation} 
With this right $U$-module structure, it is clear that $A$ is the identity object of the monoidal category $\cat{A}_{ U}$. Now, as in Eq. \eqref{Eq:Etao} of subsection \ref{ssec:HAldfibre}, the element $\overline{id_A \tensor{T_A} 1} \in  U^{\bcirc}$ is  the unit for the multiplication \eqref{Eq:mo}, and we have that 
\begin{equation}\label{Eq:etao} 
\etaup^{\bcirc}: \Ae \longrightarrow  U^{\bcirc}, \qquad \Big( (a\tensor{} \bop)\longmapsto   (\bara{\Sf{l}_a  \tensor{T_A} b})\Big),
\end{equation}
is a homomorphism of rings (the associated source and target are denoted by $\Sf{s}^{\circ}$ and $\Sf{t}^{\circ}$, respectively) where the endomorphism ring $T_A$ of $A_U$ is the commutative subalgebra of $A$ defined by 
\begin{equation}\label{Eq:TA}
T_A=\Big\{ a \in Z(A)|\;\,  a \varepsilon(u) = \varepsilon\big(  \Sf{t}(a) u\big), \text{ for every } u \, \in U \Big\}.
\end{equation}

Recall that the map $\zeta$  sends   $\bara{p^*\tensor{T_P} p} \mapsto \left[ u \mapsto p^*(pu)\right]$. In order to see that it is a homomorphism of $A^e$--rings, let us check first that it is unital. So,  
$$\zeta \circ \eta(a\tensor{}b^o)(u)\,=\,\zeta(\bara{\Sf{l}_a\tensor{T_A} b})(u)\,=\, \Sf{l}_a(bu)\,=\, a\varepsilon(\Sf{t}(b)u)\,=\,\xi(a\tensor{}b^o)(u),\; \forall u \in  U.$$ Hence  $\zeta \circ \eta\,=\, \xi$, where $\xi$ is the map defined in Eq.\eqref{Eq:xi}.  Now, using the above multiplication Eq.\eqref{Eq:mo}, we have 
\begin{eqnarray*}
\zeta\left( (\bara{p^*\tensor{T_P}p})\,.\, (\bara{q^*\tensor{T_Q}q}) \right)(u) &=& \zeta\Big(\bara{(q^* \star p^*) \tensor{T_{Q\tensor{\peque{A}}P}} (q\tensor{\peque{A}}p)}\Big) (u) \\
&=& (q^* \star p^*) ((q\tensor{\peque{A}}p)\, u) \\ &=& (q^* \star p^*) (qu_1\tensor{\peque{A}}pu_2) \,\,=\,\,p^*(q^*(qu_1) (pu_2))  \\ &=&  p^*((pu_2)(\Sf{t}(q^*(qu_1)))).
\end{eqnarray*} 
On the other hand, using the convolution multiplication defined by the coring $ U_{\Ae}$, we have
\begin{eqnarray*}
\zeta(\bara{p^*\tensor{T_P}p}) * \zeta(\bara{q^*\tensor{T_Q}q}) \, (u) &=& \zeta(\bara{p^* \tensor{T_P} p})\lr{ \zeta(\bara{q^*\tensor{T_P}q})(u_1) u_2} \\
&=& \zeta(\bara{p^* \tensor{T_P} p})( q^*(qu_1) u_2)  \\ &=& p^*( p( q^*(qu_1) u_2))  \,\,=\,\, p^*( p( u_2 \Sf{t}(q^*(qu_1)) )) \\ &=& p^*((pu_2)\Sf{t}(q^*(qu_1))).
\end{eqnarray*} 
Henceforth, $\zeta$ is multiplicative, and this finishes the proof.
\end{proof}

\begin{proposition}\label{prop:Uo}
If $(A,  U)$ is a right bialgebroid, then $(A,  U^{\bcirc})$ is a left bialgebroid.
\end{proposition}
\begin{proof}
We know that $U^{\bcirc}$ is constructed from the monoidal category $\cat{A}_U$ and the forgetful functor $\omega=\Sf{s}_{*}: \cat{A}_U \to \bimod{A}{A}$, that is, we have that $U^{\bcirc}=\rR(\cat{A}_U)$ as in the notation of subsection \ref{ssec:HAldfibre}.  Therefore, we can apply the first statement of Proposition \ref{prop:TBialgd} to obtain the claim.  
\end{proof}

\begin{remark}\label{rem:leftRight}
There is a kind of symmetry in Proposition  \ref{prop:Uo}. This means that given a left bialgebroid $(A,V)$, using it target map $\Sf{t}: A \to V$ and it associated category ${}^{V }\cat{A}$ of left $V$-modules which are finitely generated and projective as $A$-modules via the fibre functor ${}_{*}\Sf{t}: \lmod{V} \to add({}_AA)$. The reconstruction process of Section \ref{sec:ICC} leads to a right bialgebroid $\big(A,\rR({}^{V}\cat{A})\big):=(A,{}^{\bcirc}V)$.
\end{remark}

\subsection{Duality for Hopf algebroids}\label{ssec:DHalgd} 
In this  subsection we assume that $(A,  U)$ is a left Hopf algebroid over a commutative $K$-algebra $A$, and the underlying $A$-coring of $U$ is co-commutative.  Bialgebras over field extensions studied in \cite{Nichols:1985, Chase:1976} and those over a commutative algebra studied in \cite{Rumynin:2000}, as well as the universal algebras of  Lie algebroids or, in general, of Lie-Rinehart algebras (Examples  \ref{Exam: bundles} and \ref{exam:LR}), are all examples of this class of left  bialgebroids. 

It is easily checked that, under the current assumptions,  the source of $U$ is equal to its target (i.e.,~$\Sf{s}=\Sf{t}$). This fact will be implicitly used  in the sequel.  Recall that for such a bialgebroid $(A,U)$,  the category $\cat{A}_{U}$ consists  of right $U$-modules whose underlying $A$-modules are finitely generated and projective, using either the functor $\Sf{s}_{*}$ or $\Sf{t}_{*}$. 

\begin{proposition}\label{prop:rigida}
Let $(A, U)$ be a co-commutative right Hopf algebroid over a commutative algebra $A$. Consider the category $\cat{A}_{ U}$ with the fibre functor $\Sf{s}_*: \cat{A}_{ U} \to add(A)$. Then $\cat{A}_{ U}$ is a monoidal symmetric and rigid category with $\Sf{s}_*$ a strict monoidal fibre functor. 
\end{proposition}
\begin{proof}
First observe that, since $\Sf{t}=\Sf{s}$, each object in $\cat{A}_{ U}$ is a central $A$-bimodule. On the other hand, for every pair of objects $P, Q \in \cat{A}_{ U}$ the flip map $\tau_{P,\, Q}: P\tensor{A}Q \to Q\tensor{A}P $, sending $p\tensor{\peque{A}}q \mapsto q\tensor{\peque{A}}p$ is actually an arrow in $\cat{A}_{ U}$, since $ U$ is co-commutative.  The rigidity of $\cat{A}_{ U}$ is immediate from Lemma \ref{lema:duales}, and the duals are described as follows: For every arrow $f : P \to Q$ in $\cat{A}_{ U}$, its $A$-linear dual map $f^*: Q^* \to P^*$ is clearly an arrow in the same category $\cat{A}_{ U}$, where $P^* , Q^*$ are objects of $\cat{A}_{ U}$ by the action of Eq.\eqref{Eq:hom-action}.
Thus, $P^*$ is a dual object in $\cat{A}_{ U}$ for an object $P \in \cat{A}_{ U}$,  and the duality is given by  the following   arrows in $\cat{A}_{ U}$
\begin{equation}\label{Eq:ev-db}
ev: P^*\diama \, P=P^{*}\tensor{A}P \to A, \big(p^*\tensor{A}p \mapsto p^*(p)\big),\quad db: A \to P\diama \, P^*=P\tensor{A} P^{*}, \big(1 \mapsto e_{\alpha,\, P}\tensor{A}e_{\alpha,\, P}^*\big),
\end{equation}
(respectively called \emph{evaluation} and \emph{dual-basis})
where $\{e_{\alpha, P}, e_{\alpha,P}^*\}$ is a dual basis for $P$.
The rest of the proof  is now clear.
\end{proof}

The main results of this section are the last two parts of the following theorem. 

\begin{theorem}\label{thm:lduality}
Let $(A, U)$ be a co-commutative right Hopf algebroid over a commutative algebra $A$. Then
\begin{enumerate}
\item  $(A, U^{\bcirc})$ is a commutative Hopf algebroid over $T_A$.
\item The functor $\chi : \mathcal{A}_U \to \mathcal{A}^{U^{\bcirc}}$ is strict monoidal and preserves the symmetry. 
\item If $\zeta : U^{\bcirc} \to U^*$ is injective, then $\chi$ is an isomorphism of symmetric monoidal categories. 
\end{enumerate}
\end{theorem}
\begin{proof}
(1) This part follows from Propositions \ref{prop:rigida} and \ref{prop:TBialgd}, and the Hopf algebroid structure maps are explicitly given as follows. The algebra structure is given by the multiplication of equation \eqref{Eq:mo} and unit the algebra map $\etaup$ of equation \eqref{Eq:etao}. The  the comultiplication is the algebra map
\begin{equation}\label{Eq:deltao}
\Delta^{\circ} :   U^{\bcirc}  \longrightarrow  U^{\bcirc} \tensor{A}  U^{\bcirc}, \quad \overline{p^* \tensor{T_P} p} \longmapsto \sum_{\alpha_{P}}\overline{p^* \tensor{T_P} e_{\alpha,P} }\tensor{A} \overline{e_{\alpha,P}^* \tensor{T_P} p},
\end{equation}
where $\{ e_{\alpha,P}, e_{\alpha,P}^* \}$ denotes a finite dual basis for $P_A$. The counit of $ U^{\bcirc}$ is
\begin{equation}\label{Eq:counito}
\varepsilon^{\circ}:  U^{\bcirc} \longrightarrow A, \quad \overline{p^* \tensor{T_P}  p} \longmapsto p^*(p),
\end{equation}
and the antipode is the algebra map 
\begin{equation}\label{Eq:So}
\Scr{S}^{\circ}:  U^{\bcirc} \longrightarrow  U^{\bcirc}, \quad (\bara{p^*\tensor{T_P} p} \longmapsto \bara{\psiup(p) \tensor{T_{P^*}} p^*}), 
\end{equation}
where $\psiup: P \to (P^*)^*$ is the canonical  isomorphism of $A$-modules, as in subsection\ref{ssec:HAldfibre}.

(2)   Let $P, Q$ be two objects  in $\mathcal{A}_{U}$. Then the right $U^{\bcirc}$-comodule structure of $\chi\big(P\tensor{A}Q\big)$, is given as in \eqref{Eq:chip} by the coaction:
$$
\varrho_{\chi(P\tensor{A}Q)}: P\tensor{A}Q \longrightarrow P\tensor{A}Q \tensor{A}U^{\bcirc}, \quad \Big(\, p\tensor{A}q \longmapsto \sum (e_{\alpha, P}\tensor{A}e_{\beta, Q}) \tensor{A} \bara{(e_{\alpha, P}^*\star e_{\beta, Q}^*) \tensor{T_{P\tensor{A}Q}} (p\tensor{A}q)} \,\Big)
$$
where $\{ e_{\alpha,P}, e_{\alpha,P}^{*} \}$ and $\{ e_{\beta,Q}, e_{\beta,Q}^{*} \}$ denote as above the dual basis of $P_{A}$ and $Q_{A}$. On the other hand, the tensor product  $\chi(P) \tensor{A} \chi(Q)$, in the monoidal subcategory of right $U^{\bcirc}$-comodules $\mathcal{A}^{U^{\bcirc}}$, has the following comodule structure:
$$
\varrho_{\chi(P)\tensor{A}\chi(Q)}: P\tensor{A}Q \longrightarrow P\tensor{A}Q \tensor{A}U^{\bcirc}, \quad \Big(\, p\tensor{A}q \longmapsto \sum p_{(1)}\tensor{A}q_{(1)} \tensor{A} p_{(2)}q_{(2)} \, \Big)
$$
Now, using equation \eqref{Eq:chip} and the commutativity of the multiplication of $U^{\bcirc}$, as given in \eqref{Eq:mo}, we get that $\varrho_{\chi(P)\tensor{A}\chi(Q)} = \varrho_{\chi(P\tensor{A}Q)}$. Therefore, $\chi(P)\tensor{A}\chi(Q) = \chi(P\tensor{A}Q)$, for any two objects $P,Q$ in $\mathcal{A}_{U}$.

The identity object of $\mathcal{A}_{U}$ is the right $U$-module $A$ with action $a.u = \varepsilon(au)$, for any $a \in A$ and $u \in U$. The image by $\chi$ of this object has the right coaction $\varrho_{\chi(A)}: A \to A\tensor{A}U^{\bcirc}$ sending $a \mapsto 1\tensor{A}\bara{1\tensor{T_{A}} a}$. Thus, by equation \eqref{Eq:etao},  we have $\varrho_{A}(a) = \sf{t}^{\circ}(a)$, for every $a \in A$. Hence, $\chi(A)=(A,\sf{t}^{\circ})$ the identity object of the monoidal category $\mathcal{A}^{U^{\bcirc}}$. We have then shown that $\chi$ is a strict monoidal functor. Lastly, since both monoidal categories $\mathcal{A}_{U}$ and $\mathcal{A}^{U^{\bcirc}}$ have the flip as symmetry, one trivially obtains that $\chi$ is a symmetric monoidal functor.  

(3) Follows from Proposition \ref{isomor}.
\end{proof}

\begin{corollary}\label{DedekindDiff}
Let $(A, U)$ be a cocommutative right Hopf algebroid over a Dedekind domain $A$. Then the category $\mathcal{A}_U$ is isomorphic, as a symmetric monoidal category, to $\mathcal{A}^{U^{\bcirc}}$.
\end{corollary}
\begin{proof}
This is deduced from Corollary \ref{hereditary} and Theorem \ref{thm:lduality}.
\end{proof}

We also get the following consequence of Theorem \ref{thm:lduality}, whose geometrical interpretation is that, under suitable conditions, given a Lie-Rinehart algebra there is a groupoid with the ``same'' representation theory (see Example \ref{exam:LR} for details on the Hopf algebroid attached to a Lie-Rinehart algebra).

\begin{corollary}\label{LRgrupoid}
Let $L$ be a Lie-Rinehart algebra over a commutative ring $A$, and let $\mathcal{U}(L)$ denote its universal enveloping Hopf algebroid. If $\zeta$ is an injective map (e.~g.~if $A$ is a Dedekind domain), then the category of $A$-profinte right $\mathcal{U}(L)$-modules is isomorphic, as a symmetric rigid monoidal category, to the category of  $A$-profinite right comodules over the commutative Hopf algebroid $\mathcal{U}(L)^{\bcirc}$. 
\end{corollary}

\begin{remark}\label{rem:rduality}
Starting with a commutative Hopf algebroid  $(A, V)$, then, by Proposition \ref{prop:Uo}, we know that $(A,V^{\bcirc})$ is a right bialgebroid.   The fact that $(A, U^{\bcirc})$ is actually a right Hopf algebroid can be shown using the tecniques developed in \cite{P.H. Hai:2008}.  In analogy with the classical situation of Hopf algebra over fields, it stills then to check that $ U^{\bcirc}$ is a co-commutative $A$-coring. It seems that in general, there  is no direct way of proving this property. However, it can be derived under some   assumptions, which are always fulfilled in the finite case (i.e., when  $V_A$ or ${}_AV$ are finitely generated and projective). Precisely, if we assume that  the map $\zeta: V^{\bcirc} \to V^*$ of subsection \ref{ssec:zeta}  and the canonical map 
\begin{equation}\label{Eq:Psi}
\Psi_{V}: V^{*} \times_A V^{*} \longrightarrow (V\tensor{\peque{A\tensor{}A}}V)^*,\quad \lr{ (\bara{p^*\tensor{T_P}p}) \times _A (\bara{q^*\tensor{T_Q}q}) \longmapsto \left[\underset{}{} u\tensor{\peque{A\tensor{}A}}v \mapsto p^*(pu)\, q^*(qv) \right] }
\end{equation}
(here $\times_A$ is the Sweedler-Takeuchi's product \cite{Swe:GOSA, Tak:GOAOAA}), are injective and that $V_A$ is flat,  then one can deduce, from the following commutative diagram,  the co-commutativity of the comultiplication $\Delta^{\bcirc}$ of the $A$-coring $V^{\bcirc}$
$$
\xymatrix@C=50pt@R=18pt{
V^{\bcirc} \ar@{->}_-{\Delta^{\bcirc}}[dd]  \ar@{->}^-{\zeta}[r] & V^* \ar@{->}^-{\mu^{*}}[dr]  & \\ & &  (V\tensor{\peque{A\tensor{}A}}V)^* \\
V^{\bcirc} \times_A V^{\bcirc} \ar@{->}^-{\zeta \times_A \zeta}[r] &  V^{*} \times_A V^{*}  \ar@{->}_-{\Psi_{V}}[ur]  &}
$$
Following the observations of Example \ref{Exm:R+},  the duality stated in \cite[Propositions 3, 4]{Rumynin:2000} (see also \cite{KS:2003, CGK:2016}), is now a particular instance of the one established hereby.
\end{remark}

\section{Application:  The finite dual of the first Weyl algebra, differential Galois groupoid and PV extensions}\label{sec:Weyl}
In this section, we illustrate our methods by treating, in an exhaustive way,   the universal Hopf algebroid of the transitive Lie algebroid of vector fields over the complex affine line $\mathbb{A}^1_{\mathbb{C}}$. In other words the first Weyl $\mathbb{C}$-algebra. We first show that the category of differential modules is a Tannakian category (in the sense of \cite{Deligne:1990}) which is identified with the  category of comodules  over the finite dual with underlying finitely generated free $\mathbb{C}[X]$--modules. Second we show that for a fixed differential module $(M, \partial)$,  this finite dual contains a commutative Hopf algebroid denoted by $\Um$ whose category of comodules $\frcomod{\Um}$ with finitely generated underlying modules, is equivalent, as a symmetric monoidal $\mathbb{C}$-linear  category, to the full subcategory  $\langle M \rangle_{\Sscript{\otimes}}$ of differential modules  which are sub-quotients objects of $(M,\partial)$. In analogy with the classical differential Galois theory over $\mathbb{C}(X)$, the associated affine algebraic groupoid of $\Um$, is then termed the \emph{differential Galois  groupoid} attached to $(M,\partial)$ (or to the linear differential matrix equation defined by $(M,\partial)$). We also combine our result with those of \cite{Andre:2001}, in order to give an explicit description of a Picard-Vessiot exetension of $(\Cc[X], \frac{\partial}{\partial X})$ for $(M, \partial)$. In the last subsection, we compare our approach with that of Malgrange \cite{Malgrange:2001} and Umemura \cite{Umemura:2009}, and give several illustrating examples.

\subsection{The Hopf algebroid structure of the first Weyl algebra}  
Let $A=\Cc[X]$ be polynomial ring in the variable $X$ over the field of complex numbers, and consider its (noncommutative) ring of differential operators $ U:= \Cc[X][Y, \frac{\partial}{\partial X}]$, that is, the first  Weyl algebra. We shall  consider $ U$ as  free right $A$-module with basis $\{Y^n\}_{ n \in \mathbb{N}}$, and  with left action given by the  rule $a Y = Ya + \frac{\partial a}{\partial X}$, for every $a \in A$.  This algebra is clearly isomorphic to the universal right Hopf algebroid of the transitive Lie algebroid $(A, \Der{\Cc}{A})$, see Example \ref{exam:LR}. The structure maps 
are $$ \Delta(Y)\,=\, 1\tensor{A}Y + Y\tensor{A}1, \quad \varepsilon(Y)\,=\,0, \text{ and } Y_{-}\tensor{A}Y_{+}\,=\, 1\tensor{A}Y - Y \tensor{A}1.$$

We are interested in describing the relationship between the category of comodules over the finite dual $U^{\bcirc}$ and the category of differential modules over the differential ring $A$. To this end,  we will apply Theorem \ref{thm:lduality} to the pair $(A,U)$. But,  first let us make the following general remark.

\begin{remark}\label{rem:XPartial}
Theorem \ref{thm:lduality} applies for any ring of differential operators $A[Y,\delta]$, for $\delta$ any $K$--linear derivative of $A$. In particular, Corollary \ref{DedekindDiff} leads to a monoidal equivalence of categories between the category $\Diff{A}$ of $A$--profinite differential modules and the $A$--profinite representations of the affine groupoid represented by the finite dual of $A[Y,\delta]$, whenever $A$ is a Dedekind domain  (e.g.~the coordinate ring of an irreducible smooth curve over $\mathbb{C}$).
If $A$ is not a field, $\Diff{A}$ will probably fail to be abelian, as the example of $ \mathbb{C}[X][Y, \delta]$ shows by taking $ \delta (f(X)) = X\frac{\partial f(X)}{\partial X}$ for all $f(X) \in \mathbb{C}[X]$. What makes the Weyl algebra so special is that, in this case, $\Diff{A}$ is abelian. 
\end{remark}

\subsection{Differential modules $\Diff{A}$ as a Tannakian category}\label{ssec:Diff} 
Recall that a \emph{differential right module} over the differential ring $A$ is a finitely generated right $A$-module equipped with a  $\Cc$-linear map $\partial : M \to M$ such that $\partial(xa)\,=\, \partial x.a + \partial a. x$, for every $a \in A$ and $x \in M$ (here $\partial a$ stand for $\partial a(X)/\partial X$). The linear map $\partial$ is called \emph{the differential of $M$}. In all what follows the underlying $A$-module of a right differential module, will be considered as central (or symmetric) $A$-bimodule, that is, we have $a x = xa $, for every $x \in M$ and $a \in A$. 

Every differential module is in fact free of finite rank as an $A$-module: if $x \in M$ is a torsion element, then $ax =0$ for some nonzero $a \in A$. It follows that $ (\partial x)a + x \partial a = 0$, whence $ (\partial x)a^2 =0$.  This shows that the ($\mathbb{C}$-finite dimensional) torsion submodule $t(M)$ of $M$ is indeed a differential module. This is only possible when $t(M) = 0$  (see, e.g.~\cite[Lemma 4.2]{Bjork:1979}). 

Using the notation of Section \ref{sec:ICC}, the category $\cat{A}_{ U}$ is in this case the category of all differential modules over $A$, equivalently \emph{linear differential matrix equations}. 
If we denote by $\{e_1,\dots,e_m\}$ any basis of $M$ over $A$, the differential $\partial$ is then given by a matrix $\mat{M}=(a_{ij}) \in M_{m}(A)$ such that  
\begin{equation}\label{Eq:aij}
\partial e_i\,=\, -\sum_{j=1}^{m}  e_ja_{ji}.
\end{equation}
The  minus sign in introduced, not only for historical reasons, but also for  computational ones\footnote{As we will notice in subsection \ref{ssec:PV}, this depends on considering  the principal bibundle structure of a Picard-Vessiot extension or its opposite bibundle, see Remark \ref{rem:PV}.}. So if we identify  an element $ y \in M$ with  its coordinate  column in $A^n$, we have 
$$
 \partial y \,=\, \columna{\partial y_1}{\partial y_m}  -  \mat{M} \columna{y_1}{y_m} . 
$$
Thus $\ker(\partial)$ is the solution space  of the following linear differential matrix equation 
\begin{equation}\label{Eq:ldm}
\columna{y_i^{\prime}}{y_m^{\prime}}  \,\,=\,\, \mat{M} \columna{y_1}{y_n}. 
\end{equation}

In analogy to \cite[\S 2.2]{VanderPut/Singer} we denote by $\Diff{A}$ the category $\cat{A}_{ U}$. A \emph{morphism of differential modules} $f: (M,\partial) \to (N,\partial) $ is an $A$-linear map $f : M \to N$ which commutes with differentials, that is,  $\partial \circ f = f \circ \partial$.  

Next we list the properties of $\Diff{A}$. First, we know from Section \ref{sec:coHalgd} that 
the category $\Diff{A}$ inherits a monoidal symmetric structure from the category of right $U$-modules. In this case,  the tensor product of two objects $(M,\partial), (N,\partial)$ in $\Diff{A}$ is again a differential module with differential map
\begin{equation}\label{Eq:ptensor}
\partial: M\tensor{A}N \longrightarrow M\tensor{A}N, \quad \Big( \partial (x\tensor{A} y) \,=\,  \partial (x)\tensor{A} y + x\tensor{A} \partial (y)  \Big)
\end{equation}
By Lemma \ref{lema:duales}, $\Diff{A}$ is then a monoidal symmetric and rigid $\mathbb{C}$-linear category with identity object $(A,\partial)$. The forgetful functor (which going to be a fibre functor in a Tannakian sense) is given as in subsection \ref{ssec:comatrix} by  the restriction of scalars functor  $\omega=\eta_{*}: \Diff{A}:=\cat{A}_{ U} \to add(A)$, where $\eta: A \to U$ is the canonical ring extension. Observe that $\omega$ is a non trivial functor in the sense of \cite{Deligne:1990} since we know that $\omega(A,\partial)=A$ and that ${\rm Spec}(A) \neq \emptyset$. 

Moreover,  $\Diff{A}$ is a full subcategory of the category $\rmod{U}$ stable under finite limits and colimits. Hence, it is abelian. We summarize the properties of $\Diff{A}$ in the following lemma.

\begin{lemma}\label{lema:abelian}
The category $\Diff{A}$ is $\mathbb{C}$-linear abelian and the functor $\omega: \Diff{A} \to add(A)$ is strict monoidal $\mathbb{C}$-linear faithful exact functor. Moreover, we have  an isomorphism of rings 
$
{\rm End}_{\Sscript{\Diff{A}}}\Big((A,\partial)\Big)\, \cong \, \mathbb{C}.
$
\end{lemma}

\begin{remark}\label{rem:separable}
Observe furthermore, that Lemma \ref{lema:abelian} and \cite[Proposition 2.5]{Bruguieres:1994} imply that $\Diff{A}$ is a locally finite category over $\mathbb{C}$, in the sense that each object is of finite length and the Hom-vector spaces are finite dimensional $\mathbb{C}$-vector spaces.  This in particular implies that the endomorphism ring of any differential module is a finite dimensional $\mathbb{C}$-algebra.  Therefore,  the centre of the endomorphism ring of any simple object in $\Diff{A}$, coincides with the base field $\mathbb{C}$. Thus, the category $\Diff{A}$ of differential modules  is a \emph{separable category} in the sense of \cite[D\'efinitions page 5847]{Bruguieres:1994}.
\end{remark}

\begin{remark}\label{rem:V}
Let $(M,\partial)$ be a differential module with ${\rm rank}(M)=m$. Denote by $M^{\partial}:=\ker(\partial)$ the $\mathbb{C}$-vector space of solutions in $A$ of the equation \eqref{Eq:ldm}. 
Then, there is an isomorphism of $\mathbb{C}$-vector spaces
$$
\hom{\Diff{A}}{(A,\partial)}{(M,\partial)} \,\cong\, M^{\partial}.
$$
Therefore, by Remark \ref{rem:separable}, we know that $dim_{\mathbb{C}}(M^{\partial}) < \infty$. In fact we have that $dim_{\mathbb{C}}(M^{\partial}) \leq m$.
\end{remark}

Apart from the above structure, the usual linear algebra operations are also permitted in the category of differential modules.  For instance, the exterior powers $\bigwedge^d M$  of a differential module $M$, are again  differential modules with  differential   given by 
$$ \partial \,( x_1 \wedge \dots \wedge x_d) \,=\, \sum_{i=1}^d  x_1 \wedge \dots \wedge \partial x_i \wedge \dots \wedge x_d,$$ 
such that the canonical linear map $ M^{\otimes^{d}}:=M\tensor{A} \cdots \tensor{A} M \longrightarrow \bigwedge^d M $ is a differential map, i.e.~a morphism in $\Diff{A}$.  
Similarly one can endows the symmetric $d^{\Sscript{th}}$-powers  ${\rm Sym}^{\Sscript{d}}(M)$ $A$-module with a differential, for a given differential module $(M,\partial)$. 

On the other hand the internal hom-functors of Eq.\eqref{Eq:RIH}, are explicitly given in this case by the hom-functors $\hom{A}{M}{N}$ whose differential is defined by the formula
\begin{equation}\label{Eq:partial-hom}
\partial l: M \longrightarrow N, \quad \Big(x \longmapsto (\partial l )(x) \,=\, \partial l(x) - l(\partial x) \Big) . 
\end{equation}
Thus, up to isomorphisms, we have 
$$
hom_{\rmod{U}}(M,N) = \hom{A}{M}{N}.
$$
In particular, the vector space $\hom{ U}{M}{N}$ is  identified with the linear maps whose differential is zero.  The differential of the dual module is given then by 
\begin{equation}\label{Eq:dual}
\partial: M^* \longrightarrow M^*,  \quad \Big( (\partial\varphi)(x) \,=\, \partial\varphi(x) - \varphi(\partial(x)) \Big),
\end{equation}
for every $\varphi \in M^{*}$ and $x \in M$.

\subsection{The commutative Hopf algebroid attached to the first Weyl algebra}\label{ssec:aag} 
Next we state some properties of the commutative Hopf algebroid $(A, U^{\bcirc})$ constructed from $ U$ by applying Theorem \ref{thm:lduality}. 
First, we know from Example \ref{exam:Dedekind}, that the canonical algebra map $\zeta:  U^{\bcirc} \to  U^*$ defined by equation \eqref{Eq:zeta} is injective, where $ U^*$ is the right linear dual of $ U$ endowed with the convolution product.  This fact will be implicitly used in the sequel.

\begin{proposition}\label{cor:UoGalois}
Let $A=\Cc[X]$ and $ U=A[Y,\partial /\partial X]$  its differential operator algebra. Then the commutative Hopf algebroid  $(A, U^{\bcirc})$ is a Galois $A$-coring. In particular the category $\Diff{A}$ is isomorphic to the full subcategory of right $U^{\bcirc}$-comodules $\cat{A}^{ U^{\bcirc}}$ and so ${\rm End}_{\Sscript{\rcomod{U^{\bcirc}}}}\Big( (A,\Sf{t})\Big) \cong  \mathbb{C}$ an isomorphism of rings.
\end{proposition}
\begin{proof}
The first statement is a  direct consequence of Proposition \ref{prop:R0Galois}. The particular cases are immediate from Proposition \ref {isomor} and  Lemma \ref{lema:abelian}.
\end{proof}

The subsequent contains further properties of the Hopf algebroid $(A,U^{\bcirc})$.

\begin{corollary}\label{cor:Transitive}
Let $A=\Cc[X]$ and $ U=A[Y,\partial /\partial X]$  its differential operator algebra. Then the commutative Hopf algebroid  $(A, U^{\bcirc})$ enjoys the following properties:
\begin{enumerate}[(i)]
\item The algebra map $\etaup^{\Sscript{\bcirc}}: A\tensor{\mathbb{C}}A \to U^{\bcirc}$    induces on $U^{\bcirc}$ a projective $(A\tensor{\mathbb{C}}A)$-module structure.  
\item There is an isomorphism of symmetric monoidal $\mathbb{C}$-linear categories $\Diff{A}\cong \frcomod{U^{\bcirc}}$,
where  $\frcomod{U^{\bcirc}}$ denotes the full subcategory  of right $U^{\bcirc}$-comodules with finitely generated underlying $A$-modules.
\item  Every right $U^{\bcirc}$-comodule is projective and faithfully flat as an $A$-module. In particular, the modules $U^{\bcirc}_A$ and ${}_{A}U^{\bcirc}$ are projective.
\item Any right $U^{\bcirc}$-comodule is a filtered limit of subcomodules in $\frcomod{U^{\bcirc}}$.
\end{enumerate}
\end{corollary}
\begin{proof}
$(i)$, $(ii)$. We know from subsection \ref{ssec:Rcero}, that $(A,U^{\bcirc})$ was constructed from the pair $(\cat{A}_{U}, \Sf{s}_{*})$, or up to the isomorphism of Proposition \ref{cor:UoGalois}, from the pair $(\Diff{A},\omega)$ of Lemma \ref{lema:abelian}. By applying Deligne's Theorem \cite[Th\'eor\`emes 5.2, 7.1]{Bruguieres:1994}, we have from one hand that the functor $\omega$ induces the equivalence sated in $(ii)$ which proves this item. 
From another hand, we have that the Hopf algebroid $(A,U^{\bcirc})$ is transitive 
and also separable by Remark \ref{rem:separable}. Therefore,  it is geometrically transitive  by \cite[Corollaire 6.7]{Bruguieres:1994}, which by definition \cite[D\'efinition page 5845]{Bruguieres:1994} means that $U^{\bcirc}$ is a projective $(A\tensor{\mathbb{C}}A)$-module, and this shows part $(i)$. 

$(iii)$. It follows directly from \cite[Proposition 7.2]{Bruguieres:1994}, while item $(iv)$ follows from \cite[Proposition 6.2]{Bruguieres:1994}.
\end{proof}

\begin{remark}\label{rem:Uo} 
 By \cite[Theorem 5.7]{ElKaoutit/Gomez:2004b} and Corollary \ref{cor:Transitive}\emph{(iii)}, the category of right comodules over the Hopf algebroid  $(A,  U^{\bcirc})$ admits $\Diff{A}$ as a generating set of small projectives if and only if the direct sum $\oplus_{M \in \, \Diff{A}} M$ is a faithfully flat right module over the ring $\oplus_{M,\,N\, \in \, \Diff{A}}\hom{}{M}{N}$ with enough orthogonal idempotents (i.e.~Gabriel's ring of $\Diff{A}$). Both equivalent conditions are fulfilled by 
Corollary \ref{cor:Transitive}\emph{(iv)}.
\end{remark}

\subsection{The differential Galois groupoid of a differential module.}\label{ssec:PV}  
Fix a differential module $M \in\Diff{A}$ with a dual basis $\{e_i, e_i^{*}\}_{1 \leq i \leq m}$.  We set
\begin{equation}\label{Eq:dete}
det(e_1,\dots,e_m) \,:=\, \sum_{\sigma \,\in\, S_m} (-1)^{sg(\sigma)} e_{\sigma(m)} \tensor{A} \dots \tensor{A} e_{\sigma(1)} \,\,  \in \, M^{\tensfun{}^m},
\end{equation}
where $S_m$ is the permutation group of $m$ elements, $sg(\sigma)$ is the signature of $\sigma$ and  $M^{\otimes^m}:=M\tensor{A}\cdots\tensor{A}M$ denotes the $m$-fold tensor product  of the $A$-module $M$.
We also denote 
\begin{equation}\label{Eq:detM}
det_M \, : =\, \bara{ e_{m}^* \star \dots \star e_1^* \tensor{T_{M^{\otimes^m}}} det(e_1,\dots, e_m) } \,\, \in\,\,  U{}^{\bcirc},
\end{equation}
where, as in equation \eqref{Eq:star}, the $A$-linear map $e_{m}^* \star \dots \star e_1^*$ is defined by 
$$
e_{m}^* \star \dots \star e_1^*: M^{\otimes^{m}} \longrightarrow A,\quad \Big( x_{1}\tensor{A}\cdots \tensor{A}x_{m} \longmapsto e_{m}^*(x_{m}) \cdots e_1^*(x_{1}) \Big).
$$

\begin{lemma}\label{lema:det}
Keeping the previous notations, we then  have 
\begin{enumerate}[(i)]
\item As an element in $M^{\tensfun{}^m}$, the differential of $det(e_1,\dots,e_m)$ is
$$\partial \,det(e_1,\dots,e_m) \,=\,tr\big(mat(M)\big) \,det(e_1,\dots,e_m),
$$ where $tr\big(mat(M)\big)$ is the trace of $mat(M)$;
\item $(e_{m}^* \star \dots \star e_1^*) \Big( det(e_1,\dots,e_m)\Big)\,\,=\,\, 1$.
\end{enumerate}
\end{lemma}
\begin{proof}
Both part $(i)$ and $(ii)$ are routine computations by using the formulae \eqref{Eq:ptensor} and definitions. The details are left to the reader. 
\end{proof}

A crucial consequence of this lemma is the subsequent.
\begin{lemma}\label{lema:det-1}
The element $det_M $ of Eq.\eqref{Eq:detM} is invertible  in the algebra $U{}^{\bcirc}$.
\end{lemma}
\begin{proof}
Let $\bigwedge^m M^*$ be the $m$-exterior power of the dual module $M^*$ of $M$. This is a free $A$-module of rank one and basis $e_1^* \wedge \dots \wedge e_m^*$ with differential $\partial(e_1^* \wedge \dots \wedge e_m^*)\,=\,- tr(\mat{M}) e_1^* \wedge \dots \wedge e_m^*$, where as before $tr(\mat{M})$ denotes the trace of the matrix $\mat{M}$.  Thus, the module $\bigwedge^m M^*$ is an object in the category $\Diff{A}$. By Lemma \ref{lema:det}(i), we define the following differential map 
$$ f : A \longrightarrow  ({\wedge} ^m M^*)\tensor{A} M^{\otimes_m},\quad \Big(\,1 \longmapsto (e_1^* \wedge \dots \wedge e_m^*) \tensor{A} det(e_1,\dots,e_m) \, \Big),$$ which we consider as a morphism in the category $\Diff{A}$.

Now let us check that  
\begin{equation}\label{Eq:det-1}
det_M^{-1} \,=\, \left(\bara{e_m^* \star \dots \star e_1^* \tensor{T_{M^{\otimes_m}}}  det(e_1,\dots, e_m) }\right) ^{-1} \,=\,  \bara{(e^*_1 \wedge \dots \wedge  e_m^*)^* \tensor{T_{\bigwedge^m M^*}} (e^*_1 \wedge \dots \wedge  e_m^*)}. 
\end{equation}
So we compute their multiplication:
\begin{eqnarray*}
& &  \bara{e_m^* \star \dots \star e_1^* \tensor{T_{M^{\otimes_m}}}  det(e_1,\dots, e_m) } \; . \; \bara{(e^*_1 \wedge \dots \wedge  e_m^*)^* \tensor{T_{\bigwedge^m M^*}} (e^*_1 \wedge \dots \wedge  e_m^*)}  \\ \;\; &=& \bara{ \Big( (e_1^* \wedge \dots \wedge e_m^*)^{*} \star (e_m^* \star \dots \star e_1^*) \Big) \tensor{T_{({\wedge} ^m M^*)\tensor{A} M^{\otimes_m}}} \Big(  (e_1^* \wedge \dots \wedge e_m^*) \tensor{A} det(e_1,\dots,e_m) \Big) } \\ \;\; &=& \bara{ \Big( (e_1^* \wedge \dots \wedge e_m^*)^{*} \star (e_m^* \star \dots \star e_1^*) \Big) \tensor{T_{({\wedge} ^m M^*)\tensor{A} M^{\otimes_m}}} \Big(  f(1) \Big) } \\ \;\; &=& \bara{ \Big( (e_1^* \wedge \dots \wedge e_m^*)^{*} \star (e_m^* \star \dots \star e_1^*) \Big) f \tensor{T_A}  1}.
\end{eqnarray*}
The $A$-linear map $\Big( (e_1^* \wedge \dots \wedge e_m^*)^{*} \star (e_m^* \star \dots \star e_1^*) \Big) f  \in A^{*}$ is defined by 
\begin{eqnarray*} 
& & \; \Big( (e_1^* \wedge \dots \wedge e_m^*)^{*} \star (e_m^* \star \dots \star e_1^*) \Big) \big( f (1) \big) \\ &=& \Big( (e_1^* \wedge \dots \wedge e_m^*)^{*} \star (e_m^* \star \dots \star e_1^*) \Big)  \Big( (e_1^* \wedge \dots \wedge e_m^*) \tensor{A} det(e_1,\dots,e_m)\Big) \\ &=&  (e_1^* \wedge \dots \wedge e_m^*)^{*} \Big(e_1^* \wedge \dots \wedge e_m^* \Big)\; (e_m^* \star \dots \star e_1^*) \Big(det(e_1,\dots,e_m) \Big) \\ &\overset{\ref{lema:det}(ii)}{=}& 1, 
\end{eqnarray*}
which shows that 
$$ \bara{e_m^* \star \dots \star e_1^* \tensor{T_{M^{\otimes_m}}}  det(e_1,\dots, e_m) } \; .\; \bara{(e^*_1 \wedge \dots \wedge  e_m^*)^* \tensor{T_{\bigwedge^m M^*}} (e^*_1 \wedge \dots \wedge  e_m^*)} \,=\,1$$  and this finishes the proof.
\end{proof}

\begin{remark}\label{rem:gl}
Similar to the context of commutative Hopf algebras, a grouplike element $g$ in a commutative Hopf algebroid is always an invertible element with inverse $\Scr{S}(g)$, its image by the antipode. In the particular situation of Lemma  \ref{lema:det-1}, 
it is easily seen from equation \eqref{Eq:det-1} that 
\begin{eqnarray*}
\Delta\big(det_M^{-1}\big) &=& \bara{(e^*_1 \wedge \dots \wedge  e_m^*)^* \tensor{T_{\bigwedge^m M^*}} (e^*_1 \wedge \dots \wedge  e_m^*)}\, \tensor{A} \, \bara{(e^*_1 \wedge \dots \wedge  e_m^*)^* \tensor{T_{\bigwedge^m M^*}} (e^*_1 \wedge \dots \wedge  e_m^*)} \\ &=& det_M^{-1} \tensor{A} det_M^{-1}.
\end{eqnarray*} 
which shows that $det_M^{-1}$ is a groulike element, and so is $det_M$. Therefore, we have that 
\begin{equation}\label{Eq:Sdet}
\Scr{S}(det_{M})=det_{M}^{-1}.
\end{equation}

On the other hand, the structure of right $U^{\bcirc}$-comodule over $A_{A}$ which corresponds to this grouplike element  and, up to an isomorphism, is the right $U^{\bcirc}$-comodule structure of $\bigwedge^m M^*$ deduced by applying Proposition \ref{cor:UoGalois}, is given by the coaction $A \to U^{\bcirc}$, $a \mapsto det_{M}^{-1}a$.  This corresponds to the differential module $(A_{A}, \partial^{\Sscript{M}})$ with differential 
\begin{equation}\label{Eq:glike}
\partial^{\Sscript{M}}: A \longrightarrow A, \quad \Big( a \longmapsto -tr\big(mat(M)\big)\, a + \partial a\Big).
\end{equation}
\end{remark}

For a given $M$, let $ \overset{\bcirc}{U}_{M}$ be the $(A \tensor{\mathbb{C}} A)$--subalgebra of $U^{\bcirc}$ generated by the image of $M^* \otimes_{T_M} M$ via the canonical map $M^* \otimes_{T_M} M \to U^{\bcirc}$. It is not hard to see that   $\overset{\bcirc}{U}_{M}$ is generated as an $(A \tensor{\mathbb{C}} A)$--algebra  by the set $ \Big\{ \bara{e_{j}^{*}\tensor{T_{M}}e_{i}}\Big\}_{1\leq i,j \leq m }$ for any  given a dual basis $\{e_i, e_i^*\}_{1\leq i \leq m}$ as above. 

\begin{proposition}\label{prop:det}
The subalgebra $\overset{\bcirc}{U}_{M}$ is in fact a sub-bialgebroid of $U^{\bcirc}$, which contains the  element $det_{M}$ defined in Eq.~\eqref{Eq:detM}. Furthermore, its localization algebra $\overset{\bcirc}{U}_{M}[det_{M}^{-1}]$ is a Hopf subalgebroid of $U^{\bcirc}$. 
\end{proposition}
\begin{proof}
The first statement is immediate, since we already know that 
$$
\Delta\big( \bara{e_{j}^{*}\tensor{T_{M}}e_{i}} \big) = \sum_{k} \bara{e_{j}^{*}\tensor{T_{M}}e_{k}} \tensor{A} \bara{e_{k}^{*}\tensor{T_{M}}e_{i}},
$$
where $\Delta$ denotes the comultiplication of $U^{\bcirc}$. Since $\Delta$ is already multiplicative, we get a coassociative and multiplicative $A$--bimodule map $\Delta : \overset{\bcirc}{U}_{M} \to \overset{\bcirc}{U}_{M} \tensor{A} \overset{\bcirc}{U}_{M}$. Observe that we are using that $A$ is a principal ideal domain and that $U^{\bcirc}$ is torsionfree over $A$ on both sides.    
The fact that $det_{M} \in \overset{\bcirc}{U}_{M}$ clearly follows from the formula of the multiplication in $U^{\bcirc}$ displayed in \eqref{Eq:mo}.

Next we show that the antipode of $U^{\bcirc}$ restricts to $\overset{\bcirc}{U}_{M}[det_{M}^{-1}]$, since we already have by Lemma \ref{lema:det-1} the inclusion $\overset{\bcirc}{U}_{M}[det_{M}^{-1}] \subseteq U^{\bcirc}$. 
Denote by $u_{ij}$ the generating elements $\bara{e_i^*\tensor{T_M}e_j}$, for  $i,j=1,\dots,m$, and by $det_{u}$ the determinant $det_{M}$ of Eq.\eqref{Eq:detM}.  

Recall that the  $(j,i)^{th}$ cofactor of the matrix $(u_{ij})_{1 \leq i, j \leq m}$ is then defined to be 
$$ v_{ji}\,=\, (-1)^{i+j} \sum_{\sigma \,\in\, \cat{P}(j,i)} (-1)^{sg(\sigma)} u_{j_1\sigma(j_1)} \dots u_{j_{m-1}\sigma(j_{m-1})},$$ where $\cat{P}(j,i)$ is the set of all bijections from $\{1,\dots,m\} \setminus\{j\}=\{j_1< \dots < j_{m-1}\}$ to  $\{1,\dots,m\} \setminus\{i\}$, and $sg(\sigma)$ is the signature of $\sigma$ viewed as a permutation in $S_{m-1}$.  Clearly, each of the $v_{ji}$ is an element in the $(A\tensor{}A)$-subalgebra $\overset{\bcirc}{U}_{M}$.
Writing down this element using the notation of Eq.\eqref{Eq:dete}, we get
$$ v_{ji} \,=\, (-1)^{i+j} \,\bara{ e_{j_{m-1}}^*\star \dots \star e_{j_1}^* \tensor{T_{M^{\otimes_{m-1}}}} det(e_1,\dots,\overset{\curlyvee}{e_i},\dots, e_m)},$$
where $det(e_1,\dots,\overset{\curlyvee}{e_i},\dots, e_m)$ is the resulting determinant of the vector $(e_1,\dots, e_m)$ after removing the component $e_i$, see Eq.\eqref{Eq:dete}.

Define now the following morphism in the category $\Diff{A}$, by the  composition 
$$\xymatrix@C=45pt{h^{m}: M^* \ar@{->}^-{db^{\otimes_{m-1}}\tensor{}M^*}[r] & (M\tensor{}M^*)^{\otimes_{m-1}} \tensor{}M^* \ar@{->}[r] & M^{\otimes_{m-1}} \tensor{} (M^*)^{\otimes_{m}} \ar@{->>}[r] &  M^{\otimes_{m-1}} \tensor{} \Big(\bigwedge^m M^*\Big),}$$ where $db: A \to M\tensor{A}M^*$ is the coevaluation map of Eq.\eqref{Eq:ev-db} up to the differential isomorphism $M\tensor{A}M^* \cong \hom{A}{M}{M}$, the second map is the flip, and the third one is obvious.    On elements, $h^m$ acts by  
$$
 h^m(\varphi)\,=\, \sum_{k_1,\dots,\, k_{m-1}} (e_{k_{m-1}}\tensor{A}\dots\tensor{A}e_{k_1}) \tensor{A} (e_{k_1}^*\wedge \dots \wedge e_{k_{m-1}}^* \wedge \varphi), 
 $$
 for every $\varphi \in M^*$,  where each of the index $k_l$ runs the set $\{1,\dots,m\}$, for each of the $l=1, \cdots, m-1$. Reordering the indices, we obtain 
\begin{equation}\label{Eq:hm1}
h^m(e_i^*)\,=\, (-1)^i det(e_1,\dots,\overset{\curlyvee}{e_i},\dots, e_m) \tensor{A}(e_1^*\wedge \dots \wedge e_m^*). 
\end{equation}
Using this equality and the above notation, we show that 
\begin{equation}\label{Eq:hm2}
\gamma(e_j) \,= \, \left[(-1)^j  e_{j_{m-1}}^*\star \dots \star e_{j_1}^*  \star (e_1^*\wedge \dots \wedge e_m^*)^*\right] \, \circ \, h^m,
\end{equation}
where as above $\gamma: M \to (M^*)^*$ is the canonical isomorphism. Now,
we compute
\begin{eqnarray*}
det_u^{-1} \, .\, v_{ji} &\overset{\eqref{Eq:det-1}}{=}& \bara{(e^*_1 \wedge \dots \wedge  e_m^*)^* \tensor{T_{\bigwedge^m M^*}} (e^*_1 \wedge \dots \wedge  e_m^*)} \;
(-1)^{i+j} \, \bara{ e_{j_{m-1}}^*\star \dots \star e_{j_1}^* \tensor{T_{M^{\otimes_{m-1}}}} det(e_1,\dots,\overset{\curlyvee}{e_i},\dots, e_m)}  \\
&\overset{\eqref{Eq:hm1}}{=}& (-1)^j \, \bara{ e_{j_{m-1}}^*\star \dots \star e_{j_1}^*\star (e^*_1 \wedge \dots \wedge  e_m^*)^* \tensor{T_{M^{\otimes_{m-1}}\tensor{}\wedge^m M^*}} h^m(e_i^*) } 
\\
&=& (-1)^j \, \bara{ e_{j_{m-1}}^*\star \dots \star e_{j_1}^*\star (e^*_1 \wedge \dots \wedge  e_m^*)^* h^m \tensor{T_{M^*}} e_i^* }) \\ &\overset{\eqref{Eq:hm2}}{=} & \bara{\gamma(e_j) \tensor{T_{M^*}}e_i^* } \,\, =\,\, \Scr{S}_M(u_{ij}).
\end{eqnarray*} Thus, we have 
\begin{equation}\label{Eq:S}
det_u^{-1} \, .\, v_{ji} \, \, =\,\, \Scr{S}_M(u_{ij}).
\end{equation}
This with Eq.\eqref{Eq:Sdet} show that $\Scr{S}$ is restricted to the localization $\overset{\bcirc}{U}_{M}[det_{M}^{-1}]$, and this completes the proof. 
\end{proof}

From now on, we  denote by $U^{\bcirc}_{\Sscript{(M)}}:=\overset{\bcirc}{U}_{M}[det_{M}^{-1}]$ the Hopf sub-algebroid of $U^{\bcirc}$ stated in Proposition \ref{prop:det}.

Fix a differential module $(M,\partial) \in\Diff{A}$ with a dual basis $\{e_i,e_i^*\}_{1 \leq i \leq m}$, and keep the above notations. Given two positive integers $k,l$, we denote by $T^{\Sscript{(k,\, l)}}(M):= M^{\scriptstyle{\tensor{}l}} \tensor{A} (M^*)^{\scriptstyle{\tensor{}k}}$ which we consider as a differential module using the tensor product of Eq.\eqref{Eq:ptensor}. 
We denote by $\langle M \rangle_{\Sscript{\otimes}}$ the full sub-category of $\Diff{A}$ of finite sub-quotients differential modules of $M$. Thus, an object $(X,\partial)$ of $\Diff{A}$ belongs to $\Mo$ if it is a quotient of the form $X=X_2/X_1$, where $X_1 \subseteq X_2 \subseteq \oplus_{k,\, l} T^{\Sscript{(k,\, l)}}(M)$ (finite direct sum). Since $\Diff{A}$ is by Lemma \ref{lema:abelian} an abelian category, a differential module $(X,\partial)$ belongs to $\Mo$ if and only if it is a sub-object of an object finitely generated by those $T^{\Sscript{(k,\, l)}}(M)$'s.

Let us denote by $\omegam: \Mo \longrightarrow add(A)$ the restriction of the fibre functor $\omega: \Diff{A} \longrightarrow add(A)$, and by $(A,\rR(\Mo))$ its associated commutative Hopf algebroid defined in Proposition \ref{prop:TBialgd}, see also Eq. \eqref{comatInf} and Remark \ref{rem:generalcase}\footnote{In the notations of subsection \ref{ssec:PB}, this is  $\rR(\Mo)=\RAA{\omega}$.}. Since $\Mo$ is a symmetric rigid monoidal $\mathbb{C}$-linear abelian, locally finite category and $\omegam$ is an exact faithful functor, we have that $(A,\rR(\Mo))$ enjoys similar properties $(i)\text{-}(iv)$ of Corollary \ref{cor:Transitive}. Thus, $(A,\rR(\Mo))$ is a geometrically transitive Hopf algebroid, see Example \ref{exam:GT} and \cite{Bruguieres:1994, ElKaoutit:2015}.   In particular we have that the fibre functor $\omegam$ induces a symmetric monoidal equivalence of categories $\Mo \simeq \frcomod{\rR(\Mo)}$ to the category of $\rR(\Mo)$-comodules with finitely generated underlying $A$-modules.
On the other hand, the embedding $\Mo \hookrightarrow \Diff{A} = \cat{A}_{U}$ which commutes with the forgetful functors, leads to the canonical map
\begin{equation}\label{Eq:Phi}
\phiup: (A,\rR(\Mo)) \longrightarrow (A,U^{\bcirc}), \quad\lr{p^*\tensor{T_P}p +\cat{J}_{\Mo} \longmapsto  p^*\tensor{T_P}p +\cat{J}_{\cat{A}_{ U}}}.
\end{equation}

The following is our main result of this subsection.

\begin{theorem}\label{theo:unpo}
The morphism displayed in equation \eqref{Eq:Phi}, induces an isomorphism $\rR(\Mo) \cong \Um$ of Hopf algebroids, where $\Um$ is as in Proposition \ref{prop:det}. Consequently, we get that the equivalence of categories stated in Corollary \ref{cor:Transitive}(ii), rectricts to a symmetric monoidal $\mathbb{C}$-linear  equivalence of categories:
$$
\chi: \Mo \longrightarrow \frcomod{\Um},
$$
where $\frcomod{\Um}$ is the full subcategory of $\Um$-comodules with  finitely generated underlying $A$-modules.
\end{theorem}
\begin{proof} 
The map $\phiup$ is clearly an injective morphism of Hopf algebroids. We then identify $\rR(\Mo)$ with its image.
Consider now  the $(A\tensor{\mathbb{C}}A)$--subalgebra $\overset{\bcirc}{U}_{M}$  of $U^{\bcirc} $ described in Proposition \ref{prop:det}. Any generic element in $\overset{\bcirc}{U}_{M}$ of the form $\bara{e_{j}^{*}\tensor{T_{M}}e_{i}}$ obviously belongs to the subalgebra $\rR(\Mo)$, whence  $\overset{\bcirc}{U}_{M} \subseteq \rR(\Mo)$, as $M$ is a differential module in the subcategory $\Mo$. 
 The determinant $det_{M}$ of equation \eqref{Eq:detM} and its inverse given by equation \eqref{Eq:det-1}, are both elements in $\rR(\Mo)$. Therefore, $\Um =\overset{\bcirc}{U}_{M}[det_{M}^{-1}] \subseteq \rR(\Mo)$. 

Conversely, take an object $N$ in $\Mo$ of the form $N=M_{\Sscript{(k,\, l)}}$, for some positive integers, $k,l$. Consider $(N,\varrho_{\Sscript{N}})$ as a right $U^{\bcirc}$-comodule, via the isomorphism $\chi$ (see subsection \ref{ssec:xiyL}).  Denote by 
$$\cat{C}(N) \,:=\, \sum_{f \, \in \, \hom{U^{\bcirc}}{N}{U^{\bcirc}}} {\rm Im}(f)$$
the $A$-subbimodule of coefficients of $N$. By Corollary \ref{cor:Transitive}$(iii)$, $U^{\bcirc}_A$ and ${}_{A}U^{\bcirc}$ are projective modules,  and so by  \cite[Example 2.5]{ElKaoutit/Gomez/Lobillo:2004}, we can apply \cite[Proposition 2.12]{ElKaoutit/Gomez/Lobillo:2004}. Therefore,  $\cat{C}(N) \subseteq \Um$, since we have $\varrho_N(N) \subseteq (N\tensor{A}\tau)(N\tensor{A}\Um)$, where $\tau: \Um \hookrightarrow U^{\bcirc}$ is the canonical injection.  The same inclusion holds true: $\cat{C}(N) \subseteq \Um$,  if we take $N =\oplus_{\Sscript{(k,\, l)}} T^{\Sscript{(k,\, l)}}(M)$ a finite direct sum. Taking a sequence $X_1 \subseteq X_2 \subseteq N$ of differential modules, with $N$ as before, we get by \cite[Proposition 2.12]{ElKaoutit/Gomez/Lobillo:2004} that 
$$
\cat{C}(X_1) \subseteq \cat{C}(X_2) \subseteq \cat{C}(N) \subseteq \Um,\quad \cat{C}(X_2/X_1) \subseteq \cat{C}(X_2) \subseteq \cat{C}(N) \subseteq \Um.
$$
Henceforth, any object in $\Mo$ is a right $\Um$-comodule (under the isomorphism $\chi$). Now, take an element of the form $\bara{p^{*}\tensor{T_{P}}p} \in \rR(\Mo)$, thus $P$ is an object in $\Mo$. We know that $P$ is a right $\Um$-comodule, and so $p^{*}(p_{(0)})p_{(1)} \in \Um$.  Up to the canonical isomorphism $\can{U^{\bcirc}}$ of equation \eqref{Eq:can},  we have the equality $\bara{p^{*}\tensor{T_{P}}p} = p^{*}(p_{(0)})p_{(1)}$, which shows that $\bara{p^{*}\tensor{T_{P}}p} \in \Um$. Therefore, $\rR(\Mo) \subseteq \Um$. 

In summary, we have an isomorphism $(A,\rR(\Mo)) \cong (A,\Um)$ of Hopf algebroids. This leads to  the following commutative diagram of symmetric monoidal categories:
$$
\xymatrix@R=15pt{  & \Mo \ar@{^{(}->}[rr]    \ar@{->}^-{\simeq }[ld] \ar@{-->}^-{\chi}_-{\cong}[dd] & &  \Diff{A}=\cat{A}_U \ar@{->}^-{\chi}_-{\cong}[dd] \\ \frcomod{\rR(\Mo)} \ar@{->}^-{\simeq }[rd] & & & \\ & \frcomod{\Um}  \ar@{^{(}->}[rr] & & \frcomod{U^{\bcirc}}  }
$$
which shows the stated equivalence of monoidal categories. 
\end{proof}

Let $(M,\partial)$ be a differential module over $A$ and consider the attached Hopf algebroid $(A,\Um)$ of Theorem \ref{theo:unpo}. Denote by $\hHm: \Algc{} \to \Grpds$ its associated presheaf  of groupoids and let $\hHm(\mathbb{C})$ be its  fibre at $\text{Spec}(\mathbb{C})$, that is,  the  \emph{character groupoid} of $(A,\Um)$.  We know that $(A,\Um)$ is a finitely generated Hopf algebroid (i.e., both $A$ and $\Um$ are finitely generated $\mathbb{C}$-algebras), thus, $\hHm$ is an affine algebraic $\mathbb{C}$-groupoid scheme in the fpqc (fid\`element plate quasi-compacte) topology. Moreover,  the category of comodules $\frcomod{\Um}$ is identified with the category of $\mathbb{C}$-representations of the groupoid $\hHm$, see \cite{Deligne:1990}. 

In comparison with the classical case of differential Galois theory of fields (see \cite[Theorem 2.33]{VanderPut/Singer}), we have  by Theorem \ref{theo:unpo} the following definition:
\begin{definition}\label{def:GDG}
Let $(A, \Um)$ and $\hHm$ be as above. Then the character groupoid $\hHm(\mathbb{C})$ is referred to  as a \emph{differential Galois algebraic groupoid} of $(A,\partial)$  for  the differential module $(M,\partial)$ (see Remark \ref{rema:final} below for a brief discussion on the uniqueness of this groupoid). 
\end{definition}

Next, we observe that, always for $A=\Cc[X]$,  any differential Galois groupoid is transitive as an abstract groupoid. First we observe that, for any differential $A$-module $(M,\partial)$,  $\Um$ as a  Hopf algebroid over $\Cc$,  satisfies similar conditions $(i)$-$(iv)$ of Corollary \ref{cor:Transitive}. Thus, $\Um$  is a geometrically transitive Hopf algebroid over $\mathbb{C}$, so that $\hHm$ is  a \emph{transitive groupoid scheme} in the sense of \cite{Deligne:1990}, see also \cite{Bruguieres:1994, ElKaoutit:2015}. Since $\hHm(\mathbb{C}) \neq \emptyset$, we have by \cite[Corollary B]{ElKaoutit:2015} that $\hHm(\mathbb{C})$ is  a transitive  groupoid in the set-theoretical sense (i.e., a groupoid in sets with only one connected component, or equivalently, the Cartesian product of the source and the target maps leads to a surjective map).  In other words,  $\hHm(\mathbb{C})$ posses only one type of isotropy group.  In the sequel we will show that the isotropy type group of  $\hHm(\mathbb{C})$ coincides with the so called  \emph{differential Galois algebraic group} of the differential module $(M,\partial)$  following the terminology of Y. Andr\'e \cite[\S 3.2.1.2]{Andre:2001}, (see Example \ref{Exm:dimOne} for more terminologies).  

Now, we proceed to show  that $\hHm(\mathbb{C})$ is in fact a sub-groupoid of the induced algebraic groupoid of the general linear group $GL_m(\mathbb{C})$ along the map $\mathbb{A}_{\Sscript{\mathbb{C}}}^1 \to \{*\}$ (the algebraic set with one point),   where $m$ is the rank of $M$.  So let  $\{e_i,e_i^*\}_{1 \, \leq i \leq \, m}$ be a  dual basis for $M$. Consider the coordinate ring $\mathbb{C}[X_{ij}, det_X^{-1}]$  of the complex general linear group of order $m$  with indices $i,j \in \{1,\cdots,m\}$. Recall that this a Hopf $\mathbb{C}$-algebra with structure maps:
\begin{equation}\label{Eq:Xij}
\Delta(X_{ij})\,=\, \sum_{k=1}^m X_{ik}\tensor{A}X_{kj}, \quad \varepsilon(X_{ij})\,=\, \delta_{ij}, \quad S(X_{ij})\,=\, det_X^{-1} Y_{ji},
\end{equation}
where $\delta_{ij}$ is Kronecker's symbol, and  where $Y_{ji}$ is the $(j,i)^{th}$ cofactor of the matrix $(X_{ij})$.  This is also a differential $\mathbb{C}$-algebra, where the differential is given by 
\begin{equation}\label{Eq:partialX}
\partial (X_{ij})_{1 \leq,i, j \leq m}\,:=\, (X_{ij}^{\prime})_{1 \leq,i, j \leq m}\,=\, \mat{M} \, (X_{ij})_{ 1 \leq,i, j \leq m},
\end{equation}
and the differential of $\det_X^{-1}$ is by extension  $\partial(det_X^{-1})=-\partial(det_X)det_X^{-2}$.  Extending this differential to  the tensor product algebra $A\tensor{\mathbb{C}}\mathbb{C}[X_{ij}, det_X^{-1}]$ via the map:
$$
\partial\big( a\tensor{\mathbb{C}} X_{ij}\big)\,=\, \partial a \tensor{\mathbb{C}} X_{ij} \,+\, a \tensor{\mathbb{C}} \partial X_{ij} 
$$
we will  consider this tensor algebra  as a differential algebra and obviously an extension of  $(A,\partial)$.

The base extension of the Hopf algebra $\Cc[X_{ij}, det_X^{-1}]$ via the algebra map $\mathbb{C} \to A$ (see Example \ref{exam:Bext}),  leads to the commutative Hopf algebroid $A\tensor{\mathbb{C}}\mathbb{C}[X_{ij}, det_X^{-1}]\tensor{\mathbb{C}}A$ which is isomorphic to the  polynomial $(A\tensor{\mathbb{C}}A)$-algebra $(A\tensor{\mathbb{C}}A)[X_{ij}, det_X^{-1}]$. Furthermore, there is a surjective map:
\begin{equation}\label{Eq:Phi1}
\xymatrix@R=0pt{  \phiup_{\Sscript{M}}: (A\tensor{}A)[X_{ij}, det_X^{-1}]  \ar@{->}[rr] & & \Um \\ 
X_{ij} \ar@{|->}[rr] & & \bara{e_j^*\tensor{T_M}e_i}, }
\end{equation}
of commutative Hopf algebroids. The affine algebraic groupoid attached to the  Hopf algebroid $(A\tensor{}A)[X_{ij}, det_X^{-1}]$ is also transitive. Thus, $(A, A\tensor{}A)[X_{ij}, det_X^{-1}])$ is geometrically transitive Hopf algebroid over $\Cc$, see \cite[Theorem A]{ElKaoutit:2015}.  Its character groupoid is easily computed and it is given by:
$$
\xymatrix@C=45pt{ \gG^{\Sscript{m}}: \;\; \mathbb{A}_{\Sscript{\mathbb{C}}}^1 \,\times \, GL_m(\mathbb{C}) \,\times  \mathbb{A}_{\Sscript{\mathbb{C}}}^1 \ar@<1ex>@{->}|-{\scriptscriptstyle{pr_3}}[r] \ar@<-1ex>@{->}|-{\scriptscriptstyle{pr_1}}[r] & \ar@{->}|-{ \scriptscriptstyle{\iota}}[l] \mathbb{A}_{\Sscript{\mathbb{C}}}^1, }
$$ 
where the source and target are, respectively, $pr_3$ and $pr_1$ the third and first projections, the identity map sends $x \to \iota(x)=(x,I_m,x)$, where $I_m$ is the identity of $GL_m(\mathbb{C})$. The multiplication and the inverse maps  of this groupoid are given by:
$$
(x,\fk{a},y)\, . \, (y,\fk{b}, z) \,=\,  (x,\fk{a}\fk{b},z),\quad (x,\fk{a},y)^{-1}\,=\, (y,\fk{a}^{-1},x).
$$
This is clearly a transitive groupoid, and so there is only one type of isotropy group, namely,  each of them  is isomorphic to the general linear group $GL_m(\mathbb{C})$. The groupoid $\gG^{\Sscript{m}}$ is in fact the induced groupoid of the groupoid  (with only one object) $GL_m(\mathbb{C})$ along the map $\mathbb{A}_{\Sscript{\mathbb{C}}}^1 \to \{*\} $, see \cite[Example 2.4]{ElKaoutit:2015}.

Now, using the morphism of equation \eqref{Eq:Phi1} we claim the following corollary, where  the particular claim can be compared  with  \cite[Th\'eor\`eme 3.2.1.1(ii)]{Andre:2001}.

\begin{corollary}\label{coro:injection}
There is a monomorphism of affine algebraic groupoids
$\xymatrix{\hHm(\mathbb{C})\, \ar@{^{(}->}[r] &  \gG^{\Sscript{m}} }$. 
In particular, any  isotropy group of $\hHm(\mathbb{C})$ is identified with a closed sub-group of the algebraic group $GL_m(\mathbb{C})$.
\end{corollary}

Next, we discuss the uniqueness of the  differential Galois groupoid and give a motivating example. 

\begin{remark}\label{rema:final} 
Keep the above notations and assume that there is a flat Hopf algebroid $(A,\cat{H})$ over $\Cc$ with $\cH$ finitely generated as an algebra,  such that there is a symmetric monoidal $\mathbb{C}$-linear equivalence of categories 
$\Mo \simeq \frcomod{\cat{H}}$ to the category of right $\cat{H}$-comodules with free of finite rank underlying $A$-modules. Assume further that $(A,\cat{H})$ is geometrically transitive (see Example \ref{exam:GT} and \cite{ElKaoutit:2015}, or equivalently  that the associated affine $\mathbb{C}$-groupoid scheme $\hH$ is transitive). Then by Theorem \ref{theo:unpo} in conjunction with Deligne's Theorem \cite[Theorem 5.2]{Bruguieres:1994}, there is a symmetric monoidal $\mathbb{C}$-linear equivalence of categories $\rcomod{\cat{H}} \simeq \rcomod{\Um}$ of all comodules (i.e., the extension of the previous equivalence to the ind-objects categories). Therefore by \cite[Theorem A]{Kaoutit/Kowalzig:2014}, there is a two-stage zig-zag of weak equivalences connecting $(A,\cat{H})$ and $(A,\Um)$, that is, they are weakly equivalent Hopf algebroids. Therefore their character groupoids should be also weakly equivalent, meaning that  $\hHm(\mathbb{C})$ and $\hH(\mathbb{C})$ are weakly equivalent affine algebraic groupoids. This means that, in contrast with the classical case of differential Galois theory where the Galois group is unique up to isomorphisms,  in this framework the differential Galois  groupoid is unique up to weak equivalences. 
\end{remark}

\begin{example}\label{Exm:dimOne}
Let $(M,\partial)$ be a differential module whose underlying module $M=A.\fk{m}$ is a free $A$-module of rank one, endowed with the differential matrix $\mat{M}=a \in A$, that is, $\partial(\fk{m})=a(X)\fk{m}$. Then the Hopf algebroid $\Um$ is  generated as an $(A\tensor{\mathbb{C}}A)$-algebra by the invertible element $det_{M}=\bara{\fk{m}^*\tensor{T_M}\fk{m}}$. Thus  $\Um$ is isomorphic to the Hopf algebroid   $(A\tensor{\mathbb{C}}A)[T,T^{-1}] \cong A\tensor{\mathbb{C}}\mathbb{C}[T,T^{-1}]\tensor{\mathbb{C}}A$, which is  induced by the  Hopf $\mathbb{C}$-algebra $\mathbb{C}[T,T^{-1}]$ (the coordinate algebra of the multiplicative group).  
\end{example}

\begin{remark}\label{rem:PV}
As we can realize in Example \ref{Exm:dimOne}, the differential $\partial(\fk{m})=a(X)\fk{m}$ does not influence the Hopf structure of $\Um$. In other words, taking a different differential (e.g., $\partial(\fk{m})= b(X) \fk{m}$, for some $b \neq a \in A$) will leads, up to canonical isomorphism,  to the same  Hopf algebroid $(A\tensor{\mathbb{C}}A)[T,T^{-1}]$. Thus, the Hopf structure of $\Um$ does not take into account  the differential of $(M, \partial)$. This is perhaps  why  in the more general context of \cite[3.2.2.2]{Andre:2001} the algebraic group attached to the (isotropy) Hopf $A$-algebra $\Um/\langle \Sf{s} -\Sf{t}\rangle$, is referred to as the \emph{intrinsic differential Galois group} of $(A, \partial)$ for $(M,\partial)$.  

As we will see in the next subsection, it turns out that in the case we are interested in, that is, the case of linear differential matrix equations over the complex affine line, the differential $\partial$ of the module $M$ endows this Hopf algebra with a structure of simple differential algebra and convert it into a Picard-Vessiot extension of $(A, \partial)$ for $(M,\partial)$.
\end{remark}

\subsection{Picard-Vessiot extensions for linear differential matrix equation, after Andr\'e}\label{ssec:PV}
In this subsection we will perform the Picard-Vessiot  theory for the particular case of  polynomial algebra $A=\mathbb{C}[X]$. In order to do so,  we will use our results in combination with  the general  theory established  in \cite{Andre:2001} for differential noetherian commutative rings with semisimple total ring of fractions.  Precisely, we give a complete description, using results from subsection \ref{ssec:PB}, of the Picard-Vessiot algebra  attached to a differential module $(M,\partial)$ (or to a linear differential matrix equation) and observe that the outcome, in this particular situation, is  not far from the classical situation of differential vector spaces over  differential fields. 

We refer to \cite[D\'efinition 3.4.1.1]{Andre:2001} for the definition of the Picard-Vessiot extension of a commutative differential  algebra which we are going to use in the sequel.  Recall that our differential algebra  $A=\mathbb{C}[X]$ is endowed with the usual differentiation $d=\partial/\partial X$ and the $A$-bimodule of differential forms $\Omega^{\Sscript{1}}$ (the notation is that of \cite{Andre:2001}) is a central one which is a  free $A$-module of rank one (i.e., $\Omega^{\Sscript{1}}=\Der{\mathbb{C}}{A}^*$ is the $A$-linear dual of the module of derivations  $\Der{\mathbb{C}}{A}$). Thus, we are in the situation which is referred to in \cite{Andre:2001} as \emph{situation classique}. Besides, we will implicitly use the symmetric monoidal equivalence  between the category of free $A$-modules of finite rank with connections (for the previous $\Omega^{\Sscript{1}}$)  and the category of representations of  the Lie-Rinehart algebra $(A,\Der{\mathbb{C}}{A})$, that is, the category $\cA_{U}$ in the notation of sub-section \ref{ssec:Diff}, where $U$ is the first Weyl $\Cc$-algebra of $A$.  

Let us fix a differential module $(M,\partial)$ over  $A$ of rank $m$, with  dual basis $\{e_i, e_i^*\}_{1\leq i \leq m}$ and a differential matrix $\mat{M}=(a_{ij})_{1\leq i,\, j\leq m}$. Consider as before the attached $\mathbb{C}$-linear abelian category $\Mo$ and  its associated Hopf algebroid  $(A,\Um)$ over $\mathbb{C}$, constructed  as in subsection \ref{ssec:PV}. It is well known that this  category  admits a tensor generator (e.g., the differential module $M\oplus M^*$) and have a fibre functor over ${\rm Spec}(A) \neq \emptyset$, namely, the forgetful functor $\omega:=\omega_{|\Mo}: \Mo \to add(A)$. Then by applying  \cite[Corollaire 6.20]{Deligne:1990}, the category $\Mo$ admits a fibre functor over the base field $\mathbb{C}$, which we denote by $\omega': \Mo \to \vect{\mathbb{C}}$. 
On the other hand, since $\Mo$ is a neutral tannakian category over $\mathbb{C}$, we have that  the dimension of the $\mathbb{C}$-vector space $\omega'(M)$ is the rank of  the underlying $A$-module of the differential module $(M,\partial)$, that is, $dim_{\Sscript{\mathbb{C}}}\big(\omega'(M)\big)=m$, see \cite[Th\'eor\`eme 7.1]{Deligne:1990}. Let us then denote by $\{v_i, v^*_i\}_{1\leq i \leq m}$ a dual basis for the vector space  $\omega'(M)$.

Similar to the classical situation of linear differential matrix equations over a differential field, and as it was shown in \cite[Th\'eor\`eme 3.4.3.1]{Andre:2001}, the existence of $\omega'$ is a fundamental step in building up  the Picard-Vessiot extension of $(A,\partial)$ for $(M,\partial)$. Next, we will show, however, that the existence of  $\omega'$ comes for free with the information encoded in the Hopf algebroid structure of the pair $(A,\Um)$. This will provide us with a  more conceptual way in overcoming  that fundamental step in the Picard-Vissiot theory for our situation.

As we already mentioned, the Hopf algebroid  $(A,\Um)$ is  geometrically transitive  over $\mathbb{C}$ and we know  that $\hH(\mathbb{C}) \neq \emptyset$, so we are in position to apply \cite[Theorem {\bf A}]{ElKaoutit:2015}. Therefore,  for any point 
$x \in \mathbb{A}_{\Sscript{\mathbb{C}}}^1$, we can consider the associated isotropy Hopf $\mathbb{C}$-algebra $\Umx$ which is by definition (see \cite[Definition 5.1 and Lemma 5.2]{ElKaoutit:2015})  the base extension Hopf algebra $(\mathbb{C}, \Umx:=\bCx\tensor{A} \Um \tensor{A} \bCx)$, where $\bCx$ is $\mathbb{C}$ viewed as an $A$-algebra via the $\mathbb{C}$-algebra map $x : A \to \mathbb{C}$. It turns out that \cite[Theorem {\bf A}]{ElKaoutit:2015} implies that the canonical Hopf algebroid extension $\Sf{x}: (A,\Um) \to (\mathbb{C},\Umx) $ is a weak equivalence,  which means that the induced functor $\Sf{x}_{*}: \rcomod{\Um} \to \rcomod{\Umx}$ establishes an equivalence of symmetric monoidal categories.
Henceforth,  the full sub-category of finite-dimensional comodules $\frcomod{\Umx}$ (i.e.,  the full sub-category of rigid objects)  is equivalent, as symmetric monoidal category,  to the category $\frcomod{\Um}$.  Thus, by Theorem \ref{theo:unpo}, we  conclude that  $\Mo$ is equivalent,  as symmetric monoidal category,  to  $\frcomod{\Umx}$. In this way we have a chain of symmetric monoidal $\mathbb{C}$-linear faithful  and exact functors: 
$$
\xymatrix@C=40pt{ \omegax: \Mo  \ar@{->}_-{\otimes\text{-}\simeq}^-{\chi}[r] & \frcomod{\Um} \ar@{->}_-{\otimes\text{-}\simeq}^-{\sf{x}_*}[r]  & \frcomod{\Umx} \ar@{->}^-{\cO}[r] & \vect{\mathbb{C}},  }
$$
where $\cO$ is the forgetful functor.  In summary,  the previous observations assert the following, see also the proof of \cite[Corollaire 6.20, pages 163-164]{Deligne:1990}.
\begin{corollary}\label{coro:isotropy}
There is a point $x \in \mathbb{A}_{\Sscript{\mathbb{C}}}^1$ such that $\omega'=\omega_{\Sscript{x}}$, up to a canonical natural isomorphism.  In particular, the extended fibre functor $\omega' \tensor{\mathbb{C}} A: \Mo \to add(A)$  over $A$  is naturally isomorphic to $\omega$. 
\end{corollary}

So far, we have  two fibre functors the extension $\wp:=\omega' \tensor{\mathbb{C}} A$ and the restriction of the forgetful functor $\omega:=\omega_{|\Mo}: \Mo \to add(A)$, we are then in the situation of subsection \ref{ssec:PB}.  Therefore, by  Theorem \ref{thm:PB}, we can consider  the principal $(\rR_{\Sscript{A\tensor{}A}}(\omega), \rR_{\Sscript{A\tensor{}A}}(\wp))$-bibundle $(\rR_{A\tensor{}A}(\omega,\wp), \balpha,\bbeta)$, which is lifted  to  a principal $(\RA{\omega}, \RA{\wp})$-bibundle $(\rR_{A}(\omega,\wp), \biota)$,  over the quotient  Hopf $A$-algebras, see Eq. \eqref{Eq:Ra}. Combining this with  Theorem \ref{theo:unpo} and Corollary \ref{coro:isotropy}, we have that the first principal bibundle is isomorphic to a  trivial bibundle  $(\Um, \Um)$-bibundle $\uU(\Um)$, namely, to the unit principal bibundle attached to  $\Um$, see \cite{Kaoutit/Kowalzig:2014}.

Now we consider the quotient algebra $ \cP:=  \rR_{A}(\omega,\wp)= \rR_{A\tensor{}A}(\omega,\wp)/ \langle \balpha - \bbeta \rangle$ with the canonical algebra extension $\biota: A \to  \cP$ and denote by $[f]$ the equivalence class of an element $f \in \rR_{A\tensor{}A}(\omega,\wp)$. Define the map 
\begin{equation}\label{eq:difPV}
\bpartial: \cP  \longrightarrow \cP,\quad \Big( \class{p^*\tensor{T_P} (p\tensorC a)}  \longmapsto \class{\partial p^*\tensor{T_P}(p\tensorC a)} \,+\,  \class{ p^*\tensor{T_P}(p\tensorC \partial a)}  \Big),
\end{equation}
where $P \in \Mo$, $p^* \in P^*$, $p \in \omega'(X)$ and $a \in A$, and where we have used the differential of the duals as in given by Eq.~\eqref{Eq:dual}. 

\begin{proposition}\label{prop:difPV}
Keep the above notations. Then the pair $(\cP,\bpartial)$ is a differential algebra which enjoys  the following properties:
\begin{enumerate}[(i)]
\item  the map $\biota: (A,\partial) \to  (\cP, \bpartial)$ is a morphism of differential algebras;
\item $(\cP, \bpartial)$ is a simple differential algebra which is a Picard-Vessiot extension of $(A,\partial)$ for  $(M,\partial)$;
\item $\cP$ is isomorphic to the quotient Hopf $A$-algebra $\Um/\langle \Sf{s} - \Sf{t} \rangle$;
\item there is a surjective  map 
\begin{equation}\label{Eq:Psi}
\xymatrix@R=0pt{  \psiup_{\Sscript{M}}: A[X_{ij}, det_X^{-1}]  \ar@{->}[rr] & & \cP \\ 
X_{ij} \ar@{|->}[rr] & & \class{e_i^*\tensor{T_M}(v_j\tensor{\mathbb{C}}1_{\Sscript{A}})}:=f_{i j}, }
\end{equation}
of differential  algebras over $(A,\partial)$. In particular, $\cP \cong A[X_{ij}, det_X^{-1}]/\fk{I}$ as a differential algebra over $A$, where $\fk{I}$ is a maximal differential ideal. Furthermore, $\cP$ is generated as an $A$-algebra by the entries of the  matrix $F=(f_{ij})_{1 \leq i,\, j \leq m}$ and $det_F^{-1}$ with differential $\bpartial F = \mat{M} \, F $ (i.e., $F$ is a \emph{fundamental matrix of solutions} of the linear differential matrix equation \eqref{Eq:ldm}). 
\end{enumerate}
\end{proposition}
\begin{proof}
The first claim and item $(i)$ are easy verifications.  Part $(ii)$ follows from \cite[Lemme 3.4.2.1(ii)]{Andre:2001} by combining Proposition \ref{prop:IsomF} and Theorem \ref{thm:PB}$(iii)$  with  \cite[\S 3.2.2.1]{Andre:2001}. As we have seen before,   using Theorem \ref{theo:unpo} in conjunction with Corollary \ref{coro:isotropy}, we  get an isomorphism $\Um \cong \rR_{A\tensor{}A}(\omega,\wp)$ of Hopf algebroids. So part $(iii)$ is obtained by going to the quotient Hopf $A$-algebras, that is,  by considering the extended  isomorphism of Hopf $A$-algebras $\Um/\langle \Sf{s} - \Sf{t} \rangle \cong  \rR_{A\tensor{}A}(\omega,\wp)/\langle \Sf{s} - \Sf{t} \rangle=\cP$. 

The first claim of part $(iv)$ is a direct consequence of part $(iii)$ once taking into account  the surjective morphism of Hopf algebroids given in Eq. \eqref{Eq:Phi1}. The particular statement is a direct implication of the first claim in this part  and  item $(ii)$.  The last claim of $(iv)$ follows from the following computation:
\begin{eqnarray*}
 \bpartial \Big(  \class{e_i^*\tensor{T_M}( v_j\tensor{\mathbb{C}}1_A)}\Big) &=& \class{\partial e_i^*\tensor{T_M}( v_j\tensor{\mathbb{C}}1_A)} \\ & =& \sum_{k=1}^{m} \class{\partial e_i^*(e_k) e_k^*\tensor{T_M}( v_j\tensor{\mathbb{C}}1_A)}
\\ & \overset{\eqref{Eq:dual}}{=}& \sum_{k=1}^{m} \class{\big(  \partial( e_i^*(e_k)) - e_i^*(\partial e_k) \big)e_k^*\tensor{T_M}( v_j\tensor{\mathbb{C}}1_A)}
\\ & =& \sum_{k=1}^{m} \class{ \partial( e_i^*(e_k)) e_k^*\tensor{T_M}( v_j\tensor{\mathbb{C}}1_A)} -   \sum_{k=1}^{m} \class{ e_i^*(\partial e_k) e_k^* \tensor{T_M}( v_j\tensor{\mathbb{C}}1_A)}
\\ & =& 0 -   \sum_{k=1}^{m} \class{ e_i^*(\partial e_k) e_k^* \tensor{T_M}( v_j\tensor{\mathbb{C}}1_A)}
\\ &\overset{\eqref{Eq:aij}}{ =}&     \sum_{k, \, l}^{m} \class{ a_{lk}e_i^*(e_l) e_k^* \tensor{T_M}( v_j\tensor{\mathbb{C}}1_A)} \,\, =\,\,     \sum_{k=1}^{m} \biota(a_{ik})\class{  e_k^* \tensor{T_M}( v_j\tensor{\mathbb{C}}1_A)} \\ &=& \sum_{k=1}^{m} \biota(a_{ik})f_{kj},
\end{eqnarray*}
which shows that $\bpartial F =  \mat{M}\, F$ and this finishes the proof.
\end{proof}

Now we focus on  the structure of the group of differential algebra automorphisms  of a Picard-Vessiot extension. As before we consider $(\cP, \bpartial)$ a Picard-Vessiot extension of $(A,\partial)$ for $(M,\partial)$ and denote by $\Autd{A}{\cP}$ the \emph{group of differential $A$-algebra automorphisms}, that is, an element in this group is an  algebra automorphism $\bsigma: \cP \to \cP$ such that $\bsigma \circ \biota = \biota$ and $\bpartial \circ \bsigma = \bsigma \circ \bpartial$.  In this way, we obtain a functor valued in  groups: 
$$
\AutdF{A}{\cP}: \Alg{\mathbb{C}} \longrightarrow {\rm Grps}, \quad \Big( C \longrightarrow  \Autdt{C} \Big),
$$
whose fibre at ${\rm Spec}(\Cc)$ is the starting group $\Autd{A}{\cP}$.

Our next task is to prove that this in fact is an affine algebraic group. Namely, we show that is  isomorphic to each of the isotropy groups of the differential Galois groupoid $\hHm(\mathbb{C})$.  The proof will be done in several steps. First, we know from the definition of the Picard-Vessiot extension that the fibre functor $\omega': \Mo \to \vect{\mathbb{C}}$ is naturally isomorphic to the fibre functor $\kappaup: \Mo \to \vect{\mathbb{C}}$, which sends any differential module $N \in \Mo$ to the finite-dimensional $\mathbb{C}$-vector space $\ker(\Partial{N\tensor{A}\cP})$ the kernel of the differentiation of the differential module $N\tensor{A}\cP$,  see \cite[Lemme 3.4.2.1]{Andre:2001} for the proof of this fact. Let us denote by $\bV=\kappaup(M)$ and by $\{\bv_i,\bv_i^*\}_{1 \leq i \leq m}$ a dual basis of this vector space. Notice that, as we have seen before, $dim_{\Sscript{\mathbb{C}}}(\bV) = rank(M)=m$.  We will make the following choice for this basis $\bv_i= \sum_j e_j  \tensor{A} f_{ji}$, for every $i=1,\cdots,n$.
The following lemma will be implicitly used in the subsequent one.
\begin{lemma}\label{lem:oplusT}
Let $k,l $ be a positive integer. Then, for any objet $X \in \Mo$, we have a monomorphism of $\Cc$-vector spaces $\oplus_{k,l}T^{\Sscript{(k,\,l)}}\big(\kappaup(X)\big) \hookrightarrow  \oplus_{k,l}T^{\Sscript{(k,\,l)}}(X\tensor{A}\cP)$ (finite direct sums), which is extended to a monomorphism 
$$
\xymatrix{\bigoplus_{k,l}T^{\Sscript{(k,\,l)}}\Big(\kappaup(X)\tensor{}C\Big)  \ar@{^{(}->}^-{}[r] &  \bigoplus_{k,l}T^{\Sscript{(k,\,l)}}\Big(X\tensor{A}\cP\tensor{}C\Big) }
$$
of $C$-modules, for any $\Cc$-algebra $C$.
\end{lemma}
\begin{proof}
It is immediate by using  the fact that $\kappaup$ is a fibre functor. 
\end{proof}
We refer to \cite[Chapitre II. \S 1. n$^{\Sscript{\text{o}}}$ 3]{DemGab:GATIGAGGC} for the definition of the stabilizers sub-functors  occurring in the following claim which can be compared with \cite[Th\'eor\`eme 3.5.1.1]{Andre:2001}. 
\begin{lemma}\label{lema:Stab}
There is a monomorphism of functors $\AutdF{A}{\cP} \hookrightarrow {\rm GL}_{\Sscript{\mathbb{C}}}(\bV)\cong {\rm GL}_{\Sscript{\mathbb{C}}}(\omega'(M)) $. Moreover, the image of $\AutdF{A}{\cP}$ is contained in  $\underline{{\rm Stab}}\{\kappaup(N)\}$, the stabilizer  of every sub-object $N$ of a finite direct sum $\oplus_{k,\,l} T^{\Sscript{(k,\,l)}}(M)$ in the category $\Mo$.  In particular,  the image of $\AutF{\omega'}$ containts that of $\AutdF{A}{\cP}$. 
\end{lemma}
\begin{proof}
Let $C$ be an object  in $\Alg{\Cc}$ and $\bgamma \in \Autdt{C}$. We set 
$$
\brho_{\Sscript{\bgamma}}:= (M\tensor{A}\bgamma)_{|\bV}: \bV \tensor{}C \to \bV\tensor{}C
$$
which by definition is a well defined $C$-linear automorphism. This gives  the stated  morphism of functors. The fact that this is a monomorphism can be obtained as follows: Take two automorphisms $\bgamma$ and $\bgamma'$ and assume that $\brho_{\Sscript{\bgamma}} = \brho_{\Sscript{\bgamma'}}$. Using the basis $\{\bv_i\tensor{}1_{\Sscript{C}}\}_{1 \leq i \leq m}$ of the free $C$-module $\bV\tensor{}C$, we get that $\bgamma(f_{ji}\tensor{}1_{\Sscript{C}}) = \bgamma'(f_{ji}\tensor{}1_{\Sscript{C}})$, for every pair of indices $i,j=1,\cdots, m$. Thus $\bgamma=\bgamma'$,  as they are $(A\tensor{}C)$-algebra  maps. As for the second claim, take  $C$ and $\bgamma$ as before and consider  a monomorphism  $N \hookrightarrow \oplus_{k,l}T^{\Sscript{(k,l)}}(M)$ to a finite direct sum  in the category $\Mo$.  Denote by $\bW=\kappaup(N)$, so up to a canonical natural isomorphism, we have the following commutative diagram of $C$-modules:
$$
\begin{small}
\xymatrix@C=15pt@R=20pt{ \bigoplus_{k,l}T^{\Sscript{(k,l)}}(M\tensor{A}\cP\tensor{}C) \ar@{->}^-{\oplus_{k,l}T^{\Sscript{(k,l)}}(M\tensor{A}\bgamma)}[rrrr] &  & &   & \bigoplus_{k,l}T^{\Sscript{(k,l)}}(M\tensor{A}\cP\tensor{}C) &  \\
& \bigoplus_{k,l}T^{\Sscript{(k,l)}}(\bV\tensor{}C)  \ar@{->}^-{\oplus_{k,l}T^{\Sscript{(k,l)}}(\brho_{\Sscript{\bgamma}})}[rrrr] \ar@{_{(}->}^-{}[lu]  & &   & & \bigoplus_{k,l}T^{\Sscript{(k,l)}}(\bV\tensor{}C) \ar@{_{(}->}^-{}[lu] \\ 
&  & &   & & \\ 
N\tensor{A}\cP \tensor{}C \ar@{->}^-{N\tensor{A}\bgamma}[rrrr] \ar@{^{(}->}^-{}[uuu] &  & &   & \ar@{^{(}->}^-{}[uuu] N\tensor{A}\cP \tensor{}C & \\   &  \bW\tensor{}C \ar@{-->}^-{}[rrrr] \ar@{_{(}->}^-{}[lu] \ar@{^{(}->}[uuu]   & &   & & \ar@{_{(}->}[ul] \bW\tensor{}C \ar@{^{(}->}[uuu]  }
\end{small}
$$
This implies that  the front rectangle is  commutative as well. Therefore, the action given $\brho_{\Sscript{\bgamma}}$ stabilizes $\kappaup(N)\tensor{}C$ and shows the claim.     
Lastly, the particular consequence is a direct application of \cite[Th\'eor\`eme 3.2.1.1(iii)]{Andre:2001} combined with the previous statement. 
\end{proof}

\begin{proposition}\label{prop:AutAsAlgG}
Let $(\cP,\bpartial)$ be the above Picard-Vessiot extension of $(A,\partial)$ for $(M,\partial)$. Then, we have an isomorphism  $\AutdF{A}{\cP} \cong \AutF{\omega'}$ of affine group schemes. Furthermore, the attached affine algebraic  group $\Autd{A}{\cP}$ is isomorphic to each of the isotropy groups of the  algebraic groupoid $\hHm(\Cc)$ of Definition \ref{def:GDG}.  
\end{proposition}
\begin{proof}
The last claim is a direct consequence of the first one in conjunction with Corollary \ref{coro:isotropy}. In order to prove the first statement, we only need to check that (the image of) $\AutF{\omega'}$ is contained in  the image of $\AutdF{A}{\cP}$ by the monomorphism of Lemma \ref{lema:Stab}. So take an object $C$ in $\Alg{\Cc}$ and a natural  isomorphism $\xi \in \AutF{\omega'}(C)$.  Consider then the $C$-linear automorphism $\xi_{M}$ and  its associated invertible $m\times m$-matrix  $(c_{ij})_{\Sscript{i,\, j}}$ with coefficients in $C$. Now, define the following $(A\tensor{}C)$-algebra map:
$$
\bgamma_{\xi}: P\tensor{}C \longrightarrow P\tensor{}C, \quad \Big( af_{ij}\tensor{}c \longmapsto  \sum_{l} af_{il}\tensor{}c_{jl}c\Big).
$$
It is not difficult to check that $\bgamma_{\xi} \in \Autdt{C}$. We still have to check that $\brho_{\bgamma_{\xi}}= \xi_{M}$ and this follows form the following computations: For each $i=1,\cdots, m$, we compute
\begin{eqnarray*}
\brho_{\bgamma_{\xi}}\Big( \bv_i\tensor{}1_{\Sscript{C}} \Big)  &=&  \brho_{\bgamma_{\xi}}\Big( \sum_je_j\tensor{A}f_{ji}\tensor{}1_{\Sscript{C}} \Big) \,\, =\,\,  (M\tensor{A}\bgamma_{\xi})  \Big( \sum_je_j\tensor{A}f_{ji}\tensor{}1_{\Sscript{C}} \Big) \\ &=&   \sum_{j,\, l}e_j\tensor{A}f_{jl}\tensor{}c_{jl}  \,\, =\,\,  \sum_{ l} \bv_{l}\tensor{}c_{jl}  \,\, =\,\,  \sum_{ l} (\bv_{l}\tensor{}1_{\Sscript{C}}) \,  c_{jl} \\ &=& \xi_{M}\Big( \bv_i\tensor{}1_{\Sscript{C}} \Big),
\end{eqnarray*}
which shows the desired equality and finishes the proof. 
\end{proof}

Proposition \ref{prop:AutAsAlgG} suggests a terminology for the algebraic group $\Autd{A}{\cP}$, thus, we can refer to this group as \emph{the differential Galois isotropy group} of $(A,\partial)$ for the differential module $(M,\partial)$.

\begin{example}\label{exam:AZ}
Let $(M,\partial)$ be a differential $A$-module of rank $2$  with differential matrix 
$\mat{M}=\dosmatrix{0}{a}{0}{0}$, for some non-zero  polynomial $ a \in A$. The associated Hopf algebroid $\Um$ is generated as an $(A\tensor{}A)$-algebra by the set  $\Big\{\bara{e_i^*\tensor{T_M}e_j}\Big\}_{1 \leq i,\, j \leq 2}$. Now, by considering  the following   differential morphisms 
$$
A \longrightarrow M, \quad \Big( 1 \longmapsto e_1 \Big); \qquad M \longrightarrow A,\quad \Big( \begin{pmatrix} a_1 \\ a_2 \end{pmatrix} \longmapsto a_2 \Big)
$$
(these are morphisms in the category $\Mo$), we show  that $f_{11}=f_{22}=1_{\Sscript{\cP}}$ and $f_{21}=0$. Therefore $\cP$ is generated as an $A$-algebra by  the element  $f_{12}$. Since, we already know that there is a non constant polynomial $b \in A$ such that $\partial b= a$, we have another morphism in the category $\Mo$, namely, the one given by  
$$
\fk{t}: A \longrightarrow M, \quad \Big( 1 \longmapsto (e_2- e_1b)\Big).
$$
Using, this morphism, we have the following equalities in the algebra $\Um$
\begin{eqnarray*}
\bara{e_1^*\tensor{T_M}e_2} & =& \bara{e_1^*\tensor{T_M}(e_2-e_1b)} + \bara{e_1^*\tensor{T_M}e_1b} \\ &=&  \bara{e_1^*\tensor{T_M}\fk{t}(1_A)} + \bara{e_1^*\tensor{T_M}e_1b}  \\ &=& \bara{e_1^*\fk{t} \tensor{T_A}1_A} + \bara{e_1^*\tensor{T_M}e_1b}  \\ &=& -\bara{b {\sf{l}}_{1_A} \tensor{T_A}1_A} + \bara{e_1^*\tensor{T_M}e_1b} 
\end{eqnarray*}
Passing to the quotient $A$-algebra $\cP=\Um/\langle \Sf{s} - \Sf{t}\rangle$, this  implies  that $f_{12}=0$. Thus $\cP=A$ and the differential Galois isotropy group is a trivial group in this case. 
\end{example}

\begin{remark}\label{rem:PV}
Observe that the opposite principal bibundle $(\rR_{A\tensor{}A}(\wp, \omega), \bbeta, \balpha)$ of the principal bibundle $(\rR_{A\tensor{}A}(\omega,\wp), \balpha,\bbeta)$ (see Theorem \ref{thm:PB} and  \cite[Section 4]{Kaoutit/Kowalzig:2014}) provides us with another Picard-Vessiot extension of $(A,\partial)$, although, as an algebra it will be generated by a fundamental matrix of solutions $F$  satisfying $\bpartial F = -F \mat{M}$, which is obviously not a matrix of solutions for the system \eqref{Eq:ldm}, but it is for the differential module $(M^*,\partial)$. This, explain the why of the minus sign in the differentiation of equation \eqref{Eq:aij} and of the use of columns instead of rows in the formulation  of the system \eqref{Eq:ldm}. 
On the other hand, since $\cP$ is a simple differential $A$-algebra and $A$ its self is so, we can show by elementary arguments that $\cP$ has no non zero divisor (see also \cite[Proposition 3.4.4.4]{Andre:2001}), and so  consider its total field of fractions $\qQ(\cP)$. In this way, we will end up with a differential  field extension $\Cc(X) \to \qQ(\cP)$. If the field of constants of $\qQ(\cP)$  coincides with $\Cc$, then  the extension $\qQ(\cP)/\Cc(X)$ can be  referred to as the \emph{Picard-Vessiot field extension} for $(M,\partial)$.
  
Lastly, as we have seen along the previous subsections, the results as well as the examples described therein,  present a strong resemblance with the classical case of the differential field $\mathbb{C}(X)$. 
\end{remark}

\subsection{Comparison with Malgrange's and Umemura's differential Galois groupoids}\label{ssec:MalUme}
In this section we compute the Hopf algebroid structure of the coordinate ring of what is know in the literature as Malgrange's  groupoid (or $D$-groupoid) for some special cases, and illustrate the differences between this groupoid and  our approach.

Before going on, the following notations are needed. For any positive integer $n \in \mathbb{N} \setminus \{0\}$, we denote by $\epsilon_0,\epsilon_1, \cdots, \epsilon_n$ the elements $(1,0, \dots,0), (0,1,0,\cdots,0), \cdots, ( 0,\cdots,0,1) \in \mathbb{N}^{n+1}$, respectively. Given such an integer $n$, we define the following function 
$$
\xymatrix@R=0pt{ \mathbb{N}^n \ar@{->}^-{\bo{k}^n}[rr] & & \mathbb{N} \\  (k_1,\cdots,k_n) \ar@{|->}[rr] & & k_1+2k_2+3k_3+\cdots+ nk_n:=\kn.  }
$$
We denote by $\KK{n}:=(\bo{k}^n)^{-1}(\{n\})$ the inverse image of $\{n\}$ by  the function $\bo{k}^n$. Thus elements of $\KK{n}$ are $n$-tuples $(k_1, \cdots,k_n) \, \in \mathbb{N}^n$ of integers 
such that  $n=k_1+2k_2+\cdots+n k_n$. For instance, we have 
\begin{multline*}
\KK{1}=\big\{ 1 \big\},\quad  \KK{2}\,=\, \big\{ (0,1); (2,0) \big\}, \quad \KK{3}\,=\, \big\{  (0,0,1); (1,1,0); (3,0,0)\big\}, \\ 
\KK{4}\,=\, \big\{ (0,0,0,1); (1,0,1,0); (2,1,0,0); (4,0,0,0); (0,2,0,0) \big\}, \; \cdots.
\end{multline*}

Let $A=\Cc[X]$ be as before the one variable polynomial complex algebra and  $\{x_0, y_{n}\}_{n \in \, \mathbb{N}}$ be a set of independent variables over $\Cc$. Consider the following commutative polynomial $\Cc$-algebra
$$
\cH : = \Cc[x_0,y_0,y_1,\cdots,y_n,\cdots,\frac{1}{y_1}],
$$
which we also denote by $\cH_{\Sscript{\Cc}}$.
There are two algebra maps 
\begin{equation}\label{Eq:stH}
\Sf{s}: A \to \cH, \;\Big( X \mapsto x_0:=x  \Big) \quad \text{ and } \quad \Sf{t}: A \to \cH, \; \Big( X \mapsto y_0:=y  \Big)
\end{equation}
so that we can consider the $A$-bimodule ${}_{\Sscript{\Sf{s}}}\cH_{\Sscript{\Sf{t}}}$. 
Moreover, we have the following algebra maps: 
\begin{equation}\label{Eq:DH}
\begin{gathered}
\xymatrix@R=0pt{  {}_{\Sscript{\Sf{s}}}\cH_{\Sscript{\Sf{t}}} \ar@{->}^-{\Delta}[rr]  & & {}_{\Sscript{\Sf{s}}}\cH_{\Sscript{\Sf{t}}} \tensor{A} {}_{\Sscript{\Sf{s}}}\cH_{\Sscript{\Sf{t}}}  } \\  \Delta(x)\,\,=\,\, x\tensor{A}1, \quad \Delta(y)\,\,=\,\, 1\tensor{A} y,  \\ 
\Delta(y_n)\,\,=\,\,  \sum_{\kn\, \in \, \KK{n}} \frac{n!}{k_1!\, \cdots\, k_n!} \Big( \lrfrac{y_1}{1!}^{k_1}\, \lrfrac{y_2}{2!}^{k_2}\, \cdots\, \lrfrac{y_n}{n!}^{k_n} \Big) \tensor{A} y_{k_1+k_2+\cdots+k_n}, \; \text{ for  } n \geq 1. 
\end{gathered}
\end{equation}
Thus, for $n=1, 2, 3, 4$, the image by $\Delta$ of the variables $y_n$'s reads  as follows:
\begin{multline*}
\Delta(y_1)=y_1\tensor{A}y_1, \quad \Delta(y_2)=y_2\tensor{A}y_1+y_1^2\tensor{A}y_2, \quad \Delta(y_3)=y_3\tensor{A}y_1+ 3y_1y_2\tensor{A}y_2+y_1^3\tensor{A}y_3, \\ 
\Delta(y_{4})= y_{4}\tensor{A}y_{1} + 4 y_{3}y_{1}\tensor{A}y_{2} + 6 y_{2}y_{1}^{2}\tensor{A}y_{3} + 3y_{2}\tensor{A}y_{2} + y_{1}^{4}\tensor{A}y_{4},   \quad \cdots . \hspace{3cm}
\end{multline*}
It is by construction that $\Delta$ is actually a morphism of $A$-bimodules. 
There are  other $A$-bimodule morphisms  which are given as follows: 
\begin{equation}\label{Eq:SH}
\begin{gathered}
\xymatrix@R=0pt{  {}_{\Sscript{\Sf{s}}}\cH_{\Sscript{\Sf{t}}} \ar@{->}^-{\sS}[rr]  & & {}_{\Sscript{\Sf{t}}}\cH_{\Sscript{\Sf{s}}}   } \\
\sS(x) \,=\, y,\quad \sS(y)\,=\, x, \quad  \sS(y_1)\,=\, y_1^{-1}, \quad \text{ and }
\\  \sS(y_n) \,=\, \sum_{(n,0,\cdots,0)\,\neq \,  \kn \, \in \, \KK{n}} -\frac{n!}{k_1!\, \cdots\, k_n!} \, \sS\big(  y_{k_1+k_2+\cdots+k_n}\big)\,  \Big( \lrfrac{y_1}{1!}^{k_1-n}\, \lrfrac{y_2}{2!}^{k_2}\, \cdots\, \lrfrac{y_n}{n!}^{k_n} \Big), \text{ for  } n \geq 2,
\end{gathered}
\end{equation}
for instance, 
\begin{multline*}
\sS(y_1)=y_1^{-1} , \quad \sS(y_2)=-y_2y_1^{-3}, \quad  \sS(y_3)=-y_3y_{1}^{-4} + 3y_{2}^{2}y_{1}^{-5}, \\ \sS(y_{4}) = -y_{4}y_{1}^{-5} + 10 y_{3}y_{2}y_{1}^{-6}- 15 y_{2}^{3}y_{1}^{-7},  \quad \cdots . \hspace{1cm}
\end{multline*}
And
\begin{equation}\label{Eq:EH}
\begin{gathered}
\xymatrix@R=0pt{  {}_{\Sscript{\Sf{s}}}\cH_{\Sscript{\Sf{t}}} \ar@{->}^-{\varepsilon}[rr]  & &  A  } \\ \varepsilon(x)\,=\, X,\quad \varepsilon(y)\,=\, X, \quad \varepsilon(y_1)\,=\,1, \; \text{ and } \\
\varepsilon(y_n)\,\,=\,\, 0, \quad \text{ for every } \; n \geq 2. 
\end{gathered}
\end{equation}
The $\Cc$-algebra $\cH$ is a differential algebra with differential given by:
\begin{equation}\label{Eq:DiffH}
\begin{gathered}
\xymatrix@R=0pt{  \cH \ar@{->}^-{\bdelta}[rr]  & &  \cH  } 
\\ \bdelta(x)\,=\, 1,\quad \bdelta(y)\,=\, y_1, \quad \bdelta(y_n)\,=\,y_{n+1}, \; \text{ for } n \geq 1.
\end{gathered}
\end{equation}
Thus, we have  
$$
\bdelta\,\,=\,\, \frac{\partial}{\partial x}  \,+\, \sum_{i=0}^{\infty} y_{i+1} \frac{\partial}{\partial y_i}.
$$

We will consider $\cH$ as an $(A\tensor{}A)$-algebra using the $\Cc$-algebra map $\etaup:= \Sf{s}\tensor{}\Sf{t}: A\tensor{}A \to \cH$ given by equation \eqref{Eq:stH}.  It is by construction that $\cH$  is via $\etaup$ a faithfully flat $(A\tensor{}A)$-bimodule. The following proposition can be seen as the algebraic counterpart of the ``universal" geometric groupoid constructed in \cite{Malgrange:2001} and \cite{Umemura:2009}. 

\begin{proposition}\label{prop:MU}
The pair $(A,\cH)$ admits a structure of geometrically transitive commutative   Hopf  algebroid over $\Cc$, whose comultiplication, counit and antipode, are given by equations \eqref{Eq:DH}, \eqref{Eq:EH} and \eqref{Eq:SH}, respectively. Furthermore, $(\cH,\bdelta)$ is a differential extension of $(A,\partial)$ via the source map.
\end{proposition}
\begin{proof}
The last claim is clear.  As for the proof of the first one, there are different ways to achieve it and the details are left to the reader. One way could be the use of direct computations, using a certain kind of induction, to show  that these maps are compatible (i.e., they satisfy the pertinent commutative diagrams for the definition of a commutative Hopf algebroid, see for instance  \cite[3.1]{Kaoutit/Kowalzig:2014} or \cite[3.2]{ElKaoutit:2015}), since we already know that the stated maps  are obviously  $\Cc$-algebra maps.  Another way, is to show that the associated presheaf $\hH$ of the pair $(A,\cH)$ of $\Cc$-algebras lands  in fact in groupoids. To do so, one need to think of the variable $y_0=y$ as if it was a function on $x$ and the monomials $y_i^{\alpha_i}$ as if they were formal derivative $\frac{\partial^{\alpha_i} y}{ \partial x^{\alpha_i}}$. In this way the compositions in the fibre groupoids are given by the chain rule of derivatives, and the inverse is by thinking this time that $x$ is a function on $y$ and using its derivative, see \cite{Morikawa/Umemura:2009, Umemura:2009}.   Lastly, the fact that $(A, \cH)$ is geometrically transitive follows from the fact that $\etaup$ is a faithfully flat extension, as we have  mentioned above, and from \cite[Theorem A]{ElKaoutit:2015}.
\end{proof}

\begin{remark}\label{rem:MJet}
Let $(A,\cH)$ be the Hopf algebroid of Proposition \ref{prop:MU} and  consider the sub $\Cc$-algebra $\cH_{\Sscript{r}}:=\Cc[x,y_{0}, y_{1},\cdots, y_{r}, \frac{1}{y_{1}}]$ of $\cH$, for each $r \in \mathbb{N} \setminus \{0\}$ and set $\cH_{\Sscript{0}}=\Cc[x,y_{0}] \cong A\tensor{}A$. It is clear from the definitions that the structure maps given in equations  \eqref{Eq:DH}, \eqref{Eq:EH} and \eqref{Eq:SH}, when restricted to $\cH_{\Sscript{r}}$ (denoted by $\Sf{s}_{\Sscript{r}}$, $\Sf{t}_{\Sscript{r}}$, $\varepsilon_{\Sscript{r}}, \Delta_{\Sscript{r}}$ and $\sS_{\Sscript{r}}$),  leads to a Hopf algebroid $(A, \cH_{\Sscript{r}})$. Thus, the family $\{(A,\cH_{\Sscript{r}})\}_{r \, \in \mathbb{N}}$ is a family of Hopf sub-algebroids of $(A,\cH)$ and the canonical inclusions $\cH_{\Sscript{r}} \to \cH_{\Sscript{r+1}}$ give an inductive  system of Hopf $A$-algebroids such that 
$$
\cH \, =\,  \underset{ r \, \in \, \mathbb{N}}{\varinjlim}\big(\cH_{\Sscript{r}}\big).
$$

In order to clarify the connection with our approach, it is convenient to  explain the idea which relates these Hopf algebroids with the Lie algebroids of jet vector bundles of the tangent bundle of $\affLine{\Cc}$.  

Notice first that each of the algebras $\cH_{\Sscript{r}}$ is clearly the coordinate ring of a complex variety which we denote by $J_{\Sscript{r}}^{*}$ and $\cH$ is the coordinate ring of a variety denoted  by $J_{\infty}^{*}$. This means that we have a family of groupoids (Lie groupoids indeed)  $\left\{\big(J^{*}_{\Sscript{r}}, \affLine{\Cc}\big)\right\}_{ \Sscript{r \, \in \, \mathbb{N}}}$ with ``prolongation''  the groupoid $J_{\infty}^{*}= \varprojlim \big(J^{*}_{\Sscript{r}}\big)$.  The later is the Zariski open set of the infinite dimensional analytic space $\Cc \times \Cc^{\Sscript{\mathbb{N}}}$ defined by the sequences $(x,y_{n})_{\Sscript{n \geq 0}}$ with $y_1 \neq 0$ (referred to in the literature as the \emph{Lie groupoid of invertible jets}). 

Following the general idea of \cite[3.2]{Malgrange:2001}, if we consider the Lie algebroid $Lie(J_{\Sscript{r}}^{*})$ of the groupoid $J_{\Sscript{r}}^{*}$ (i.e., the relative tangent via the source (or target) map), then $Lie(J_{\Sscript{r}}^{*})$ is identified with $J_{\Sscript{r}}T$, the jet vector bundle of order $r$ of the tangent bundle $T$ of $\affLine{\Cc}$. In this way, the  module of global sections of the associated coherent sheaf is isomorphic as (right) $A$-module to the differential operators bimodule ${\rm Diff}_{\Sscript{r}}(A)$ of order $r$ (see \cite[Definition 9.67]{nestruev} for the definition of this $A$-bimodule). The general construction of the Lie algebroid structure of the jet bundles for a given Lie algebroid can be found in \cite[4.1]{Crainic/Fernandes:2010}, the explicit formulae for the bracket on the global sections of these jet bundles is detailed in \cite[page 478]{Malgrange:2001}. It is noteworthy to mention also that these constructions can be performed in a purely algebraic way by using Lie-Rinehart algebras and their jets (or differential operators), in the sense of \cite{nestruev}.

In relation with the constructions we have seen so far, we can easily check that he Hopf algebroid $\cH_{\Sscript{1}} =\Cc[x, y_{0}, y_{1}^{\pm 1}]$ is isomorphic  to $\Um$ for a given differential module $(M,\partial)$ of rank one, as was mentioned in Example \ref{Exm:dimOne}. For a higher order, that is, for $r \geq 2$, it is possible that the Hopf algebroid $\cH_{\Sscript{r}}$ can be related to the finite dual of the co-commutative Hopf algebroid $\cV_{\Sscript{A}}(L_{\Sscript{r}})$ the  universal enveloping algebroid of  the Lie-Rinehart algebra $L_{\Sscript{r}}:=\Gamma(J_{\Sscript{r}}T)$ the global sections of the Lie algebroid of jets bundle of order $r$ of the bundle $T$ (see Example \ref{Exam: bundles}(2)).  

In summary it seems that with our approach (and also that of \cite{Andre:2001}) we only can treat the study of a system of linear differential equation associated to a representation of a given Lie algebroid (or Lie-Rinehart algebra, which in the situation of the present section is $\Gamma(T\affLine{\Cc})$ the global sections of the tangent bundle), and we are not able to analyse a system of partial differential equations attached to the jets of order higher  than $2$.  This perhaps can be seen as a disadvantage of our approach, however,  with this approach we are able also to  study a system of linear differential equation attached to differential module with rank higher than $2$ without being forced to change the base algebra (i.e., without being forced to increase the number of the coordinates system), which is the situation adopted in \cite[page 493]{Malgrange:2001}. To be more precise, 
let us consider a differential module $(M,\partial)$ over $(A, \partial)$ of rank $m$.  In comparison with \cite{Malgrange:2001}, the underlying module  $M$ can be considered as the global sections module of a trivial vector bundle of the rank $m$ over $\affLine{\Cc}$, that is, we have $M =\Gamma(E)$ (or $\underline{E}$ in the notation of \cite{Malgrange:2001}), where $E=\Cc^{m} \times \Cc$. The differential structure map $\partial$ on $M$, is then interpreted as a flat connection $\nabla$ on the vector bundle $E$ (recall here, as in \cite{Malgrange:2001}, were are in the case of smooth connected spaces).  Following \cite[5.3, page 493]{Malgrange:2001},  (here we take the discrete subset $Z=\emptyset$ of $\affLine{\Cc}$), the corresponding Hopf algebroid whose quotient leads to the description of the differential Galois group  of $\nabla$ (or to the associated one dimensional foliation without singularity, as $Z=\emptyset$), is the one with the base algebra $\cO_{\Cc^m \times \Cc}$, since the Lie groupoid of invertible jets has the space $\Cc^{m}\times \Cc$ as the space of objects. Such a  Hopf algebroid is given by the following differential polynomial algebra 
\begin{equation}\label{Eq:HCm}
\cH_{\Sscript{\Cc^{m+1}}} := \Cc\Big[x_0,x_1,x_2,\cdots,x_m, y_0, y_1,\cdots, y_m, y_{i\alpha}, det(y_{i\epsilon_j})^{-1}\Big]_{ i, j=1,\cdots, \,m; \,\,  \alpha\,  \in\, \mathbb{N}^{m+1}\setminus\{0, \epsilon_0,\cdots, \epsilon_{m}\}}, 
\end{equation}
viewed as an Hopf algebroid with base algebra  $\Cc[X_0,X_1,\cdots,X_m]$, see
Example  \ref{exam:Rank2} below where we illustrated the construction of this Hopf algebroid for $m=1$. Therefore, if we want to employ the method of \cite{Malgrange:2001} for studying a differential module of rank $m \geq 2$ over $A=\Cc[X_0]$, then we are forced to use   the Hopf algebroid of equation \eqref{Eq:HCm} as a Hopf algebroid over $\Cc[X_0,X_1,\cdots,X_m]$. This shows that our approach is somehow different from that adopted in \cite{Malgrange:2001}, since all our Hopf algebras and Hopf algebroids have for the base algebra the polynomial algebra with no more than one variable, i.e., ~ the algebra $A$. 
\end{remark}

Next, we introduce the notion of \emph{Malgrange Hopf algebroid} which is the coordinate ring of what is known in the literature as the \emph{$D$-groupoid}, or \emph{Malgrange groupoid}. We will treat only  the case of Hopf algebroids of over the polynomial algebra  $A$, that is, Hopf quotients of the above Hopf algebroid $\cH_{\Sscript{\Cc}}$, since this is the case we are interested in this section. 

Recall first that \emph{a Hopf ideal}  of a given Hopf algebroid $(A,\cH)$ is an ideal $\cI$ of $\cH$ such that the pair of algebras $(A,\cH/\cI)$ admits a unique structure of Hopf algebroid for which the morphism $(id_A, \pi): (A,\cH) \to (A,\cH/\cI)$ becomes a universal morphism of Hopf algebroids. That is,  any morphism $(id_a,\phiup): (A,\cH) \to (A,\cK)$ of Hopf algebroids such that $\cI \subseteq \ker(\phiup)$, factors  through $(id_A, \pi)$ by a morphism of Hopf algebroids. This in fact is a naive definition  which only  consider subgroupoids of $\hH$ with same set of objects  $\hH_{\Sscript{0}}$. The strict definition is a pair of ideals $(\cI_{\Sscript{0}}, \cI_{\Sscript{1}})$ of the pair of algebras $(A,\cH)$ satisfying pertinent conditions; since we will not use this general notion,  we will not go on to the details.   

Here is the promised definition:

\begin{definition}[\cite{Malgrange:2001, Umemura:2009}]\label{def:MU}
Given $(A,\cH)$ as in Proposition \ref{prop:MU} with $\hH$ the associated presheaf of groupoids. A  \emph{Malgrange Hopf $A$-algebroid} over $\Cc$ is a quotient Hopf algebroid of the form $(A,\cH/\cI)$, where $\cI$ is a differential Hopf ideal, that is,  a Hopf ideal $\cI$ with $\bdelta(\cI) \subseteq \cI$, where $\bdelta$ is the derivation of equation \eqref{Eq:DiffH}. 
Obviously the pair $(A,\cH/\cI)$ defines a presheaf of groupoids which we denote by $\kK_{\Sscript{\cI}}$, and for every commutative $\Cc$-algebra $C$ we have that $ \kK_{\Sscript{\cI}}(C)$ is a  sub-groupoid of $\hH(C)$. Thus, $\kK_{\Sscript{\cI}}$ is a sub-groupoid of $\hH$. An 
alternative terminology and the one which can be found in the literature,  is to say that $\kK_{\Sscript{\cI}}$ is a \emph{$D$-groupoid}  (or \emph{Malgrange groupoid} over $\affLine{\Cc}$) defined by the Hopf ideal $\cI$ (see also \cite{Gasale:2004}, where the notion of Hopf  algebroid or that of presheaf of groupoids were not specified). 
 \end{definition}

\begin{remark}\label{rem:GT}
As we have seen before, it is by construction that the Hopf algebroid $(A,\cH)$ of Proposition \ref{prop:MU} is geometrically transitive in the sense of \cite{ElKaoutit:2015}, see  also Example \ref{exam:GT}. Therefore,  the associated presheaf of groupoids $\hH$ is  a transitive groupoid scheme in the sense of \cite{Deligne:1990}. It is noteworthy to mention that, taking a Malgrange groupoid  $\kK_{\Sscript{\cI}}$ associated to a differential Hopf ideal $\cI$,  even if $\hH$ defines a transitive groupoid scheme, it is not necessarily that this the case for $\kK_{\Sscript{\cI}}$. That is, the $(A\tensor{}A)$-module $\cH/\cI$ does not necessarily have to be faithfully flat. Henceforth, Malgrange Hopf algebroids fail in general to be geometrically transitive. 
\end{remark}

In the subsequent we give examples of Malgrange Hopf algebroids and describe, in some cases,  Umemura's method \cite{Umemura:2009} dealing  with the study of  certain partial differential equation.

\begin{example}[\cite{Umemura:2009}]\label{exam:YI}
Let $\cI$ be the  ideal of $\cH$ generated by the set $\{y_1-1, y_n\}_{n \geq 2}$. Since $\bdelta^k(y_1)=y_{k+1}$, for all $ k \in \mathbb{N}^*$, $\cI$ is  a differential ideal of $\cH$. Let us check that is also a  Hopf ideal. We know that 
$$
\Delta(y_1-1)=(y_1-1)\tensor{A}1+ 1 \tensor{A}(y_1-1)\, \in \cH\tensor{A}\cI + \cI \tensor{A}\cH,
$$
and by using the formula \eqref{Eq:DH} with $n \geq 2$, we  conclude that  $\Delta(\cI) \in \cH\tensor{A}\cI + \cI \tensor{A}\cH$.
On the other hand, we have $\varepsilon(y_1-1)=0$ and $\varepsilon(y_{n})=0$, whence $\varepsilon(\cI) =0$. Concerning the image of $\cI$ by the antipode, we  have that $\sS(y_1-1)=y_1^{-1}-1=y_1^{-1}( 1-y_1) \in \cI$, and by induction employing equation \eqref{Eq:SH}, we get  $\sS(\cI) \subseteq \cI$. Therefore, $(A,\cH/\cI)$ is a Malgrange Hopf $A$-algebroid which can be shown to be isomorphic to the Hopf algebroid $(A, A\tensor{}A)$. The associated groupoid is just the fibrewise groupoid of pairs (\cite[Example 2.2]{ElKaoutit:2015}) and not the action groupoid (\cite[Example 2.1]{ElKaoutit:2015}) of the additive on it self, as was claimed in \cite[page 446]{Umemura:2009}, since obviously none of the images of $x, y$ is a primitive element in $\cH/\cI$.  
\end{example}

\begin{example}(\cite[3.5]{Malgrange:2001})\label{exam:XY}
Consider $A$ with the derivation $X\partial/\partial X$. As we have seen in Remark \ref{rem:XPartial}, in this case the category $\Diff{A}$ fails to be an abelian category. Thus, the classical Tannakian theory fails too, when considering the Lie-Rinehart algebra $L=A.(X\partial/\partial X)$. Following \cite{Malgrange:2001}, this is due perhaps to the fact that  associated $D$-groupoid is not  a reduced one. Let us check this by computing its coordinate ring, and compare Umemura's  \cite{Umemura:2009} method with this case. 
Consider  then the Hopf algebroid $(A,\cH)$  of Proposition \ref{prop:MU} with differential as in equation \eqref{Eq:DiffH}, and take  the  differential ideal $\cI$ generated by $\big\{ \bdelta^{k} (xy-y_{1}) \big\}_{k\geq 0}$ (as we have mentioned before we thing of $y_{1}$ as if it was the first derivative of $y$ with respect to the variable $x$). Let us check first that $\cI$ is a Hopf ideal. So, by the formula 
$$
\bdelta^{n+1}(xy_{1}-y)\, = \, ny_{n+1}+ xy_{n+2},\quad n \geq 0
$$
we have that 
$$
\varepsilon(xy_{1}-y) \, =\, 0, \; \text{ and } \; \varepsilon\big(\bdelta^{n}(xy_{1}-y) \big)\, =\, 0, \; \text{ for any } n; \geq 1
$$
whence $\varepsilon(\cI)=0$. As for the image of $\cI$ by the comuplitplication, we have that 
\begin{multline*}
\Delta(xy_{1}-y)\,=\, (xy_{1}-y)\tensor{A}y_{1} + 1\tensor{A}(xy_{1}-y)\, \in \cI\tensor{A}\cH + \cH\tensor{A}\cI, \\
\Delta(\bdelta(xy_{1}-y))\,=\, \Delta(xy_{2})\,=\, xy_{2}\tensor{A}y_{1} + xy_{1}^{2}\tensor{A}y_{2}\,=\,  xy_{2}\tensor{A}y_{1} + (xy_{1}-y)y_{1}\tensor{A}y_{2} + y_{1}\tensor{A} xy_{2} \, \in  \cI\tensor{A}\cH + \cH\tensor{A}\cI,
\end{multline*}
and 
\begin{eqnarray*}
\Delta\Big( \bdelta^{2}(xy_{1}-y)\Big)  &=& \Delta(y_{2}+xy_{3}) \\ &=& y_{2}\tensor{A}y_{1} +  y_{1}^{2}\tensor{A}y_{2}+ xy_{1}^{3}\tensor{A} y_{3}+ 3xy_{1}y_{2}\tensor{A}y_{2} + xy_{3}\tensor{A}y_{1} \\ &=& 
(y_{2}+xy_{3})\tensor{A} y_{1} + (xy_{1}-y)\tensor{A}y_{3} + y_{1}^{2}\tensor{A}(xy_{3}+y_{2}) + 3xy_{1}y_{2}\tensor{A}y_{2} \, \in \cI\tensor{A}\cH + \cH\tensor{A}\cI
\end{eqnarray*}
Seeking for a general formula of $\Delta\Big(\bdelta^{n}(xy_{1}-y) \Big)$ and using induction shows that $\Delta(\cI) \subseteq  \cI\tensor{A}\cH + \cH\tensor{A}\cI$. Concerning the image by the antipode, we have
\begin{multline*}
\hspace{3cm} \sS( xy_{1}-y)\,=\, yy_{1}^{-1}-x\,=\, -(xy_{1}-y)y_{1}^{-1} \, \in \cI,\\ \sS\Big(  \bdelta(xy_{1}-y)\Big) \,=\,\sS\Big( xy_{2} \Big) \,=\, -y(y_{2}y_{1}^{-3})\,=\, (xy_{1}-y)y_{1}^{-3}y_{2} - xy_{1}y_{2} \, \in \cI \text{ and }    \\ 
\sS\Big(  \bdelta^{2}(xy_{1}-y)\Big) \,=\,\sS\Big(  y_{2}+xy_{3}\Big) \,=\, -(y_{2}+xy_{3})y_{1}^{-3} + (xy_{1}-y)\Big(y_{3}-y_2^2y_1^{-1}\Big)y_1^{-4} +3xy_{2}\big( y_{2}y_{1}^{-4} \big) \, \in \cI, \hspace{2cm}
\end{multline*}
and one can also show, using a kind of induction,  that $\sS\Big(\bdelta^{n}(xy_{1}-y)\Big) \in \cI$, for any $n \geq 3$. Thus $\cI$ is a differential  Hopf ideal and so $(A,\cH/\cI)$ is a Malgrange Hopf algebroid.  Since $xy_{2} \in \cI$ and none of the elements $x, y_{2}$ belong to $\cI$,  the algebra $\cH/\cI$ is not reduced. 

Now, we come back to Umemura's method as promised. Following  \cite{Umemura:2009}, if we use the universal Taylor morphism, 
$$
\biota: \Cc[x] \longrightarrow \Cc[x][[Z]], \quad a \longmapsto \sum_{n \, \geq \, 0} \frac{\bara{\partial}^{n}(a)}{n!} Z^{n}, 
$$
where the differential is $\bara{\partial}= x \frac{\partial}{\partial x}$, then we can see   that  the image of $x$ satisfies  the equality $\biota(x) = x \, {\rm Exp}(Z)$, and so $x\, \frac{\partial \biota(x)}{\partial x} = \biota(x)$.  Moreover, $\biota$ can be extended to an algebra map $\biota: \cH \to A[[Z]]$ by sending $y \mapsto \biota(x)$, $x \mapsto x$ (i.e.,~the power series $(x, 0, \cdots, 0)$) and $\biota(y_n) =\frac{\partial^n \biota(x)}{\partial x^n}$, for $n \geq 1$. Thus, $\biota(y_1)={\rm Exp}(Z)$ and $\biota(y_n)=0$, for $n \geq 2$. Henceforth, we have a morphism $\kappaup: \cH/\cI \rightarrow A[[Z]]$\footnote{ Note that the $I$-adic completion  of this algebra map leads to a morphism of complete Hopf algebroids over $A$, where $A$ is considered as a discrete topological ring, see \cite{Kaoutit/Saracco:2017}.} of $(A\tensor{}A)$-algebras, which is, of course, not injective nor a differential morphism. Thus there is no hope in obtaining an analogue result to the one claimed in \cite[Proposition 6.1]{Umemura:2009} for this case, and so Umemura's method only works for some specific cases. This is due perhaps to the fact that the morphism $\kappaup$ does not take into account 
 the whole system of equations
\begin{equation}\label{Eq:NY}
(n-1) y^{(n)}  -x y^{(n+1)}\,\, =\,\,0, \; \text{ for } \; n \geq 0, 
\end{equation}
in the sense that the images of $y_n$ by $\kappaup$ are zero up to degree $2$. 

There are observations which should be highlighted here with respect to what we have seen so far. The previous power series algebra can in fact be seen as the convolution algebra of the co-commutative Hopf algebroid $U$, that is, we have $U^*\cong A[[Z]]$ where  $U$  is the universal enveloping algebroid of the Lie-Rinehart algebra $A.\, \bara{\partial}$  and where the universal Taylor map $\biota$ is nothing but the target map for the $(A\tensor{}A)$-algebra $U^*$ (this is in fact a topological commutative Hopf algebroid,  see \cite{Kaoutit/Saracco:2017} for more details).  In this way, it is not clear, at least to us, if  the commutative Hopf algebroid $\cH/\cI$ can be identified with  the finite dual of a certain co-commutative Hopf algebroid. According to what was explained in Remark \ref{rem:MJet}, an answer to this question can be perhaps performed by using not only the finite dual of the universal enveloping algebroid of a given Lie-Rinehart algebra, but also the finite duals of the universal enveloping algebroids of the attached jets Lie-Rinehart algebras of any order. In this direction, it is perhaps possible that one could also establish a certain  Picard-Vessiot theory for the system of equation \eqref{Eq:NY}. On the other hand, if we allow solutions of the system \eqref{Eq:NY} to be with coefficients in the localized algebra $\Cc[X^{\pm 1}]$, then the system can be linearized and one can try to employ directly the theory we have developed hereby in order to solve this new linear system. 
\end{example}

In the following remark we illustrate the difference between the approach to the Galois groupoid of Definition \ref{def:GDG} and the one introduced by B. Malgrange in \cite{Malgrange:2001} (see also \cite{Umemura:2009}).

\begin{remark}\label{rem:MU} 
Let place ourselves in the context of  \cite[5.3]{Malgrange:2001} by taking the analytic smooth connected curve $\cX$ to be the affine complex line $\affLine{\Cc}$ and set as before $A=\Cc[X]$ its coordinate ring. This is indeed  the case we are treating in this section. Now given a Malgrange groupoid $\kK_{\Sscript{\cI}}$, as in Definition \ref{def:MU},  defined by a differential Hopf ideal $\cI$,   then, following \cite{Malgrange:2000, Malgrange:2001}, the attached \emph{Galois groupoid} has $\cX$ as the space of objects and an arrow   is a germ of local diffeomorphism $\bo{g}: (\cX, p) \to (\cX, g(p))$ from an open neighbourhood of $p$ in $\cX$ to an open neighbourhood of $g(p)$ in $\cX$, such that $g$ is a solution of the differential equations generating the ideal $\cI$. The structure groupoid is the one induced from the groupoid $Aut(\cX)$ of germs of local diffeomorphisms (\'etale Lie groupoid in fact, see \cite[Example 2.5(4)]{MoeMrc:LGSAC} for details).  
Therefore, for a differential module $(M, \partial)$ of rank one over $(A,\partial)$, the Galois groupoid $\hH_{(M)}(\Cc)$  of Definition \ref{def:GDG}, have the same set of objects as the Galois groupoid attached to a Malgrange groupoid $\kK_{\Sscript{\cI}}$, for a given differential Hopf ideal $\cI$ (notice that there are  case when this two Galois groupoids coincide, for instance, when $\cI$ is the differential  Hopf ideal $\langle y_n \rangle_{n \geq 2}$). 

For a differential module with rank $m \geq 2$ over $A$, the situation  is totally different. Precisely, let us consider such a differential module and consider as in Remark \ref{rem:MJet} the associated locally free vector bundle  $(E,\pi)$ of constant rank $m$ over $\cX$.  Following the procedure of  \cite[5.3]{Malgrange:2001}, in order to define the Galois groupoid $Gal(F)$ (notation of \cite[5.2]{Malgrange:2001})  of the foliation $F$ attached to $(E,\pi)$ (and given by the connection $\nabla$ corresponding to the differential $\partial$ of $M$), one has to consider a Malgrange Hopf algebroid extracted form  a certain  differential Hopf ideal $\cI_{\Sscript{F}}$ of the Hopf algebroid $\cH_{\Sscript{\Cc^{m+1}}}$ of equation \eqref{Eq:HCm}. The existence and the construction of this differential Hopf ideal, or equivalently that of the Malgrange groupoid $\cK_{\Sscript{\cI_{F}}}$ is guaranteed by \cite[Th\'eor\`eme 4.5.1]{Malgrange:2001}. Indeed $\cI_{\Sscript{F}}$ is the largest differential Hopf ideal for which the germs of sections of $F$ are solutions. In this direction, the  Galois groupoid $\hH_{(M)}(\Cc)$ still having $\affLine{\Cc}$ as set of objects, while the Galois groupoid $Gal(F)$ has the space $\mathbb{A}_{\Sscript{\Cc}}^{m+1}$ (the analytic variety underlying the bundle $(E,\pi)$, see \cite[page 493]{Malgrange:2001}) as a set of objects. This  at first glance shows that  $\hH_{(M)}(\Cc)$ is not isomorphic to $Gal(F)$, it is not clear, however,  at least to us, if they are weakly equivalent in the sense of Remark \ref{rema:final}. 
\end{remark}

As we have seen in Remark \ref{rem:MU}, perfunctorily our differential Galois groupoid attached to a differential $A$-module is different from the one introduced by Malgrange in \cite{Malgrange:2001}.  It seems that the two approaches are also far from begin similar. Specifically, 
let $(M,\partial)$ be a differential $A$-module of rank $2$. If we want to study the system of linear differential equations attached to $(M,\partial)$ by using  Malgrange's Hopf algebroids, then the `universal' Hopf algebroid described in Proposition \ref{prop:MU} is useless, since no isotropy group of the attached presheaf of groupoids  will provide us with an algebraic closed subgroup of $GL_2(\Cc)$. Besides, even for differential modules of rank one, like for instance  the one considered in Example \ref{Exm:dimOne}, it is not clear how to construct  a differential Hopf ideal of this Hopf algebroid  form this differential module (apriori the defining ideal should be the differential ideal generated by $y_1-xy$ which is clearly not a Hopf ideal, nor the one generated by $z_1-xz$, where $z=y-x$ and $z_1=z'$). 

On the other hand, the `universal' Hopf algebroid of Proposition \ref{prop:MU} (or the one defined by any reduced algebraic variety as the coordinate ring of invertible jets \cite{Malgrange:2001, Umemura:2009}, like the one in Example \ref{exam:Rank2}), leads to partial differential equations defined by given differential Hopf ideals. This means that this approach and its methods of studying these equations goes somehow in the contrary direction of what is traditionally done in the differential Galois theory of differential fields (unless perhaps one is interested in founding  partial differential equations attached to some dynamical system and providing its groupoid of ``symmetries'' \cite{Umemura:2009, Morikawa/Umemura:2009}).  More precisely, usually  one considers, in the first step, a system of (algebraic) differential equations as an initial data and look for a representation over a  certain Lie algebroid (i.e., Lie-Rinehart algebra) which defines such system of equations. In the second step, one seek for a Picard-Vessiot extension of the differential base algebra as a formal solution space of the starting system of equations. Furthermore,  if we want to use Malgrange's Hopf algebroids, then it is not clear from \cite{Malgrange:2001} or \cite{Umemura:1996}, how to  construct  a Picard-Vessiot extension, in the sense of \cite{Andre:2001},  associated to the system of equations defining a differential Hopf ideal. 

In summary, our (algebraic) approach to differential Galois theory over differential rings  seems to be different form the (geometric) one adopted in  \cite{Malgrange:2001, Umemura:1996},  although, there are some similar aspects between the two approaches.

In the following example, we give a map of bialgebroids involving the underlying bialgebroid of the  Hopf algebroid of Proposition \ref{prop:MU}. Besides, we try to mimic the method employed in \cite[Example 5.1]{Umemura:2009}, for the case affine complex plan (by taking a certain sub Lie-Rinehart algebra of the Lie algebra of derivations of the coordinate ring), and arrive to the conclusion that this method doesn't  behave well in this case, which in some sense bears out the above explanations. 

\begin{example}\label{exam:Rank2}
In this example we compute the structure maps of the  Hopf algebroid describe in Proposition \ref{prop:MU}, but for two variables instead of one. That is, a Hopf algebroid over   $\Cc[x_0,x_1]$ the coordinate ring of the affine plan.  The quotients by differential ideals of this `universal' Hopf algebroid lead to Malgrange Hopf algebroids which allows the study of systems of certain algebraic partial differential equations with two variables.  Let $\{x_{0},x_{1},y_{i\, \alpha} \}_{i=0,1; \, \alpha \in \mathbb{N}^{2}}$ be a set of independent variables over $\Cc$. To simplify the notations we will adopt the following one:
\begin{multline*}
y_{i\, 0} := y_{i}, \quad \text{ for }  i=0,1; \\ 
y_{i\, \epsilon_{j}}:= y_{ij}, \quad \text{ for }  i, j =0,1; \text{ and where }\; \epsilon_{0}=(1,0),\,  \epsilon_{1}=(0,1);\\
y_{i\, \alpha}:= y_{i\alpha_{1}\alpha_{2}},   \quad \text{ where }  \alpha=(\alpha_{1}, \alpha_{2}) \in \mathbb{N}^{2} \setminus \{0, \epsilon_{0}, \epsilon_{1}\}.
\end{multline*}
Then we have   a family of variables:
$$
\begin{cases}
x_{0}, x_{1}, y_{0}, y_{1}; \\
y_{00}, y_{10}, y_{01}, y_{11}; \\
y_{i11}, y_{i02}, y_{i20}, y_{i12}, y_{i21}, y_{i03}, y_{i30}, \cdots\cdots, i=0,1.
\end{cases}
$$
As before, we think of $y_{i}$ as functions of two variables $x_{0}, x_{1}$, and $y_{ij}$ their first partial derivatives $\frac{\partial y_{i}}{\partial x_{j}}$, and $y_{i\alpha_{1}\alpha_{2}}$ are their higher derivatives $\frac{\partial^{|\alpha|} y_{i}}{\partial x_{0}^{\alpha_{1}}\, \partial x_{1}^{\alpha_{2}}}$. In this way the determinant $det(y_{ij})$ can be viewed as the Jacobian of the function $(y_{0}, y_{1})$ in two variables $x_0, x_1$. We consider the following commutative $\Cc$-algebra
$$
\cK:= \cH_{\Sscript{\Cc^2}}=\Cc\Big[x_{0}, x_{1}, y_{0}, y_{1}, y_{ij}, det(y_{ij})^{-1}, y_{i\alpha_{1}\alpha_{2}}\Big]_{i,j=0,1,\, \alpha \, \in \mathbb{N}^{2},\, |\alpha| \geq 2 }
$$
This is a $(B\tensor{}B)$-algebra where $B=\Cc[x_{0}, x_{1}]$ is the polynomial algebra with two variables. Clearly it is a free $(B\tensor{}B)$-module. The source and the target are $\Sf{s}(x_{i})=x_{i}$ and $\Sf{t}(x_{i})=y_{i}$, for every $i=0,1$. The algebra $\cK$ is a partial differential algebra extension of $(B,\partial/\partial x_{0}, \partial/\partial x_{1})$ via the source map. The partial derivations of $\cK$ are given as in \cite[page 470]{Malgrange:2001} by 
$$
\bpartial_{i}P= \frac{\partial P}{\partial x_{i}}  + \sum \frac{\partial P}{ \partial y_{j\alpha}} y_{j(\alpha + \epsilon_{i})}, \quad   i=0,1,
$$
for any polynomial function $P \in \cK$. 

Next, we illustrate the Hopf algebroid structure of the pair $(B, \cK)$. The counit is the map which sends:
\begin{multline*}
\varepsilon(x_i) \,=\, \varepsilon(y_i) \,=\, x_i,\quad \text{ for all } i=0,1; \\ 
\varepsilon(y_{ij})\,=\, \delta_{ij} \; (\text{Kronecker symbol }),\quad \text{ for all } i, j=0,1; \\
\varepsilon(y_{i\alpha})\,=\, 0, \quad \text{ for all } i=0,1 \text{ and } \alpha \in \mathbb{N} \setminus \{ 0,\epsilon_0, \epsilon_1\}. \hspace{2cm}
\end{multline*}
The comultiplication and the antipode are given as follows: 
$$
\Delta(y_{ij})\,=\, \sum_{k} y_{ik}\tensor{A}y_{kj}, \quad \text{ for } i,j=0,1.
$$
While for any $i=0,1$, we have 
\begin{eqnarray*}
\Delta(y_{i11}) &=& y_{i20}\tensor{B}y_{00}y_{01} + y_{i11}\tensor{B}y_{00}y_{11} + y_{i11}\tensor{B}y_{10}y_{01}+ y_{i02}\tensor{B}y_{10}y_{11} + y_{i0}\tensor{B}y_{011} + y_{i1}\tensor{B}y_{111}, \\
\Delta(y_{i02}) &=&  y_{i20}\tensor{B}y_{01}^2 + y_{i11}\tensor{B} y_{11}y_{01} + y_{i0} \tensor{B}y_{002} + y_{i11} \tensor{B}y_{01}y_{11} + y_{i02}\tensor{B} y_{11}^2 + y_{i1}\tensor{B} y_{102}, \\ 
\Delta(y_{i20}) &=& y_{i20}\tensor{B}y_{00}^2 + y_{i11} \tensor{B}y_{10}y_{00} + y_{i0}\tensor{B}y_{020} + y_{i11}\tensor{B}y_{00}y_{10} + y_{i02}\tensor{B}y_{10}^2 + y_{i1}\tensor{B}y_{120},
\end{eqnarray*}
the antipode is given by:
\begin{multline*}
\sS(x_i) \,=\, y_i, \quad \text{ for all } i,j=0,1; \\ \sS(y_{ij}) \,=\, det(y_{ij})^{-1} (-1)^{i+j} y_{ji}, \quad  \text{ for all } i, j=0,1, \hspace{4cm}
\end{multline*}
and for $\alpha \in \mathbb{N}^2$ such that $|\alpha|=2$ and for any $i=0,1$, we have 
\begin{multline*}
\sS(y_{i11}) \, det(y_{ij})^2 \,=\, \big( y_{00}y_{11}+y_{10}y_{01} \big) 	\Big(\sS(y_{i0})y_{011} + \sS(y_{i1})y_{111}\Big) - y_{10}y_{00}\Big(\sS(y_{i1})y_{102} + \sS(y_{i0})y_{002}\Big) \\ -y_{11}y_{01}\Big(\sS(y_{i0})y_{020} + \sS(y_{i1})y_{120}\Big), \\
\sS(y_{i02})\,  det(y_{ij})^2 \,=\, 2 y_{00}y_{01}\Big(\sS(y_{i0})y_{011} + \sS(y_{i1})y_{111}\Big) - y_{00}^2\Big(\sS(y_{i1})y_{102} + \sS(y_{i0})y_{002}\Big) \\ -y_{01}^2\Big(\sS(y_{i0})y_{020} + \sS(y_{i1})y_{120}\Big), \\
\sS(y_{i20}) \, det(y_{ij})^2 \,=\, 2 y_{11}y_{10}\Big(\sS(y_{i0})y_{011} + \sS(y_{i1})y_{111}\Big) - y_{10}^2\Big(\sS(y_{i1})y_{102} + \sS(y_{i0})y_{002}\Big) \\ -y_{11}^2\Big(\sS(y_{i0})y_{020} + \sS(y_{i1})y_{120}\Big).
\end{multline*}
For higher degrees, that is, for $\alpha \in  \mathbb{N}$  with $|\alpha| \geq 3$, one can use the higher partial derivative Chain rules for both comultiplication and the antipode.  We have that $(B,\cK)$ is a geometrically transitive Hopf algebroid.

Assume we are given on $B$ the following derivation $\delta(x_0)=1$ and $ \delta(x_1)=x_0x_1$ (in other words we are considering the sub Lie-Rinehart algebra of $\Der{\Cc}{B}$  generated by $\frac{\partial}{\partial x_0}$ and $x_0x_1\frac{\partial}{\partial x_1}$). As we will observe below, this Lie-Rinehart algebra can be used to seek the solutions of the linear differential equation $y'=x_0y$ over the polynomial algebra $\Cc[x_0]$. 

From the definition of $\delta$,  we have that  $\delta^n(x_1)\,=\, p_n(x_0) x_1$, for any $n \in \mathbb{N}$, where the sequence polynomials $\{p_n(x_0)\}_{n \geq 0}$ in $x_0$, satisfies the following non homogeneous recurrence: 
\begin{equation}\label{Eq:Pis}
\begin{cases}
p_{n+1}(x_0)\,=\, x_0 p_n(x_0) + p_n(x_0)' , \quad n \geq 1,\\
p_0(x_0)\,=\, 1
\end{cases}
\end{equation}
Following the method of \cite{Umemura:2009} and by using the universal Taylor map $\bo{i}: B \to B[[Z]]$ sending $b \mapsto \sum_{n \geq 0} \frac{\delta^{n}(b)}{n!} Z^{n}$, we can conclude the following partial differential equations:
\begin{equation}\label{Eq:PDE}
\frac{\partial \bo{i}(x_0)}{\partial x_0 }\,=\, 1,\quad \frac{\partial \bo{i}(x_0)}{\partial x_1 } \,=\, 0, \quad x_1 \frac{\partial \bo{i}(x_1)}{\partial x_1 }\,=\, \bo{i}(x_1).
\end{equation}
On the other hand, extending  $\delta$ to the $B[[Z]]$ by putting $\delta(Z)=0$ and considering the derivation $\partial_{\bullet}\,=\, \frac{1}{2}\big( \delta + \frac{\partial}{\partial Z}\big)$ on $B[[Z]]$, we have that $\bo{i}: (B,\delta) \to (B[[Z]], \partial_{\bullet})$ is a morphism of  differential algebras, so that  the equation $\partial_{\bullet}\, \bo{i}(x_1) \, =\, x_0 \bo{i}(x_1)$ holds true.  Therefore, $B[[Z]]$ contains a solution of the above linear differential equation (up to a certain scalar).  

Now, if we want to use equations \eqref{Eq:PDE} with the purpose of guessing a differential Hopf ideal and  construct a Malgrange Hopf algebroid according to Definition \ref{def:MU} (this is the method adopted, for instance  in the example expounded in \cite[\S 5]{Umemura:2009}), then the adequate set of generators seems to be  $\{y_{00}-1, y_{01}, x_1y_{11}-y_1\}$ and their higher partial derivatives: $\bpartial_i^k(y_{00}-1), \bpartial_i^l(y_{01}), \bpartial_i^m(x_1y_{11}-y_1)$, $k,l,m \geq 1$, $i=0,1$. Unfortunately, this is not the case, since the differential ideal $\cJ$ generated by this set, only leads to a structure of bialgebroid on the quotient $B$-bimodule $\cK/\cJ$. This is due to the fact that  $\cJ$ is not stable under the antipode of $\cK$, as one can check from the equality $\sS(y_{01}) =det(y_{ij})^{-1}y_{10}$. Nevertheless, we still have a morphism $\cK/\cJ \to B[[Z]]$ of $(B\tensor{}B)$-algebras, and one can also construct a morphism $\cH \to \cK/\cJ$ of bialgebroids, sending $y_1 \mapsto y_{11} + \cJ$ and $y_n \mapsto y_{1(n-1)1} + \cJ$ for $n \geq 2$, where $\cH$ is the Hopf algebroid of Example \ref{exam:YI}. It is noteworthy to mention that it is also possible that $\cK/\cJ$ becomes a Hopf algebroid after localizing on certain denominator set. 

In case of dimension two, that is, for a differential module  $(M, \partial)$ over $A=\Cc[x_0]$ of rank $2$ with matrix $\cM:=\mat{M}=(a_{ij})_{1 \leq i,j \leq 2} \in M_2(A)$, the previous method can be described as follows. Assume this time that $\delta$ is a derivation of the polynomial  complex algebra $C=\Cc[x_0,x_1,x_2]$, sending, $x_0 \mapsto 1$, $x_i \mapsto \sum_{1 \leq j \leq 2} a_{ij}x_j$, for $i=1,2$.    
In this case the Hopf algebroid to be considered is the pair  $(C,\cW)$, where  $\cW$ is the commutative polynomial algebra 
$$
\cW:=\cH_{\Sscript{\Cc^{2}}}=\Cc[x_0,x_1, x_2,y_0,y_1,y_2, y_{ij}, y_{i\alpha}, det(y_{ij})^{-1}]_{i,j=0,1,2,\; \alpha \,\in\, \mathbb{N}^3\setminus\{ \epsilon_0,\epsilon_1,\epsilon_2 \}}
$$
whose structure maps are given as above, by thinking of the $y_i$'s as if they were certain functions on the variables $x_0, x_1, x_2$. 
The equivalent recurrence of the recurrence given in equation \eqref{Eq:Pis}, is the following recurrence on $2 \times 2$-matrices with entries in $\Cc[x_0]$
\begin{equation}\label{Eq:Ais}
\begin{cases}
\cM_{n+1}\,=\, \cM_n^{\prime} + \cM_n\, \cM , \quad n \geq 1,\\
\cM_0\,=\, I_2,
\end{cases}
\end{equation}
where the notation $\cM_n^{\prime}$ stand for the matrix whose entries are the derivations of the entries of $\cM_n$.
Employing the universal Taylor map $\bo{i}: C \to C[[Z]]$, this means that,  we have
$$
\frac{\partial \bo{i}(x_i)}{\partial x_j }\,=\, \sum_{n \geq 0}\frac{a^{(n)}_{ij}}{n!}Z^n, \quad \text{ for any } \, i, j \in \{1,2\},
$$ 
where, for any $n \geq 0$, the coefficients $a^{(n)}_{ij}$ are the entries of the matrix $\cM_{n}$. Using the recurrence system of \eqref{Eq:Ais}, we end up with  the subsequent system of partial differential equations 
\begin{equation}\label{Eq:PDE2}
 \begin{array}{ccc}
 \frac{\partial \bo{i}(x_0)}{\partial x_0 }\,=\, 1,  &  \frac{\partial \bo{i}(x_0)}{\partial x_1 }\,=\,0,  &  \frac{\partial \bo{i}(x_0)}{\partial x_2 }\,=\,0, \\ \frac{\partial \bo{i}(x_1)}{\partial x_0 }\,=\, 0, &    \frac{\partial \bo{i}(x_2)}{\partial x_0 }\,=\, 0, & \\ 
  \bo{i}(x_i)\,=\,   \sum_{k=1}^2 x_k \frac{\partial \bo{i}(x_i)}{\partial x_k }  & \text{ for } i =1,2. & 
 \end{array}  
\end{equation}
The corresponding differential ideal $\cJ$ is the one generated by the set:
$$
\Big\{ y_{00}-1, y_{01}, y_{02}, y_{10}, y_{20}, y_{i}- x_1 y_{i1}-x_2y_{i2}\Big\}_{i=1,2}.
$$
The quotient leads here also to a bialgebroid over the algebra $C$ and it is not clear, up to now,  how to relate this bialgebroid with the Hopf algebroid $\Um$ of subsection \ref{ssec:PV} attached to $(M,\partial)$.
\end{example}

\bigskip
\noindent \textbf{Acknowledgement.} The authors would like to thank the referee for bring our attention to the references \cite{Andre:2001} and  \cite{Malgrange:2001} which leads us to improve considerably the earlier exposition of this paper.

\end{document}